%% file: main.tex
\documentclass[oneside]{aims} 

\usepackage{paralist}
\usepackage{amsmath}
\usepackage[misc]{ifsym}
\usepackage[colorlinks=true]{hyperref}

\usepackage{amssymb,amsfonts}%
\usepackage{amsthm}%
\usepackage{mathrsfs}%

\usepackage{mathtools}

\usepackage[font=footnotesize,labelformat=simple]{subcaption}
\usepackage[numbers]{natbib}

\usepackage{csquotes}
\usepackage{braket}

\usepackage{footnote}

\usepackage{tikz,pgfplots}
\pgfplotsset{compat=1.18}
\usepackage{tikz-qtree,tikz-qtree-compat}
\usepackage{tikz-cd}
\usetikzlibrary{arrows,decorations.markings,babel}
\usetikzlibrary{shapes.geometric,fit}

\tikzset{
commutative diagrams/.cd,
arrow style=tikz,
diagrams={>=latex}}

\tikzset{double line with arrow/.style args={#1,#2}{decorate,decoration={markings,%
mark=at position 0 with {\coordinate (ta-base-1) at (0,1pt);
\coordinate (ta-base-2) at (0,-1pt);},
mark=at position 1 with {\draw[#1] (ta-base-1) -- (0,1pt);
\draw[#2] (ta-base-2) -- (0,-1pt);
}}}}
\tikzset{
   dashellipse/.style={ellipse,draw,dashed,inner sep=0pt,blue,fit={#1}}
}

\makeatletter
\newcommand{\succprec}{\mathrel{\mathpalette\succ@prec{\succ\prec}}}
\newcommand{\nsuccprec}{\mathrel{\,\not\!\succprec}}

\newcommand{\succ@prec}[2]{\succ@@prec#1#2}
\newcommand{\succ@@prec}[3]{%
  \vcenter{\m@th\offinterlineskip
    \sbox\z@{$#1#3$}%
    \hbox{$#1#2$}\kern-0.4\ht\z@\box\z@
  }%
}
\makeatother

\allowdisplaybreaks

\def\currentvolume{X}
 \def\currentissue{X}
  \def\currentyear{200X}
   \def\currentmonth{XX}
    \def\ppages{X-XX}
     \def\DOI{xx.xxxx/xx.xxxxxxx}

\newtheorem{theorem}{Theorem}[section]
\newtheorem{corollary}[theorem]{Corollary}

\newtheorem{lemma}[theorem]{Lemma}
\newtheorem{proposition}[theorem]{Proposition}

\theoremstyle{definition}
\newtheorem{definition}[theorem]{Definition}
\newtheorem{remark}[theorem]{Remark}

\newtheorem{example}[theorem]{Example}
\newtheorem{observation}[theorem]{Observation}

\makeatletter
\newcommand{\centersymbol}[2]{%
  \mathrel{\vphantom{#1#2}\mathpalette\center@symbol{{#1}{#2}}}%
}
\newcommand{\center@symbol}[2]{%
  \center@@symbol{#1}#2%
}
\newcommand{\center@@symbol}[3]{%
  \ooalign{\hss$#1\m@th#2$\hss\cr\hss$#1\m@th#3$\hss\cr}%
}
\makeatother

\newcommand{\cut}{\centersymbol{/}{\mathrel{-}\joinrel\mathrel{-}}}
\newcommand{\glue}{\centersymbol{\bullet}{\mathrel{-}\joinrel\mathrel{-}}}

\newcommand{\djunion}{\sqcup}
\newcommand{\bigdjunion}{\bigsqcup}

\title[Discrete Level Set Persistence for Finite Discrete Functions]
{Discrete Level Set Persistence for Finite Discrete Functions} 

\author[Robin Belton and Georg Essl]{}

\subjclass{Primary: 55N31; Secondary: 37M10, 05C90.}
\keywords{time series, topological data analysis, finite ordered sets, duality, surgery}

\usepackage[final]{changes}
\newcommand{\marginnote}[1]{}
\newcommand{\addedbegin}{}
\newcommand{\addedend}{}
\newcommand{\robin}[1]{} 
\newcommand{\georg}[1]{}

\begin{document}

\pgfplotsset{
standard/.style={
    axis x line=middle,
    axis y line=middle,
    enlarge x limits=0.15,
    enlarge y limits=0.15,
    every axis x label/.style={at={(current axis.right of origin)},anchor=north west},
    every axis y label/.style={at={(current axis.above origin)},anchor=north east},
    every axis plot post/.style={mark options={fill=white}}
    }
}

\maketitle

\centerline{\scshape
Robin Belton$^{{\href{mailto:rbelton@vassar.edu}{\textrm{\Letter}}}1}$
and Georg Essl$^{{\href{mailto:essl@uwm.edu}{\textrm{\Letter}}}2}$}

\medskip

{\footnotesize
 \centerline{$^1$
Department of Mathematics and Statistics, Vassar College, Poughkeepsie, NY 12604, USA}
} 

\medskip

{\footnotesize
 \centerline{$^2$Department of Mathematical Sciences, University of Wisconsin-Milwaukee, Milwaukee, WI 53211, USA}
}

\bigskip



\begin{abstract}
We study sublevel set and superlevel set persistent homology on discrete functions through the perspective of finite ordered sets of both linearly and cyclically ordered domains. Finite ordered sets also serve as the codomain of our functions making all arguments finite and discrete. We prove duality of filtrations of sublevel sets and superlevel sets that undergirds a range of duality results of sublevel set persistent homology without the need to invoke complications of continuous functions or classical Morse theory. We show that Morse-like behavior can be achieved for flat extrema without assuming genericity. Additionally, we show that with inversion of order, one can compute sublevel set persistence from superlevel set persistence, and vice versa via a duality result that does not require the boundary to be treated as a special case. Furthermore, we discuss aspects of barcode construction rules, surgery of circular and linearly ordered sets, as well as surgery on auxiliary structures such as box snakes, which segment the ordered set by extrema and monotones.
\end{abstract}

\input{intro.tex} 
\input{background.tex}

\input{setting.tex}

\input{morse.tex}

\input{barcodes.tex}

\input{dual.tex}
\input{snakeboxes.tex}
\input{surgery.tex}
\input{applications.tex}
\input{conclusion.tex}
\appendix
\input{appendix.tex}


\bibliographystyle{AIMS}
\bibliography{sn-bibliography}

\end{document}

%% file: intro.tex
\section{Introduction}\label{sec1}

Finite discrete functions are the fundamental data representation in numerous application areas. An important class of this type are finite discrete time series. Any sequential sensor data read out at discrete moments of time and stored digitally will lead to such a series. A few other examples include digital audio signals that take microphone readings of air undulations and is then converted to digital representations via an analog-digital-converter \cite{steiglitz1997digital,sanderson2017computational}, EEG signals that are measurements of electric nerve activity  \cite{chaddad2023electroencephalography}, and single and aggregate stock market data \cite{shah2021filtering}. Discrete functions also occur in a wide range of models of time-dynamic natural phenomena such as neural activation functions \cite{urban2018neural}. Unless computation is symbolic, computer representation of any one-dimensional function tends to be represented discretely both in terms of the domain and the codomain. 

A vast body of literature exists that create methodology to analyze, synthesize, and manipulate such data (for digital signal processing see \cite{steiglitz1997digital}). Topological methods have become a topic of interest over the last two decades \cite{edelsbrunner2010computational} and are part of the rapidly growing field of topological data analysis (TDA) which focuses on the ``shape" of data \cite{carlsson2009topology}. Homology focuses on the connected components and holes of a topological space. \emph{Persistent Homology} is a fundamental tool in TDA where the idea is to analyze how the homology of a topological space changes as we vary a parameter. A common setting is to vary the height parameter $h$ of \emph{sublevel sets} of a function, $f^{-1}(-\infty, h]$. Encoded in the \emph{persistence barcode} are the heights when a homological feature first appeared (``born") and the heights when homological features merged together (``died").

The purpose of this paper is to create a rigorous and user-friendly framework for discrete level set persistent homology of finite discrete functional data. Level sets of functions \cite{bendich2013homology} provide a way to describe topological invariants of functions via the homological changes in level sets.

\subsection{Techniques for Analyzing Discrete Functions}
An important goal of the current work is to avoid making any restriction on data. Data should be able to be processed without the need to impose restrictions or perturbations. We want to apply techniques and theories directly on the dataset.\marginnote{Review Y: Major 4}

\marginnote{Review Y: Minor 1}\added{There are numerous advantages of having a technique not contain restructions, including the ease of use, the ability to use it in settings where the perturbations are undesirable, and a theoretical understanding about the case the restrictions try to avoid.} 

In prior work, one routinely finds restrictions on the data, either for theoretical reasons or to
simplify structures or arguments. These have numerous origins, such as an attempt to characterize extrema, the ability to restrict to binary merge trees, or an attempt to guarantee some desirable algebraic structure. For example, genericity, the requirement that only one extremum is present at any one level, allows one to guarantee that merge trees are binary \cite{baryshnikov2024time,CURRY2024102031}. The notion of genericity also helps to simplify the auxiliary data structures \cite{biswas2023geometric}. Lack of genericity is dealt with by perturbation or tie-breaking arguments \cite{glisse2023fast}. Extrema are often assumed to be isolated, leaning on Morse-theoretic arguments \cite{cohen2005stability,curry2018fiber}, or alternatively assumed to be characterized by homological change \cite{cohen2005stability,bubenik2014categorification,govc2016definition}.
In this paper, we show that in the finite discrete setting, it is possible to compute persistent homology without these restrictions. Furthermore, we have a rigorous notion of when homological change occurs. Our work is comparable to saddle point treatment regarding merge trees for real analytic functions \cite{Beketayev2014DistanceMergeTrees} in the context of discrete data without function theoretic assumptions. \marginnote{Review Y: Minor 2}\added{In the language of Morse-theory, we show how to precisely handle the total or partial absence of a gradient in the discrete setting by showing how to identify extrema types without it.}

Furthermore, implementation on a digital computer leads to discrete numerical representations. Continuous domains such as the real line are replaced by discrete stored entities. Additionally, continuous codomains are often thought of as potentially bounded intervals of the real line, and are approximated by discrete entities such as floating-point number representations. In this paper, we discuss discrete models of all domains and codomains in consideration. We will see that this introduces a certain level of convenience because \marginnote{Review Q: Minor 1}\replaced{pathological situations}{pathological cases} of the continuous case that need to be handled with care \cite{govc2016definition} are not present in this setting.

Purely discrete models are not new in data processing. One of the most successful and useful discrete data analysis algorithms is the Fast Fourier Transform (and more generally Discrete Fourier Transform) \cite{cooley1965algorithm}, which are algorithms that take discrete finite data as input and produce the same as output. Hence, these algorithms are understood and articulated as discrete-to-discrete relationships. \marginnote{Review Y: Major 3}\added{Discretizations are classical in digital signal processing and are known as sampling - referring to the signal domain - and quantization, when referring to the signal codomain \cite{steiglitz1997digital}. Given that we merely require an order of these discrete sets, we arrive at a more general setting that captures deformations of both the domain and codomain. All our results are valid up to order preservation. More specifically, some classical results such as the symmetry theorem from \cite{cohen2009extending,edelsbrunner2010computational} relate sublevel and superlevel set persistent homology by considering sublevel sets of both $f$ and $-f$. These results can be expanded more generally to $f$ and any \emph{order inversion} of $f$ in which the order of the heights of the samples, or image of $f$ are reversed. Taking the negative of a function is one type of order inversion.}

In this paper, we \deleted{mirror this setting and} seek to create a level set persistence theory that operates on finite discrete spaces assuming nothing more.

\subsection{Related Work on Level Sets of Discrete Functions}
The use of level sets on discrete functions is not new. It is already used for a range of signal analysis and processing tasks. A sublevel set persistence computation for time series was given by \cite{Myers2022additivenoise} and available as part of the Teaspoon library \cite{myers2020teaspoon}\footnote{\url{https://teaspoontda.github.io/teaspoon/sublevel_set_persistence.html}}. Due to the use of the {\tt find\_peaks} function of the scipy Python library, this algorithm handles flat extrema. Similarly, Glisse discusses a fast algorithm for computing level set persistence \cite{glisse2023fast} on the model assumption that extrema are at least locally generic, though suggesting that this limitation can be removed by local tie-breaking rules, which implies the ability to handle flat extrema\footnote{Available as part of GUDHI \url{https://github.com/GUDHI/TDA-tutorial?tab=readme-ov-file}}. \marginnote{Review Q: Major 1}\replaced{Our work gives a rigorous explanation that flats can indeed be handled directly.}{The validity of these claims of handling flat extrema does not appear to have been fully articulated. The current paper fills this gap.} \marginnote{Review Y: Major 2}  Baryshnikov \cite{baryshnikov2024time} discussses an algorithm for finite time series data but requires genericity and isolated extrema. Our work removes these restrictions.

\marginnote{Review X}\added{Our setting can also be interpreted as a lower dimensional version of work on digital gray-scale images in two and three dimensions as well as other flow-based simulations in these dimensions. In these settings, Morse theory plays a prominent role in informing techniques \cite{gyulassy2008practical,gunther2011efficient,robins2011theory,magillo2013morphologically,rocca2024disambiguating} and the issues of flats and saddle points are addressed, usually by perturbative or tie-breaking rules. A classical approach to resolve degeneracies in this manner is known as simulation of simplicity \cite{edelsbrunner1990simulation}, which performs an application-agnostic general preprocessing layer to ensure the dataset as presented to an algorithm is non-degenerate. However, data points that are at the same level are perturbed in this sense and can lead to artifacts that have been described as spurious critical points \cite{gyulassy2008practical,finken2024localized,rocca2024disambiguating}. Perturbative methods for dealing with degeneracies in computational geometry have also been critiqued on theoretical grounds \cite{seidel1998nature}. In areas where flow assumptions are natural, or gradients are sought, Morse theory forms a natural fit, and there is a body of work that addresses the issues of flats and saddles in this context. For example, the digital simulation of water flow on terrains is an area of application. Approaches include assigning flow to flat areas to maintain global flow and thus Morse properties \cite{magillo2013morphologically} or combining waterfall transforms with discrete Morse theory \cite{comic2016computing}.}\marginnote{Review Y: Major 2} 

Our work has some similar features as a project to study sublevel set persistence of functions and branching graphs of functions \cite{biswas2023geometric,ost2024banana} and their dynamic update \cite{cultrera2024dynamically}. This work uses a structure they call a {\em window} to help construct barcodes following the elder rule \cite{edelsbrunner2010computational} and generally arguments are made on continuous functions rather than discrete ones. We \deleted{will} use a box snake structure \cite{essl2024deform} that is loosely similar, but makes no assumption on a given barcode construction rule. Additionally, none of our arguments assume genericity. 

\subsection{Contributions}

Contributions of this work include:
\begin{enumerate}
    \item Provide a user-friendly framework for discrete level set persistence on finite discrete functions. These can come from sampling processes from application domains or thought of as discrete versions of familiar functions such as, $f:X\to \mathbb{R}$. In particular, $X$ is a finite set and this domain can be viewed as an interval or as a circle (see Figure \ref{fig:sampledsine}). We develop a framework using a set theoretic perspective and provide a web-based applet for users to explore techniques presented throughout this paper. This allows for a framework that data \replaced{practitioners}{practicioners} \marginnote{Review Z: Minor 2} from a wide range of expertise can understand and use.
    \item Develop sublevel and superlevel set persistence jointly without making any genericity assumptions or requiring isolated extrema. In particular, we prove results on the Morse-like behavior of level set persistence for finite discrete functions. We also discuss duality results between sublevel and superlevel set persistence in both the interval and circular domain setting. 
    \item Discuss the nature of barcode construction rules in this setting, and articulate novel notions of barcode constructions that differ from elder rules, and instead address concerns arising from shifting data. These modifications lead to a more stable barcode, meaning that adding a new point to the discrete function only slightly changes the barcode. 
    \item A theory of surgery for box snake structures. Box snakes keep track of minima, maxima, \deleted{as well as} \added{and} monotone sequences between them. They allow for computation of barcodes as well as homology-preserving deformations. The surgery allows us to realize arbitrary length shifts of data via surgery steps, and hence enables a streaming application of the method.
\end{enumerate}

The remainder of the paper is organized as follows. In section \ref{sec:motivation}, we discuss the motivation for the current approach from previous work on sublevel set persistent homology on functions. In section \ref{sec:setting}, we discuss our settings and key results from the settings, including the central Proposition \ref{prop:trisection} and Theorem \ref{thm:strictsetduality}. The Morse-like behavior \replaced{for}{in our} discrete finite set\replaced{s}{ setting} \marginnote{Review Q: Minor 2} is developed in section \ref{sec:morse}. In section \ref{sec:barcode}, we discuss barcode construction, including a proposal for alternative constructions different from the widely used elder rule. This is followed by section \ref{sec:dual}, which discusses duality results, where virtually all of them are a consequence of Proposition \ref{prop:trisection} and Theorem \ref{thm:strictsetduality} on the duality under chains of inclusion \ref{thm:strictsetduality}. In section \ref{sec:snakeboxes} we discuss box snakes, an auxiliary data structure that segments \deleted{the} discrete data into homologically relevant and homologically invariant blocks. Surgery of barcodes and box snakes is discussed in section \ref{sec:surgery}, and we conclude in section \ref{sec:conclusion}. The appendix gives details of some relationships of interest such as the relationship of the codomain $\mathbb{R}$ and finite discrete sets in Appendix \ref{app:quantsamp}, and the relationship of finite ordered sets to a simplicial representation in Appendix \ref{app:connected}.

%% file: background.tex
\section{Motivation and Prior Approaches}\label{sec:motivation}

The typical setting for sublevel set persistent homology of functions is that of a function over the reals $f:\mathbb{R}\rightarrow\mathbb{R}$ with additional assumptions. Sometimes the domain does follow a discrete process such as that of a time series \cite{baryshnikov2024time} yet the codomain remains that of the reals.

\subsection{Morse Theory and its Connections to TDA}
Morse theory \cite{milnor1963morse} is a way of studying the topology of a manifold or complex via differentiable real-valued functions on that manifold or complex. The basic observation is that the topology changes at the critical points. 

Consider a real-valued smooth function $f:M\to \mathbb{R}$ on a differentiable manifold $M$. If the matrix of second partial derivatives (the \emph{Hessian} matrix) at a critical point is non-singular, then that critical point is \emph{non-degenerate}.  A corollary to the \emph{Morse Lemma} is that non-degenerate critical points are isolated.

Two fundamental results in Morse theory are:
\begin{enumerate}
    \item If $a<b$, $f^{-1}[a,b]$ is \emph{compact}, and there are no critical values in $[a,b]$ then $f^{-1}(-\infty,a]$ is \emph{diffeomorphic} to $f^{-1}(-\infty, b]$ (i.e., they have the same topology).
    \item If $p$ is a non-degenerate critical point of $f$, then for small $\varepsilon>0$, $f^{-1}(-\infty, p-\varepsilon]$ is not diffeomorphic to $f^{-1}(-\infty, p+\varepsilon]$ (i.e., the topology changes at $p$). 
\end{enumerate}

In regards to (2), knowing the \emph{index} of the critical point tells you exactly how the topology changes \marginnote{Review Y: Minor 2}\added{\cite{milnor1963morse}}, and one can construct a \emph{CW-complex} that is homotopy equivalent to the manifold \added{by using appropriate attachment maps \cite{banyaga2004lectures}}. Specifically, $M$ is homotopy equivalent to a CW-complex that has one cell of dimension $i$ for each critical point of $f$ of index $i$.

Classical Morse theory needs to make assumptions on the function $f$. In particular, it requires that extrema be isolated such that the Hessian is indeed non-degenerate. Then we have a unique way to locally characterize the extrema. Furthermore, for a function $f$ to be Morse, it is required that extrema are generic in order to make the map a function and not multi-valued. Underlying this is the characterization of extrema by local curvature and as a function.

\marginnote{Review Y: Minor 3}\added{Bott \cite{bott1954nondegenerate} showed that classical Morse theory, often referred to as Morse-Bott theory \cite{banyaga2004lectures}, can be generalized from critical points to critical manifolds. The requirement is that the critical set is a disjoint union of submanifolds whose Hessian is non-degenerate in the normal direction of these manifolds. Geometrically, this is saying that these manifolds lie flat. The connected set of critical points in this theory now forms a critical submanifold, and expands the role of the critical point. For example, a round torus lying on a surface can be modeled this way. To be clear about the context, the Morse function may be referred to as the Morse-Bott function in this setting. However, the same requirements apply to Morse and for Morse-Bott functions that remain in place. Critical submanifolds are required to be isolated. For example, a flat annulus cannot be Morse-Bott as the Hessian vanishes on the flat part of the annulus. Critical submanifolds still needs to be generic.}

\added{It is a standard proof in Morse theory that while not all functions are Morse, most functions are Morse, and all functions can be made to be Morse because they are generic, meaning that a mild perturbation of non-Morse function will make it Morse \cite{bott1982lectures}. This is the basis for perturbative arguments in the use of Morse theory in topological data analysis.}%

Discrete Morse theory was developed by Robin Forman \cite{forman2002user} and is an analog to Morse theory on smooth manifolds. A common use of discrete Morse theory is to reduce the number of simplices in a simplicial complex so that the homotopy type of the simplicial complex is preserved. This is done by constructing a \emph{discrete Morse function} on the simplicial complex using a gradient vector field. This tells the user how to apply a sequence of \emph{elementary simplicial collapses} to reduce the original simplicial complex while preserving topological information. Forman's Discrete Morse theory mirrors the requirements of classical Morse theory in that discrete gradients on simplices have to be locally definable.

Both classical and discrete Morse theory are ubiquitous in TDA. Many computational implementations of persistent homology use discrete Morse theory in order to reduce the \marginnote{Review Z: Minor 1}\replaced{simplicial}{simiplicial} complex \cite{harker2014discrete, mischaikow2013morse}. Furthermore, \emph{Reeb graphs} (Chapter VI, Section 4 of \cite{edelsbrunner2010computational}), traditionally a tool of Morse theory, are commonly used to capture the changes in level sets of a function. Similarly, merge trees capture changes in \marginnote{Review Y: Minor 4}\added{sub-}level sets of a function but are guaranteed to be trees and are more computationally feasible \cite{SmirnovTriplet20}. \emph{Mapper graphs} \cite{singh2007topological} are a visualization tool in TDA and often viewed as a discrete analog to Reeb graphs \cite{munch2016convergence}.

In regards to this paper, many previous results and applications on sublevel set persistence utilize Morse theory, whether explicitly stated or not. This is why many results use assumptions relying on critical points being \replaced{at distinct levels}{isolated}. We refer to this condition as \emph{generic} or general position assumptions \cite{biswas2023geometric, belton2023extremal, CURRY2024102031, curry2018fiber, belton2020reconstructing}. Sometimes ad hoc observations are used to claim that flats do not change barcodes  \cite{glisse2023fast} and that pattern values with identical minima can be perturbed without problems due to stability. Some approaches handle flat extrema implicitly, such as by using extrema finding libraries that handle flat extrema \cite{Myers2022additivenoise}. Finally, Morse theory plays an important role in setting the stage for Extended Persistence \cite{cohen2009extending}. Extended persistence combines sublevel and superlevel set persistence to arrive at a finite barcode representation. It has been used recently as the chosen encoding of persistent homological data on time series with data points in $\mathbb{R}$ \cite{biswas2023geometric,cultrera2024dynamically}.

Although these Morse and general position assumptions are nice for theory and may be suitable for some applications, they are not always practical. Perturbations fundamentally change data. For instance, in audio data, the sound of silence is flat. If we perturb this signal in order to obtain isolated extrema for the purposes of applying Morse theory, we would add irrelevant information to the signal. Furthermore, any extremum from the perturbation of a constant function has no discernible interpretive value. \marginnote{Review Y: Minor 1}\deleted{Additionally, many data practitioners would be hesitant to add noise to data before applying some technique, as that is contrary to how they have been trained.} Incorporating a technique with restrictions on the dataset would require conditioning steps should the technique be incorporated in a larger processing chain. Assume for example that one wants to compute sublevel set persistence on the output of each eigenfunction of a Discrete Fourier Transform. Given that all eigenfunctions are sinusoids, their extrema are not generic. Hence, genericity requirements would necessitate introducing a signal conditioning step that guarantees genericity of this output\added{ and relating of persistent homology of a signal to the persistent homology of its Fourier transform becomes difficult if this relationship is not precise but subject to intermittened perturbation}.

Fortunately, as we show in this paper, flat minima and maxima do not pose an issue when determining where the topology changes in sublevel set persistence of finite discrete functions. Furthermore, the isolation of extrema per level is not required and can be replaced by a sub-filtration within a level. Hence, we construct a framework for analyzing level set persistence without having to change our dataset. \marginnote{Review Y: Minor 2}\added{With respect to Morse theory, we do not require a definition of a flow or gradient and can handle functions that have no gradient anywhere such as constant functions. Our function levels only differ by order, so we do not require any metric structure for our approach.}

%% file: setting.tex
\section{Setting}\label{sec:setting}

Our object of study is a finite discrete sequence of discrete levels. This is motivated by the typical computational representation of digital computation, as well as the nature of collected data by digital sensors,  that is stored discretely, and with finite discrete resolution. 
The goal is to provide a user-friendly framework that completely operates in a discrete setting. This means that we will use finite ordered sets for both the domain and the codomain of our discrete functions. The relationship of continuous functions to the discrete setting is discussed in more detail in Appendix \ref{app:quantsamp}.

\begin{remark}
The setting of ordered sets allow us to construct a purely set theoretic discussion of the finite discrete level functions. It also has the advantage that this provides a general setting for many different discretizations that result in the same order.    
\end{remark}

\begin{figure}[ht]
\includegraphics[width=\textwidth]{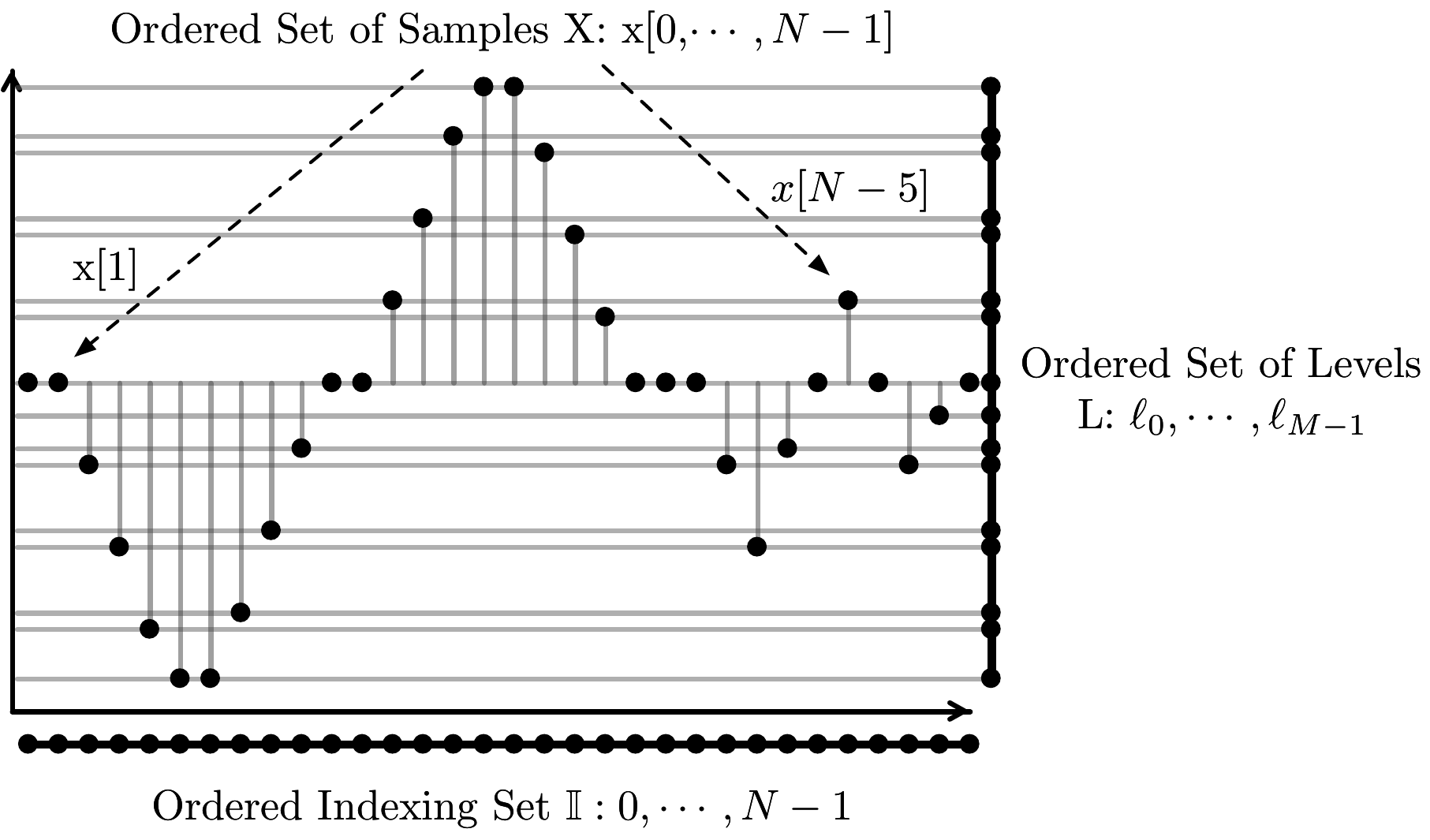}
\caption{Data in our setting. A discrete finite function $x: \mathbb{I}\rightarrow L$ from an ordered finite set in the domain $\mathbb{I}$ to an ordered finite set in the codomain $L$.}\label{fig:setfunction}
\end{figure}

\subsection{Discrete Functions via Ordered Sets}

Consider a finite discrete sequence of a totally ordered set $X:=x[0,\ldots,N-1]$. Denote $\mathbb{I}:=\{0,1,\dots, N-1\}$ as the \emph{indexing set}\footnote{The reader should be warned, that we will be using the zero-based numbering (sometimes also called Engineer's convention). Hence, indices start from $0$, which is different from the convention widely used in pure mathematics which starts index labeling at $1$. This means that the last element is $N-1$ for an index of $N$ elements. This convention is adopted so that the paper's content matches array indexing in source code attached to the paper.} and $L:=\{\ell_0,\ell_1,\ldots,\ell_{M-1}\}$ as the \emph{set of levels}. Then we will use the notation $x[0],x[1]\dots,x[N-1]$ to denote individual maps of the \emph{finite discrete function} 
$x:\mathbb{I} \to L$, where $i\mapsto x[i]$ (see Figure \ref{fig:setfunction}). Since $L$ contains all sequence levels and nothing more, the function $x$ is surjective ($|\mathbb{I}|\geq |L|$ hence $N\geq M$). An element $\ell_\bullet\in L$ is a \emph{level}, and any individual $x[i]$ for $i\in \mathbb{I}$ is called a \emph{sample}. 
    
\subsection{Topology of the Domain}

Recall that $N$ is the number of samples, and all our index sets start with $0$. For all domains, we are considering $x[0], \dots, x[N-2]$ as having a successor defined as $x^\succ[n]=x[n+1]$. We call the relationship $x^\prec[n]=x[n-1]$ as the predecessor relationship, which for all domains is defined for $x[1],\dots, x[N-1]$. If both a successor and predecessor are present, we call it an adjacency relationship $x[n]\succprec x[n-1]$. We call samples that do not have a successor or predecessor {\em boundary samples}. All samples in the complement of the boundary are {\em interior}. In the absence of a boundary, all samples are interior.

    \begin{figure}
    \begin{subfigure}[t]{.49\textwidth}
    \centering
        \begin{tikzpicture}[scale=\textwidth/6.8cm,samples=200]
            \begin{axis}[%
                standard,
                domain = 0:31,
                samples = 32,
                xlabel={$n$},
                ylabel={$x[n]$},
                xticklabel=\empty,yticklabel=\empty
                ]
                \addplot+[ycomb,black,thick] {sin(2*180*x/16)};
                \addplot[black,dashed] {1};
                \addplot[black,dashed] {-1};
            \end{axis}
        \end{tikzpicture}
        \caption{}\label{subf:linear}
        \end{subfigure}
        \hspace{0.001\textwidth}
        \begin{subfigure}[t]{.49\textwidth}
                \begin{tikzpicture}[scale=\textwidth/6.8cm,samples=200]
\begin{axis}
    [
    view={60}{30},
    axis lines=none
    ]
\addplot3[
    domain=0:2*pi,
    samples = 32,
    samples y=0,
]
({sin(deg(x))},
{cos(deg(x))},
{0});
\addplot3[
    domain=0:2*pi,
    samples = 32,
    samples y=0,
    dashed
]
({sin(deg(x))},
{cos(deg(x))},
{-1});
\addplot3[
    domain=0:2*pi,
    samples = 32,
    samples y=0,
    dashed
]
({sin(deg(x))},
{cos(deg(x))},
{1});

\foreach \i in {0,...,32}
    \addplot3 [domain=0:1, samples y=1, black] (({sin(deg(2*pi/32*\i))},{cos(deg(2*pi/32*\i))}, {x*sin(deg(4*pi/32*\i))});
\foreach \i in {0,...,32}
      \addplot3+[mark=*,only marks, draw=black, mark options={color=black},mark options={fill=white},thick] coordinates {({sin(deg(2*pi/32*\i))},{cos(deg(2*pi/32*\i))},{sin(deg(4*pi/32*\i))})};

    \addplot3 [domain=-1.35:1.35, samples y=1, black] (-1,0, x);
    \addplot3 [mark=triangle,black,smooth,point meta=explicit symbolic,nodes near coords] coordinates {(-1,0, 1.35)[{$x[n]$}]};

\end{axis}
\end{tikzpicture}
\caption{}\label{subf:circular}
\end{subfigure}
        \caption{A finite discrete sequence of samples $x[n]$ with $32$ samples of a sampled sine of period $16$. \subref{subf:linear} $\mathbb{I}_{[N]}$: sample sequence domain with boundaries at $0$ and $N-1$. \subref{subf:circular} $\mathbb{I}_N$: sample sequence domain without boundaries on a discrete circular domain of period $N$ .}\label{fig:sampledsine}
    \end{figure}
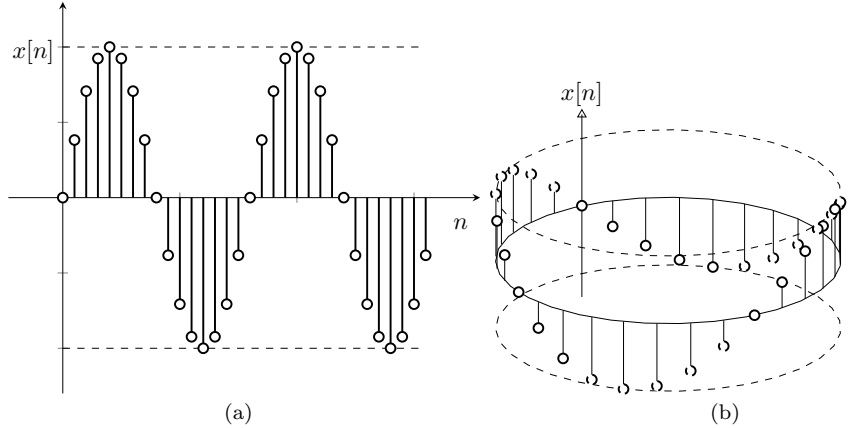

We will consider two cases of sequences (illustrated in Figure \ref{fig:sampledsine}). In the first case, the sequence $x[0],\dots, x[N-1]$ is a linear order with boundary samples at $0$ and $N-1$. We denote this case as $\mathbb{I}_{[N]}$ to indicate the presence of boundaries at $0$ and $N-1$, and call it a {\em linear domain} or {\em linear sequence}. The second case is a periodic sequence with the same number of samples, which we denote as $\mathbb{I}_N$\footnote{This notation alludes to the familiar notation $\mathbb{Z}_N$ for finite cyclic groups.}. We call this a {\em circular domain} or {\em circular sequence}. In the case of a circular sequence $\mathbb{I}_N$, we have $x^\succ[N-1]=x[0]$ and $x^\prec[0]=x[N-1]$. Hence the periodic sequence has no boundary samples and all samples have adjacency relationships on both sides. If we refer to a {\em sequence}, the statement applies to both cases and we simply denote the domain as $\mathbb{I}$.




\subsection{Level sets}

In this paper, we work with level sets that are discrete and finite. We begin by defining the three types of level sets we work with: level, sublevel, and superlevel.

\begin{definition}[Level Set]
A \emph{level set} $L_l$ is the set of all sample positions $i \in \mathbb{I}$ with sample values $x[i]$ at a given level $l\in L$. That is,
$$L_l:=\Set{i \in \mathbb{I} | x[i]=l, \; l\in L}.$$
\end{definition}

\begin{definition}[Sublevel Set]
\label{def:structsublevel}
A sublevel set $L_{<l}$ is the union of all level sets $L_s$ such that $s<l$ in the ordered set $L$. That is,
$$L_{<l}:=\bigcup_{\{s\in L \mid s<l\}}L_s.$$
\end{definition}

\begin{definition}[Superlevel Set]
\label{def:structsuperlevel}
A superlevel set $L_{>l}$ is the union of all level sets $L_s$ such that $s>l$ in the ordered set $L$. That is,
$$L_{>l}:=\bigcup_{\{s\in L \mid s>l\}}L_s.$$
\end{definition}

Notice, that the definition of sublevel and superlevel sets are formally identical except for the direction of the order.

\begin{remark}
In the literature it is customary (compare \cite{dey2007stability}) to define a sublevel set as follows: $L_{\leq l}=f^{-1}(-\infty,l]$. A superlevel set is thusly defined as $L_{\geq l}=f^{-1}[l,\infty)$. These two definitions are not disjoint due to the overlap at $l$. Notice that our versions exclude the level $l$, hence $L_{<l}$, $L_l$, and $L_{>l}$ are disjoint sets.
\end{remark}





\subsection{Relationships of Sublevel, Level, and Superlevel sets}

We will use the following notations for set operations. $\complement$ indicates the set complement $\complement:A=T\setminus B$ where $A$ is the complement, $T$ is the total set, $B$ is the set to be complemented, and $\setminus$ is the usual set subtraction. Notice, that obviously $B=T\setminus A$. Also $\complement\complement$ is the identity operation. This is part of what is known as the {\em principle of duality of sets}. This states that inclusions and equations using intersections, unions, empty sets and total sets are dual in the sense that intersections are exchanged with unions, total sets are exchanged with empty sets, or more generally, subsets are replaced with their complements, and inclusions are replaced by inclusions in the opposite direction \cite{halmos1974naive}. The disjoint unions of $B=A\djunion C$ is characterized by its (left and right) inverses, the set subtraction $A=B\setminus C$ and $C=B\setminus A$. 


The following \deleted{Theorem} \added{proposition} is central to many of the results regarding duality in our setting:

\begin{proposition}[Disjoint Union Trisection of Level Sets]
The disjoint union of a level, sublevel, and superlevel set always form the total set: $T=L_{<l}\djunion L_l\djunion L_{>l}$. This disjoint union is dual under inversion of order $<\leftrightarrow>$.\label{prop:trisection}
\end{proposition}
This follows from $L_{<l}\djunion L_l\djunion L_{>l}=L_{>l}\djunion L_l\djunion L_{<l}$. Many of our results in this paper are a consequence of this Proposition. The first will be in the context of detecting extrema.

\subsection{Detecting Local Extrema and Monotonicity from Level Sets}

Local extrema play a crucial role throughout. Here we describe how they can be detected from level set properties, leading to their definition in our given setting. In particular we see that flat extrema introduce no difficulties for us.

\begin{definition}{(A Flat).}
A {\em flat} is \marginnote{Review Y: Minor 5}\replaced{the maximal}{a} sequence of one or more consecutive samples of the same level\deleted{, bounded by differing values}:  $x[m-1]\neq x[m]=\ldots =x[n]\neq x[n+1]$. The length of a flat is $n-m$.
\end{definition}

Observe that our definition of a flat includes a single sample at a given level, which we \deleted{will} call {\em isolated}. This definition of flat does not refer to extrema. It applies to any sequence of samples on the same level. 

A sequence of samples is {\em monotonic} if it obeys a partial order: $x[m]\leq \ldots \leq x[n]$. Evidently, our definition of a flat is contained in the definition of a monotonic (sub)sequence. A flat in a monotonic sequence, is bounded by samples either decreasing or increasing (see Figure \ref{fig:flats} for an example of an increasing \marginnote{Review Z: minor 3}\replaced{monotone}{monotome}.)

There are two further cases for flat extrema: (1) a {\em flat minimum}, bounded by samples both larger, (2) a {\em flat maximum}, bounded by samples both smaller. 

The same distinctions can be made within the framework of Proposition \ref{prop:trisection}. Consider any level set $L_l$. It consists of a subset of the total space and we can identify connected components by finding neighboring members in the set to define a {\em flat} in the level set. Hence\added{,} we observe that the notion of a flat can be defined solely within the level set. We call samples directly neighboring a flat in $L_l$ but not part of it the {\em outer boundary} of the flat.

To define a notion of an extremum, or a monotone, however, we need at least one of the sublevel or superlevel sets. Consider the level set $L_l$ and the sublevel set $L_{<l}$ and the superlevel set $L_{>l}$. By the definition of a flat in $L_l$, the outer boundary of the connected component of the flat in $L_l$ is not in $L_l$. Therefore they have to be either in $L_{<l}$ or $L_{>l}$. If both outer boundaries are in the same set then the flat is a local extremum. It is a {\em local minimum} if the outer boundary is in the superlevel set $L_{>l}$ and it is a {\em local maximum} if the outer boundary is in the sublevel set $L_{<l}$. However, due to Proposition \ref{prop:trisection}\added{,} this same observation also holds for superlevel sets, with order inverted. More importantly, we can check just the level set for flatness and just one sublevel or superlevel set to determine if there is an extremum or not, and if yes, we can detect the type of extremum.

For example, we can detect an increasing monotone in a sublevel set if the left outer boundary is in the sublevel set but the right boundary is not. When we dualize to the superlevel set\added{,} we invert the order. Hence\added{,} an increasing monotone is a decreasing monotone in the other. If the domain is linear, we have a further case, namely that one or both outer boundaries of a flat may be outside the finite set of the domain. If one outer boundary lies within the set, then we determine that it is always a local extremum, determined from the outer boundary inside the set. We will see that homologically they play a special role. If there is no outer boundary within the set, then the whole set is necessarily flat and we consider it both a global minimum and maximum (see Corollary \ref{corr:constant}).

Figure \ref{fig:flats} illustrates examples of different ways connected components (flats) in a level set $L_l$ are classified by their outer boundary. We see on the far left and right examples of flats at the boundary. We will call them {\em boundary extrema} to indicate that their classification is only arrived at by one outer boundary point (away from the boundary points of the linear domain). All other points are classified according to our classification based on both outer boundary points.

With local and global extrema defined and classified, we are ready to consider Morse-like behavior in our setting.

\begin{figure}[ht]
\includegraphics[width=\textwidth]{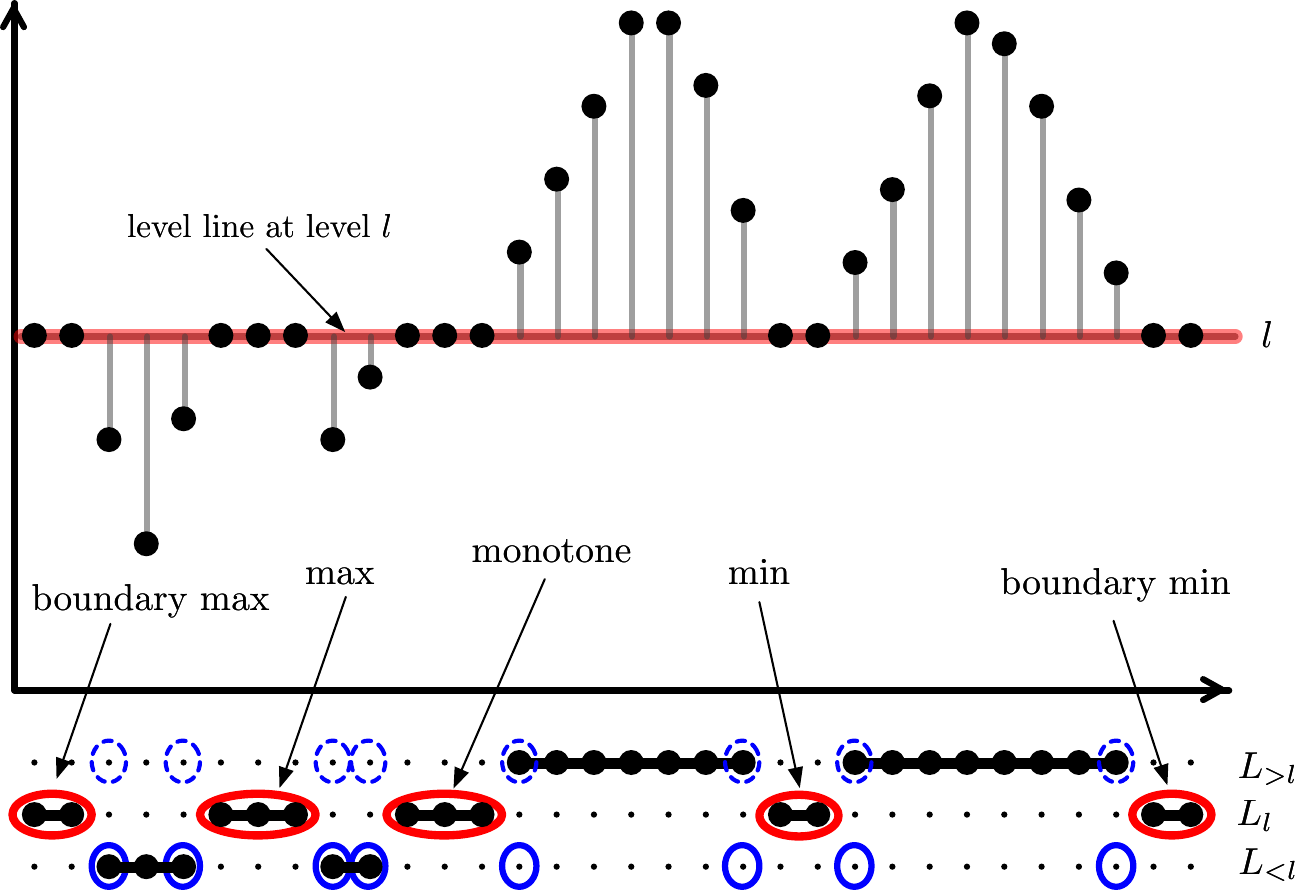}
\caption{Examples of connected components in the level set $L_l$ (red circled) are characterized by their outer boundary in the sublevel set (blue circled positions). Notice that by duality we could equally use the outer boundaries in the superlevel set (dashed blue) to arrive at the same classification.}\label{fig:flats}\marginnote{Review Y, Major 6}
\end{figure}

\subsection{Homology of a Discrete Level Set}

Central to analyzing the topology of sublevel, level, and superlevel sets of a function, is studying the number of \emph{connected components}. On first observation, we are dealing with ordered finite sets. Hence in the sense of point-set topology every member of the set is a singleton and thus we have a discrete topology and nothing is connected. However, due to the predecessor and successor relationship given by the order, we have an {\em adjacency relationship} which carries the spirit of connected components in our setting. 

\begin{definition}[Connected Component]
Let $L_\bullet$ be some subset of $\mathbb{I}$. A \emph{connected component} of $L_\bullet$ is a \added{non-empty} subset $C \subset L_\bullet$ \marginnote{Review Y: minor 6}\replaced{it includes all elements in $C$ that can be}{such that for every $c, d \in C$, $d$ can be} reached by a finite number of successor or predecessor operations.
\end{definition}

\begin{remark}
The above definition can be used on any subset of $\mathbb{I}$. We only study connected components of level sets, sublevel or superlevel sets. Hence the notation $L_\bullet$ is a placeholder for $L_l$, $L_{<l}$, or $L_{>l}$. 
\end{remark}

We can represent the successor/predecessor relationship with a $1$-simplex (abstract line) and turn any subsequence into a simplicial complex of simplicial lines (alternating $0$-simplicies at sample points, and connecting $1$-simplicies). This gives us a conventional connected component in the simplicial setting\deleted{s}\marginnote{Review Q: Minor 3}. See technical details in Appendix \ref{app:connected}. \added{Both these pictures are classical in digital signal processing and have been used in graph and topological signal processing \cite{ortega2022introduction,robinson2014topological} and can be viewed as a dimension-reduced version of $v$-construction cubical complexes used in topological analysis of digital images \cite{bleile2022}.}\marginnote{Review Y: Major 2}

\begin{example}
   Consider the sequence $X: x[0, 1, \ldots, n-1, n, \ldots N-1]$ where $x[0], x[1], \dots, x[n-1]$ are all equal to level $l$ and each $x[n], \dots, x[N-1]$ is not equal to $l$. Then the level set $L_l = \{0,1,\dots, n-1\}$ is a connected component given that each member in the set is adjacent to another. No element of the sequence with index $n$ or greater is in the connected component. The associated simplicial complex of $L_l$ is a line graph with $n$ vertices and $n-1$ edges. 
\end{example}


Furthermore, the usual homology of such a simplicial complex coincides with the notion of homology of finite discrete sequences as used here. Colloquially, torsion-free (or homology on coefficients that do not capture interesting torsion information such as $\mathbb{Z}/2\mathbb{Z}$ or $\mathbb{R}$) homology can be described as counting dimensional voids, which in turn can be computed by counting the number of independent cycles that do not come from boundaries. In the case of dimension zero homology, this counts the number of connected components.

In our setting we have the following definition:

\begin{definition}[Zeroth Betti Number]\label{def:homology0}
The \emph{zeroth Betti number} (or Betti-0, $\beta_0$) is the number of connected components in a subset of $L_\bullet=\mathbb{I}$. We take Betti-0 to capture the zeroth homology $H_0(L_\bullet)$ of the level set $L_\bullet$.
\end{definition}

\begin{remark}
A more conventional algebraic perspective via simplicial homology can be equivalently defined. See Appendix \ref{app:connected}.
\end{remark}





If a subset $L_\bullet$ of a finite discrete indexing set $\mathbb{I}$ contains the whole indexing set, it is called {\em total}.\footnote{What we call a total set is also called a universal set in set theory or sometimes a complete set in functional analysis. We chose total because it unifies the language. We will use the same notion for total sequence, total space, etc., i.e., denoting the colloquial and intuitive notion of containing the totality of what is considered.} If all samples are connected, we call this situation a {\em total connection}. For example, if we have a constant set of samples at level $y$, the level set $L_y$ contains the total connection or is total.

\begin{definition}[First Betti Number]\label{def:homology1}
The \emph{first Betti number} (or Betti-1, $\beta_1$) is the number of cycles in our domain. In our setting, it can only be $0$ or $1$. Namely, $\beta_1$ of a level set is $1$ if and only if the level set is total and the domain is circular. Otherwise, $\beta_1 = 0$. Betti-1 captures the first homology $H_1(L_\bullet)$ of the level set $L_\bullet$. 
\end{definition}

There exists a canonical correspondence between Betti numbers and homology groups via the fundamental theorem of finitely generate\added{d}\marginnote{Review Q: Minor 4} abelian groups \cite{munkres1984elements}. This tells us that the Betti number counts the number of free groups in the homology group. Using $\mathbb{Z}$ as the symbol for a free group, we can describe homology as follows:

The homology (with coefficients in $\mathbb{Z}$) of the total set of the linear domain is $H_0(\mathbb{I}_{[N]})=\mathbb{Z}, H_1(\mathbb{I}_{[N]})=0$ and of the periodic sequence is $H_0(\mathbb{I}_{N})=\mathbb{Z}, H_1(\mathbb{I}_{N})=\mathbb{Z}$. Hence, we observe that the two cases differ in the 1st homology, capturing the presence or absence of a cycle in the domain. In our setting, the absence of a cycle implies the presence of boundaries.
\marginnote{Review Q: Major 2, part 2}\added{$H_1$ in our setting only serves to distinguish between the topology of the total domain, and is not used otherwise.}


\subsection{Sublevel and Superlevel Set Filtrations from Level Sets}

In this section, we develop a key result that states that sublevel and superlevel set filtrations contain dually the same information. It is known that circular domains \deleted{we} have a duality under sign inversion \cite[Corollary 3.4]{biswas2023geometric}. Our result in the same case does not require genericity. Furthermore, we are able to extend this duality to our bounded case in Section \ref{sec:dual}. Proposition \ref{prop:trisection} together with the principle of duality of sets then leads to duality with respect to a sequence of inclusions between sublevel and superlevel sets.

\begin{theorem}[Duality of Level Sets]
Any construction from inclusion maps involving level sets that satisfy Proposition \ref{prop:trisection} is dual under change of order and inclusions in the opposite direction. \label{thm:strictsetduality}
\end{theorem}

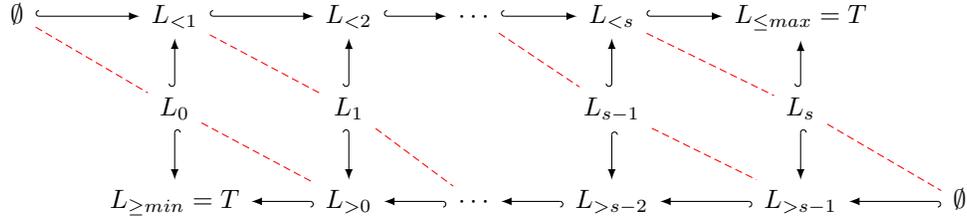
\begin{figure}
\centering
\begin{tikzcd}
\emptyset \arrow[r, hook] \arrow[rd, no head,red,dashed] & L_{<1} \arrow[r, hook] \arrow[rd, no head,red,dashed]              & L_{<2} \arrow[r, hook]                                   & \cdots \arrow[r, hook] \arrow[rd, no head,red,dashed] & L_{<s} \arrow[r, hook] \arrow[rd, no head,red,dashed]                  & L_{\leq max}=T                                              &                           \\
                                              & L_0 \arrow[u, hook] \arrow[d, hook] \arrow[rd, no head,red,dashed] & L_1 \arrow[u, hook] \arrow[d, hook'] \arrow[rd, no head,red,dashed] &                                            & L_{s-1} \arrow[u, hook] \arrow[d, hook] \arrow[rd, no head,red,dashed] & L_s \arrow[u, hook] \arrow[d, hook] \arrow[rd, no head,red,dashed] &                           \\
                                              & L_{\geq min}=T                                                & L_{>0} \arrow[l, hook]                                 & \cdots \arrow[l, hook]                     & L_{>s-2} \arrow[l, hook]                                    & L_{>s-1} \arrow[l, hook]                                & \emptyset \arrow[l, hook]
\end{tikzcd}
\caption{Duality under set complement of sublevel inclusions and superlevel inclusions preserving the disjoint union of sublevel, superlevel, and level sets. Observe that the sequence is dual under inversion of order (indicated by the red dashed line) due to Proposition \ref{prop:trisection}.
}
\label{fig:inclusionduality}
\end{figure}

\begin{proof}
We construct the inclusion of level sets into a sublevel set for all levels. We construct the same for the superlevel set. We observe that the resulting diagram of Figure \ref{fig:inclusionduality} is dual under Proposition \ref{prop:trisection} (red dashed lines in the diagram.)
\end{proof}

From Figure \ref{fig:inclusionduality}, we also see notions of a global minimum and maximum are dualized. These can be defined with respect to the inclusion sequence given that the first level of the sequence with respect to the order is empty, and the last level of the sequence is the total set. We can now define a global minimum and maximum independent of the choice of order by calling the global minimum as the first inclusion into the empty set, and the global maximum as the case of the last inclusion leading to the total set. This means that the global minimum of the sublevel set corresponds to the global maximum of the superlevel set, and vice versa. This is captured by the following two lemmas and hold by duality also for superlevel sets.



\begin{figure}[bht]
\centering
\includegraphics[width=1.0\textwidth]{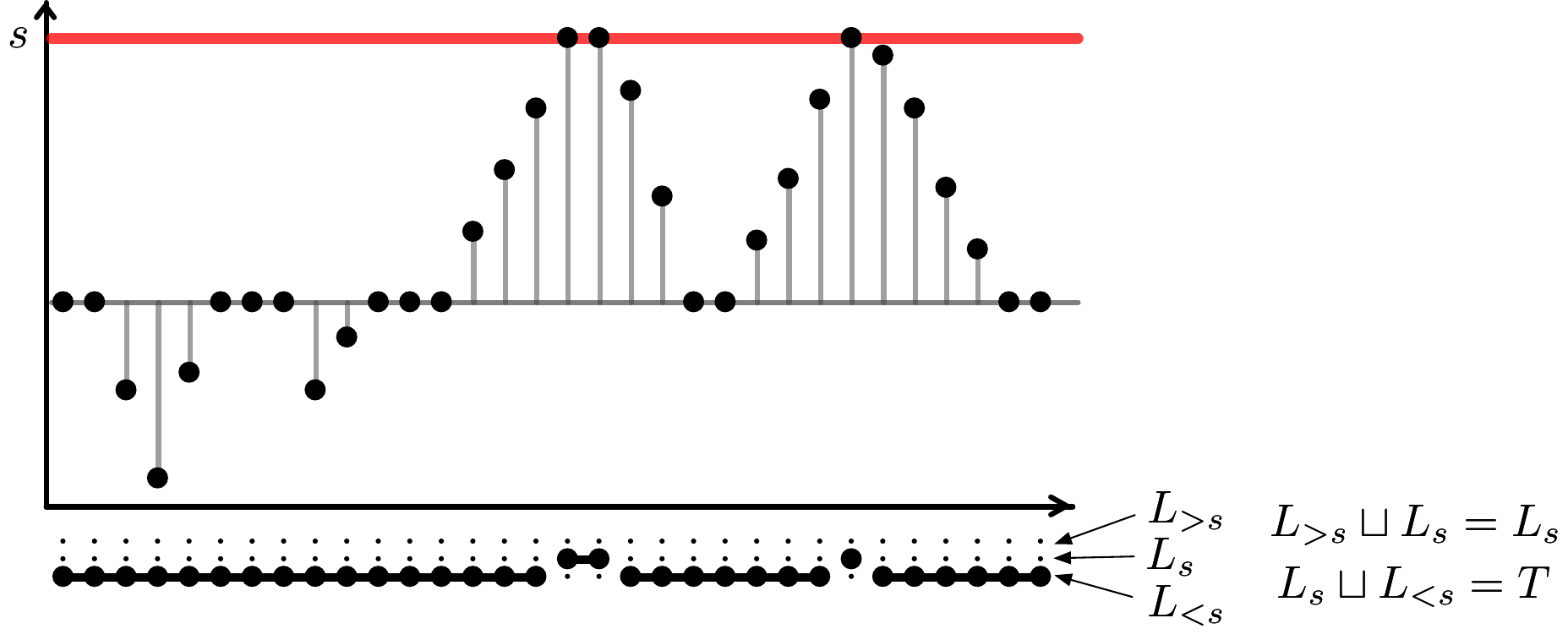}
\caption{Level sets at global extrema at the level with index $s$ (red line). 
Below the graph we see (top) the superlevel set $L_{>s}$, (middle) the level set $L_s$, and (bottom) the sublevel set $L_{<s}$.}\label{fig:globalextrema}   
\end{figure}

\begin{lemma}[Level Sets and Global Minima]\label{lem:globalmin}
A sublevel set has at least one global minimum at level $i$ if and only if $L_{<i+1}\djunion L_i=L_i$, meaning $L_i=L_{<i+1}$. A superlevel set has a global minimum at level $i$ if and only if $L_{>i-1}\djunion L_i=L_i$, meaning $L_i=L_{>i-1}$\marginnote{Review Q: Minor 5}.
\end{lemma}

\begin{lemma}[Level Sets and Global Maxima]\label{lem:globalmax}
A sublevel set has a global maximum at level $i$ if and only if $L_{<{i}}\djunion L_i=T$, meaning $\complement L_{<i}= L_{i}$. A superlevel set has a global maximum at level $i$ if and only if $L_{>i}\djunion L_i=T$, meaning $\complement L_{>i}=L_{i}$.
\end{lemma}

These lemmas tell us that global minima and maxima can be checked by comparing the level set with a single extremal set. The lemmas are illustrated in Figure \ref{fig:globalextrema}. It shows the level line at the global maximum $s$ of a sublevel set and dually at the global minimum of the superlevel set obeys Lemma \ref{lem:globalmax} and Lemma \ref{lem:globalmin}, respectively. Finally, as a consequence of these lemmas we get \deleted{a} that a constant function must both be a global minimum and maximum:

\begin{corollary}[Constant Functions are both \deleted{a} Global Minima and Maxima]\label{corr:constant}
A constant function is both a global minimum and a global maximum as both conditions hold ($L_{<0}=0$ hence $L_0=T$ and $L_{<1}=L_0$ and $\complement L_{<0}=L_{0}$ hence $L_0=T$) and dually for superlevel sets.    
\end{corollary}


\subsection{Persistent Homology via Level Set Filtration\added{s} and Within-Level Subfiltration\added{s}}\label{subsec:withinlevelfiltration}

Persistent homology tracks the persistence and changes of homology under some process, such as the successive inclusion of connected pieces. Sequences of inclusions in this context are called a {\em filtration}. In this paper, we are jointly considering two filtrations (sublevel sets and superlevel sets) that are constructed from sequences of inclusions of level sets due to the duality of Figure \ref{fig:inclusionduality}.

We get a filtration (a sequence of inclusions) in the total order of distinct levels and its dual:

\begin{align}\label{eq:levelfiltration1}
0&\xhookrightarrow{} H_\bullet(l_0) \xhookrightarrow{} \cdots \xhookrightarrow{} H_0(l_{M-1})=H_\bullet(\mathbb{I})\\\label{eq:levelfiltration2}
H_\bullet(\mathbb{I})&=H_\bullet(l_0)\xhookleftarrow{} \cdots \xhookleftarrow{}H_\bullet(l_{M-1}) \xhookleftarrow{}  0
\end{align}


This inclusion gives us sublevel set persistence on a finite discrete sequence and dually the same for superlevel set persistence. Notice that the inverse of an inclusion in the sublevel set is the inclusion in the superlevel set.

\subsubsection{Within-level Filtration}

Given that we do not assume genericity, multiple extrema can fall on the same level set. If a minima or interior maxima fall onto the level set, then the homology change by inclusion of that level set will be:

\begin{align}
\beta_0(L_{<n}\djunion L_n)=\beta_0(L_{<n})+\#\text{minima}-\#\text{interior maxima} 
\end{align}

Most interestingly\added{,} if the number of minima and interior maxima is identical in the level set\added{,} then there is no change in homology at the granularity of the level set. An example of this phenomenon is illustrated in Figure \ref{fig:subfiltration}.

\begin{figure}[th]
\centering
\centering
    \begin{subfigure}[t]{.325\textwidth}
    \centering
    \includegraphics[width=0.91\textwidth]{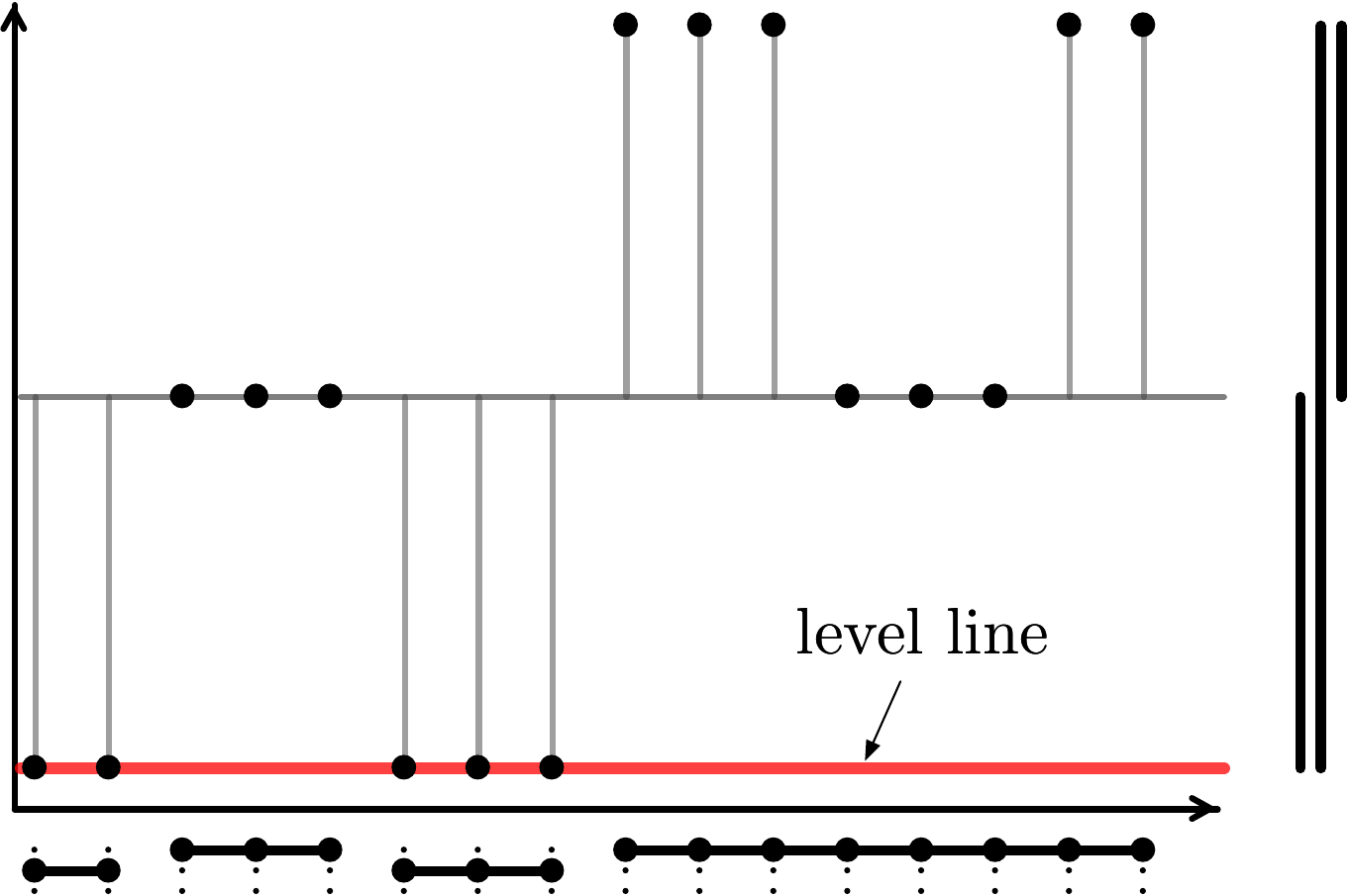}        \caption{}\label{subf:subfiltA}
    \end{subfigure}
    \begin{subfigure}[t]{.325\textwidth}
    \centering
    \includegraphics[width=0.91\textwidth]{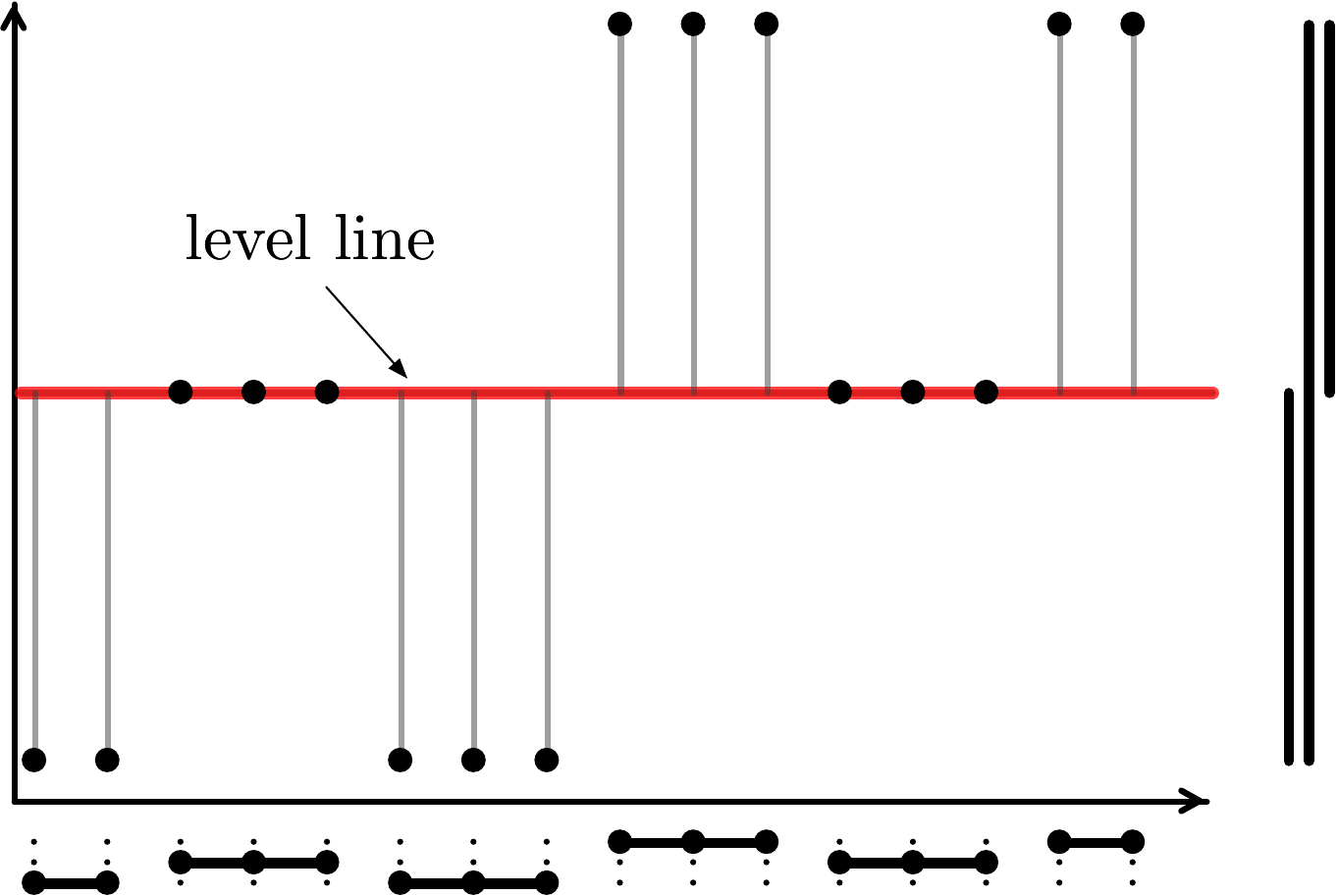}        \caption{}\label{subf:subfiltB}
    \end{subfigure}
    \begin{subfigure}[t]{.325\textwidth}
    \centering
    \includegraphics[width=0.91\textwidth]{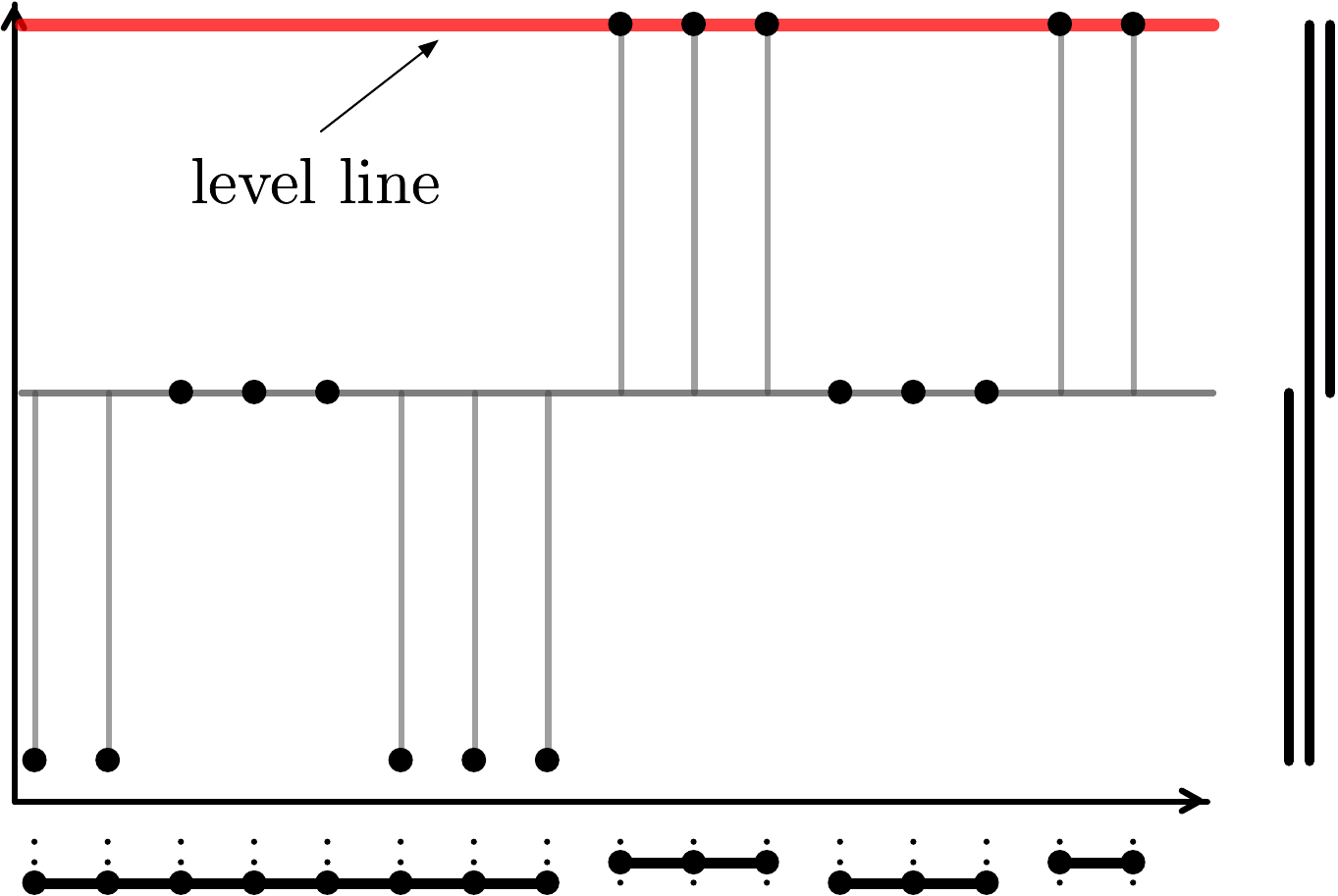}        \caption{}\label{subf:subfiltC}
    \end{subfigure}
\caption{Sublevel set persistent homology where the number of connected components is two at all levels. \subref{subf:subfiltA} Two connected components are created at two minima. \subref{subf:subfiltB} A local maximum merges two connected components while at the same time a local minimum creates a new connected component, keeping the number of connected components at two. \subref{subf:subfiltC} The two connected components persist until the global maximum, where they are merged together.}\label{fig:subfiltration}   
\end{figure}

Hence, in order to capture this information, we require a {\em within-level filtration}. Let $L_n=L_n^0\djunion\ldots\djunion L_n^c$ be the decomposition of a level set $L_n$ into its connected components. Then

$$H_0(L_n^0) \xhookrightarrow {} H_0(L_n^0\djunion L_n^1) \xhookrightarrow {} \cdots \xhookrightarrow {} H_0(\bigdjunion_{i=0}^c L_n^i)=H_0(L_n)$$

is a within-level filtration. If we then refine the level set filtration of Equation~(\ref{eq:levelfiltration1}-\ref{eq:levelfiltration2}) we get a definition of a filtration guaranteed to show changes by local extrema within a level. The order at which within-level extrema are included is arbitrary. \marginnote{Review Y: Minor 7}\added{This should be distinguished from approaches that use tie-breaking rules to resolve non-generic extrema in prior approaches \cite{edelsbrunner1990simulation}. All within-level filtrations are equivalent when it comes to the effect to the construction of sublevel inclusions. This means that all within-level inclusions have to be completed before continuing to the next level.}

%% file: morse.tex
\section{Morse-like Behavior}\label{sec:morse}

The recognition that under suitable assumptions, homological changes are consequences of extrema in the level set is core to classical Morse theory \cite{morse1934calculus,milnor1963morse}. This has been adopted in the discrete setting for simplicial complexes by Forman \cite{forman2002user}. Both classical and Forman's discrete Morse theory are often invoked to argue properties of level set persistent homology. Even if Morse theory is not directly invoked, the notion of a homological critical value \cite{cohen2005stability,bubenik2014categorification,govc2016definition} captures the relationship of homological change and critical values in persistence. A key aspect of this work is to analyze the Morse-like behavior of finite discrete functional data.

\begin{figure}[th]
\includegraphics[width=\textwidth]{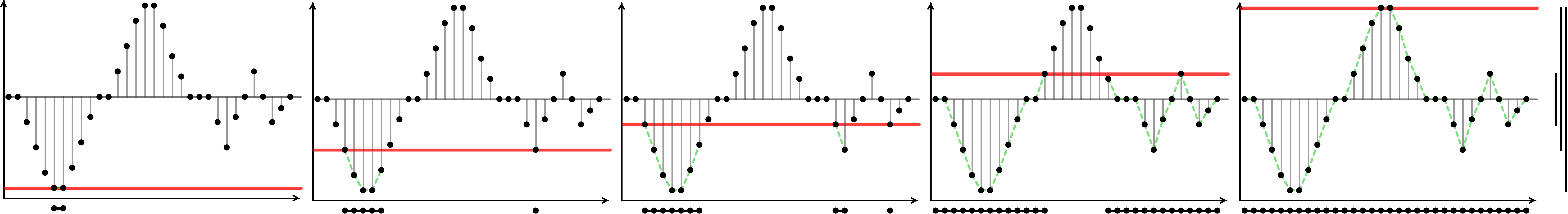}
\caption{Sublevel Set Persistent Homology of a Sequence. For each figure, the sublevel set is illustrated by the line graph for the level highlighted in red. The green dashed line highlights the points that are contributing to the sublevel set. The far right shows the persistence barcode.}\label{fig:sublevelset}   
\end{figure}

In our setting of taking level sets of finite discrete functions, we ask the question: \emph{When is homological information added, and when is it removed?} We observe that flat minima create one connected component. This has the homotopy-type of a point and thus is homotopy equivalent to a classical Morse minimum. We also observe that a flat maximum that is not in the boundary, joins two connected components. This has the same property as a classical Morse maximum. See Figure \ref{fig:sublevelset} for examples of this behavior. Notice that both flat and isolated minima and maxima are present, and that the homological information changes precisely at extrema independent of the flatness of the extrema. We prove this next.

\begin{proposition}[Connected Components are Created at Minima]
\label{prop:min}
Let $x:\mathbb{I}\to L$ be a finite discrete function and consider the sublevel set filtration. A connected component is created at level $l$ if and only if $l$ is the level of a local (flat) minimum.
\end{proposition}

\begin{proof}
    We first assume a connected component is created (or born) at a level of $l$ and show that $l$ is the level of a local minimum. Since a connected component is created at $l$, there must exist a subsequence $m, m+1, \dots, n-1, n$ such that: 
    \begin{itemize}
        \item $x[m]=x[m+1]=\dots =x[n-1]=x[n]=l$
        \item neither $m-1$ nor $n+1$ (if they exist) are contained in $L_{< l}$.
    \end{itemize}  
    Note that if $m-1$ and/or $n+1$ were in $L_{< l}$, then $m,m+1, \dots, n-1, n$ would add to the connected component that includes $m-1$ and/or $n+1$ at the level $l$. Therefore, $x[m-1]>l$ and $x[n+1]>l$. This shows that $x$ has a local (flat) minimum at $m, m+1, \dots, n-1, n$. Observe this result still holds if either $x[m-1]$ and/or $x[n+1]$ is not in the sample.

    Next we assume $l$ is the level of a local (flat) minimum and show that a connected component is created (or born) at a level of $l$. Suppose $m, m+1, \dots, n-1, n$ is the subsequence of the local flat minimum and let $l$ be the level of the minimum, so $x[m]=\dots =x[n]=l$. Since no value is less than $l$ for $m, m+1, \dots, n-1, n$, no connected component contains $m, m+1, \dots, n-1, n$. Furthermore if $x[m-1]$ or $x[n+1]$ exist, their values are strictly larger than $l$ by definition of a local minimum. Assuming $x[m-1]$ and $x[n+1]$ are in the sample, we find that no connected components contain $m-1, m, \dots, n, n+1$ for the strict sublevel set $L_{<l}$. At the level set $L_l$, note that the flat minimum $m, \dots, n$ is in the level set but $m-1$ and $n+1$ are not. Therefore $m, \dots, n$ is a connected component in the sublevel set $L_{\leq l}$, but is not in the sublevel set $L_{<l}$. This shows a local flat minimum creates a connected component. Observe that this result still holds if either $x[m-1]$ and/or $x[n+1]$ is not in the sample.

\end{proof}

Note that the proof only uses set properties and a total order on the level. Incidentally, the complementary proof for local flat maxima is essentially the same, except that we necessarily need existing entries in at $m-1$ and $n+1$ so that the local flat connects two connected components together, i.e., merges two connected components.

\begin{proposition}[Connected Components Merge at Maxima]
\label{prop:max}
    Let $x:\mathbb{I}\to L$ be a finite discrete function. Two connected components merge together at a level $l$, if and only if $l$ is the level of a local (flat) maximum $x[m], x[m+1] \dots, x[n]$ where $m\neq 0$ and $n\neq N-1$.
\end{proposition}

In general, minima and maxima play this complementary role for reasons of a duality of sublevel set persistence and superlevel set persistence. We discuss this in greater detail next.

\begin{remark}
Notice that Proposition \ref{prop:max} requires two connected components to merge. A maximum at the boundary of the domain has no connected component to merge as there is nothing to merge with outside the domain. This means that maxima at the boundary do not contribute to the homology.
\end{remark}

Minima do not show this property. They always create a connected component, and hence contribute to homology. This leads to an asymmetry between minima and maxima for the linear domain which has boundaries.

%% file: barcodes.tex
\section{Barcodes and their Construction Rule}\label{sec:barcode}


When a level set containing local maxima is included in a sublevel set, two or more connected components are merged into one. This means there is one connected component that continues, but the remaining components are now subsumed into that surviving connected component. One can think of this as one connected component ``surviving" the merge event, while one or more have ``died".  This will decrease the number of connected components. However, the attribution of which one gets to be designated as survivor or not is arbitrary. 

We will call the rule that decides which connected component is designated to continue, the {\em bar construction rule}. By far, the most frequently used bar construction rule is called the ``elder rule" \cite{edelsbrunner2010computational} or ``youngest first" \cite{dey2022computational} depending if one wants to emphasize a metaphor of survival or not. We use the ``elder rule" nomenclature. \added{The basic idea of the elder rule is that the ``oldest" meaning the earliest created homological class survives when merged. In our setting, this would refer to the lowest minimum participating in a merge.}
If the data at a level set are generic, that is, no two \marginnote{Review Y: Minor 8}\replaced{minima}{maxima} can fall onto the same level, then the elder rules gives a unique bar construction rule \cite{edelsbrunner2010computational,baryshnikov2024time,curry2018fiber}, and this is the typical context in which this rule is encountered.

Our setting is emphatically not generic, so even if we try to retain the spirit of the elder rule, we should in general expect the barcode to be non-unique \cite{curry2018fiber}. For example, if three minima at the same level are merged at two maxima between them which are at the same level, three connected components are merged into one and we have three identical candidates to survive the merge. Given that there is no distinguishing information, any selection of survivor is, at least with respect to level, arbitrary. Morally we could call this case of the elder rule a ``random tribal council elder", indicating that of numerous ``equally old" options, one arbitrary elder is selected to carry on. Other elder selections can be envisioned in this case such as using the order in the sequence, such as first elder or last elder. We will not pursue details of these choices in this paper, and will invoke an arbitrary rule (such as left-most elder).


\subsubsection{Finite Essential Barcode over Finite Level Sets}

Often a filtration is constructed from some metric information (compare for example \cite{cohen2005stability}). In the case of sublevel set persistence, this would be a height function defined over $\mathbb{R}$. For any finite height function, there is a global maximum. This means there is a moment where the whole function is included in the sublevel set, and we have one final connected component. But given that the height function is defined over $\mathbb{R}$, one can imagine continuation without change of the homological information. This is usually captured by an infinite-length barcode $[b,\infty)$ where $b$ is the birth event. We do not use a height function over $\mathbb{R}$ in our work. All bars are born (and die) in the closure of all non-empty levels. Hence by construction, we do not have infinite bars. However, we have a direct association between an infinite bar by recognizing that the connected component of the total sequence plays the identical role. This bar is sometimes called {\em essential} \cite{dey2022computational}. As a consequence, our barcodes are naturally bounded by the closure of the level set, and thus gives a natural answer for truncating infinite bars for practical display or processing that is present in the case of height functions over $\mathbb{R}$.

\subsection{Barcode Construction for Circular Domains}

On circular domains, persistent homology of sublevel and superlevel sets are dual, meaning one contains the same content as the other. A related property of interest is the fact that the number of minima and maxima are guaranteed to be the same. Hence for every minimum there is a maximum to create a bar in a barcode. Furthermore, for every pair of minima and maxima (henceforth {\em pair of opposite extrema}) there are two paths that connect them. This means every constructed barcode always has two interpretations: going either left or right in order to have the minimum and maximum be a part of the same bar. For every bar in the sublevel set, one has an identical bar in the superlevel set with births and deaths reversed. This suggests that in the circular case, the denotation of birth and death is merely a consequence of having picked one orientation over a dual one, and hence has no intrinsic meaning. An example of this duality is shown in Figure \ref{fig:circularduality} for a local barcode construction rule that picks the local minimum to the neighboring local maximum on the left and right.

\begin{figure}[ht]
\centering
    \begin{subfigure}[t]{.49\textwidth}
    \centering
    \includegraphics[width=0.91\textwidth]{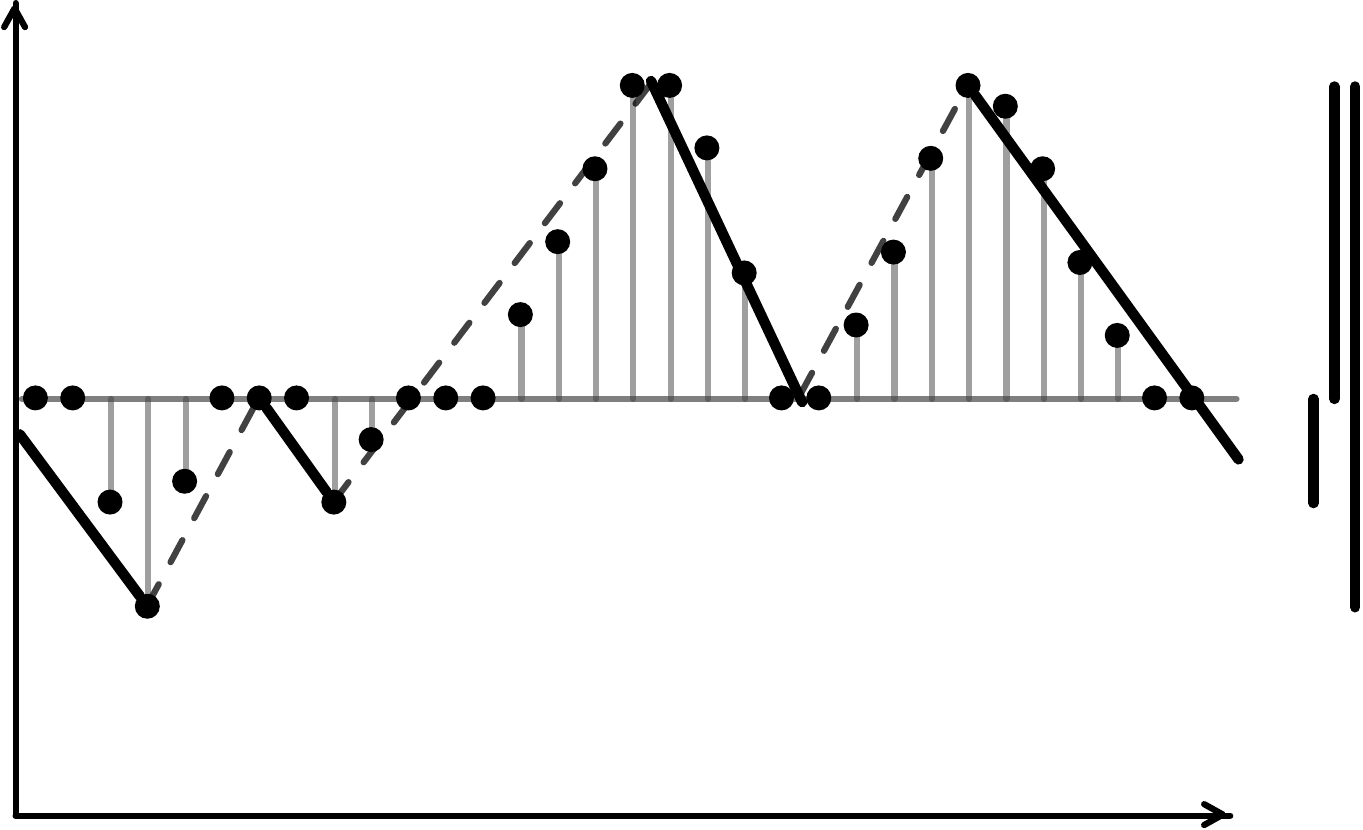}        \caption{}\label{subf:circdualA}
    \end{subfigure}
    \begin{subfigure}[t]{.49\textwidth}
    \centering
    \includegraphics[width=0.91\textwidth]{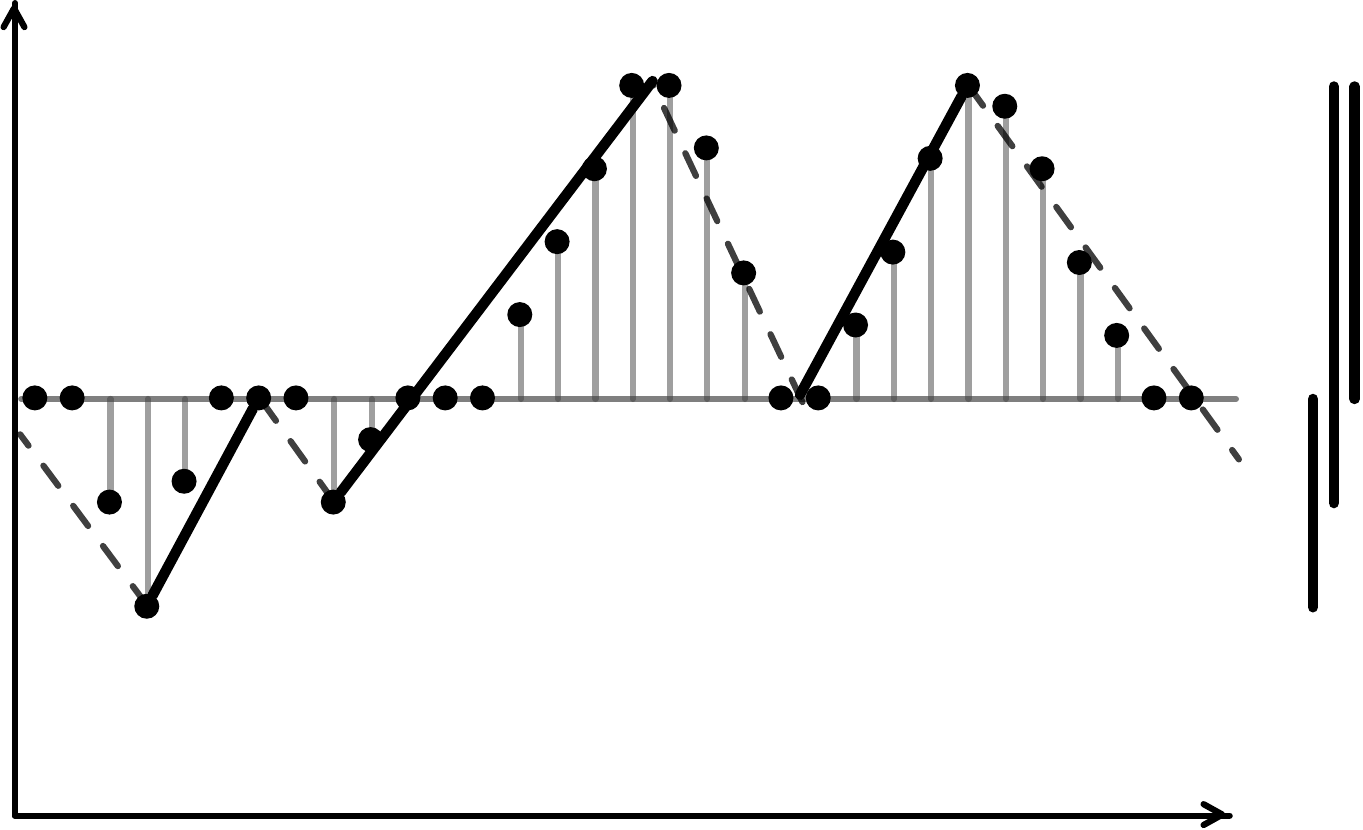}        \caption{}\label{subf:circdualB}
    \end{subfigure}
\caption{Dual barcodes on a cirular domain based on local barcode construction rules. \subref{subf:circdualA} Barcode represents the left neighboring maximum of a local minimum and \subref{subf:circdualB} barcode represents the right neighboring maximum of a local minimum. Observe that if you look at the barcodes as those of superlevel sets, their roles are reversed.}\label{fig:circularduality}
\end{figure}  

\subsection{Barcode Construction Rule for Linear Domains}


In this section, we introduce motivations for choosing different barcode construction rules for linear domains. We believe that the choice of barcode construction is application dependent. Different barcode rules create differing outputs which in turn serve different needs of interpretation in applications.

The linear domain is more complicated than circular domains due to effects at the boundary. To better help us illustrate this more complicated setting we will use merge trees. Merge trees are a richer structure that captures the homological behavior due to minima and maxima \cite{carr2003computing,curry2018fiber,curry2022decorated,dey2022computational}. One can view the difference between merge trees and barcodes as a reduction in information due to a rule that the homological contributor survived at a merge.

\begin{definition}[Merge tree]
Given a finite discrete sequence of a totally ordered set $X=x[0,\dots, N-1]$, the \emph{merge tree} of $X$ is a rooted tree with decorated vertices called \emph{nodes}. The \emph{nodes} $N:=\{n_c, c\in \mathbb{N}, i\in\mathbb{I}, x[i]\in L\}$ are decorated by the index $i$ and the level $x[i]$ of the extrema of $X$ which change the homology under a chosen filtration, and are connected to zero, or more other nodes $n_c$ which we will call \emph{children}.
The root of the tree $r\in N$ is the only node in a merge tree that is not a child of another node.
\end{definition}

 
Any node without children is a \emph{leaf} and \emph{interior nodes} are nodes that have children. The leaves of a merge tree are in one-to-one correspondence to minima. Interior nodes of merge trees correspond to levels at which connected components merge. 
Given that merge positions can be ambiguous due to the possibility of multiple maxima at the same level and hence merging into the same node, we use the first maximum that merged connected components of this node in all our figures. Notice that even though the index $i$ in the annotation is ambiguous due to the possibility of a flat region, $x[i]$ for any choice of $i$ surjects onto one unique level.

Given that connected components only merge at maxima, nodes correspond to one or more maxima at the same level that all connected the same merged connected component. Given that maxima are allowed at the same level, this merge tree is not guaranteed to be binary. With this information, it is easy to visualize the correspondence of the data and the merge tree by overlaying the merge tree at its decorated positions on the data. The nodes are associated with extrema that served the captured homological function. Clearly, there is a bijection between homological changes at each level and nodes in the merge tree, and without the decoration, there is a one-to-one correspondence if extrema are isolated. Otherwise, there is a surjection if multiple connected components are merged at the same level due to maxima that share that level.

Our definition has some relation to an ordered merge tree \cite{curry2018fiber} which decorates edges with a total order. By comparison, we decorate nodes; however, the total order of nodes induces a total order on edges. Hence, ordered merge trees are contained in our definition.

The barcode construction rule can be viewed as picking one branch at each interior node of the merge tree. The birth of each bar is a leaf of the merge tree. The death of a bar is either not being picked as the surviving branch at an interior node, or having reached the global maximum. We allow multiple maxima at the same level. This means that multiple connected components can simultaneously merge at a given level. A single maximum that is not at th3e boundary will always merge two connected \marginnote{Review Q: Minor 6}component\added{s}. Each maximum at the same level that is not separated by a higher maximum will also merge in one additional connected component. Finally, \marginnote{Review Q: Minor 7}\deleted{a} the root of the merge tree merges all connected components represented by branches entering it, hence each branch must be part of a merge tree.

\begin{figure}[ht]
\centering
    \begin{subfigure}[t]{.49\textwidth}
    \centering
    \includegraphics[width=0.91\textwidth]{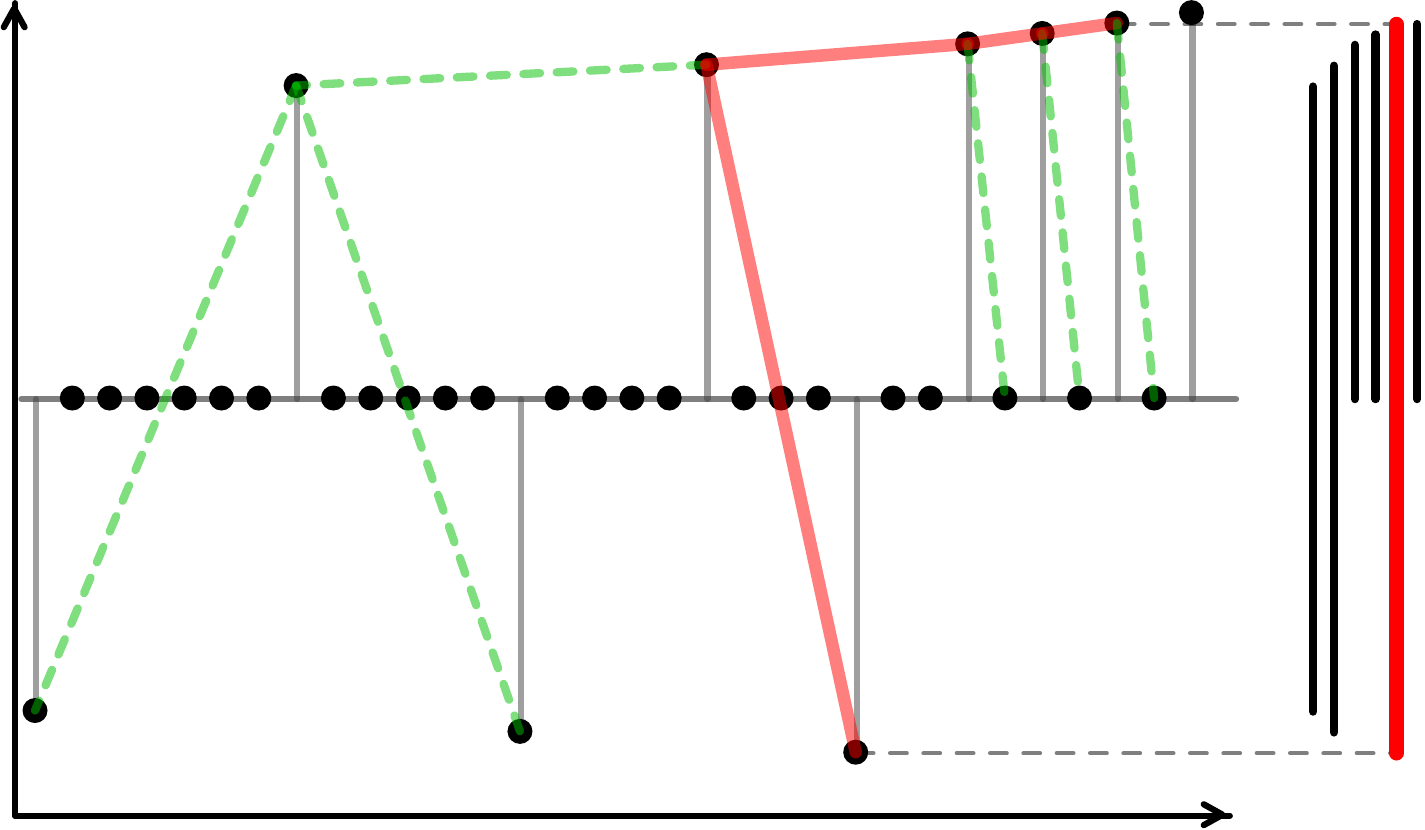}        \caption{}\label{subf:elderlocalA}
    \end{subfigure}
    \begin{subfigure}[t]{.49\textwidth}
    \centering
    \includegraphics[width=0.91\textwidth]{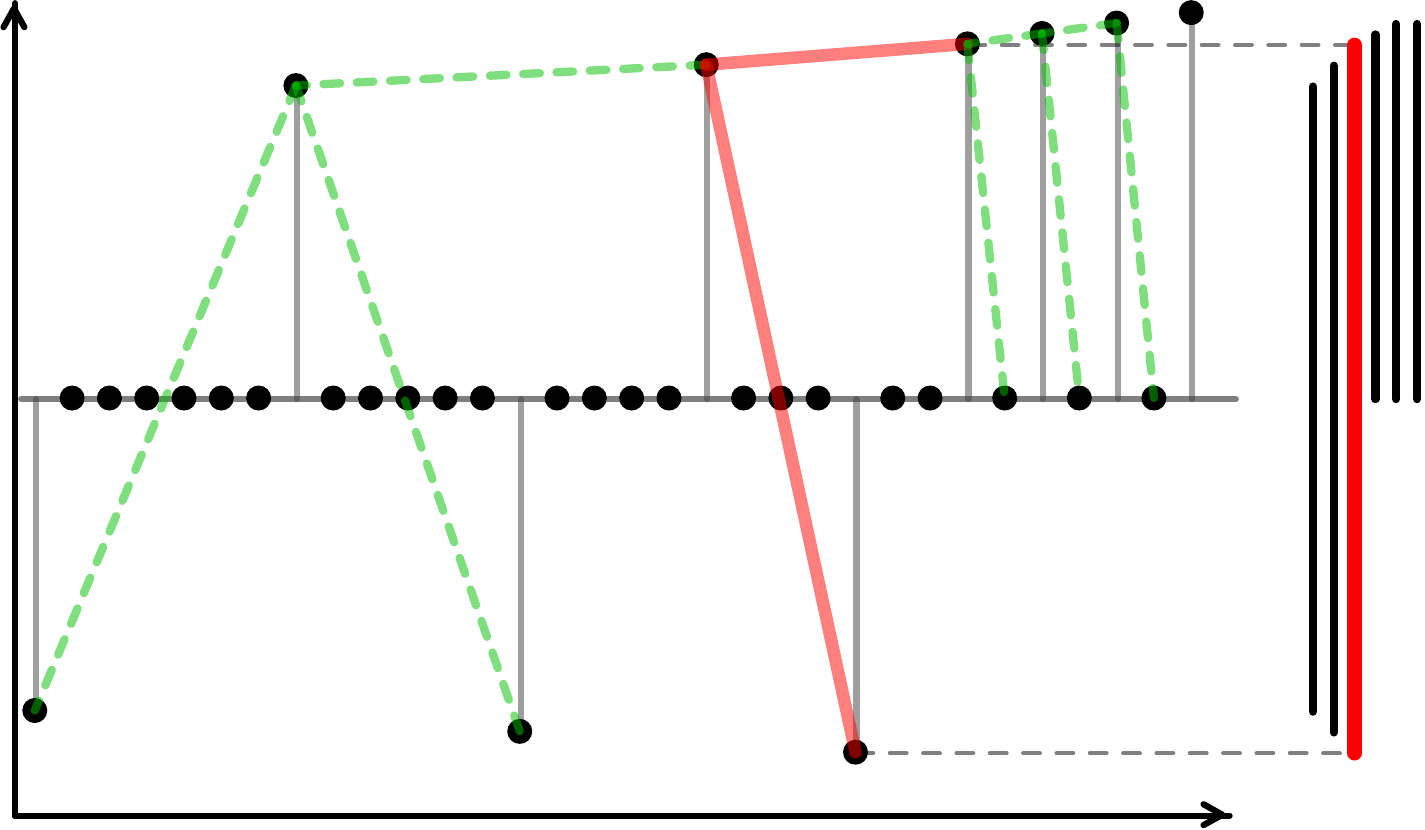}        \caption{}\label{subf:elderlocalB}
    \end{subfigure}
\caption{Rules of barcode construction. \subref{subf:elderlocalA} The elder rule connects a global minimum with a global maximum, which in this example is non-local. \subref{subf:elderlocalB} A local rule connects neighboring extrema. Which neighbor is dictated by the direction to the global maxima. In this example, all but one bar connect to the right neighbor. The sole left neighbor is the minima on the right of the global maximum. Notice that the boundary maximum is not considered as it does not contribute to homology.}\label{fig:globalmineldervslocal}
\end{figure}

The classical {\em elder rule (with tribal council resolution of ambiguity)} says that the branch at a node of a merge tree survives that contains the lowest minimum. \marginnote{Review Q: Minor 8}If multiple branches have \replaced{equally-}{equually }\marginnote{Review Y: Minor 9}low minima the tie is broken by some tribal council rule (we use left-most survives).

Contrast this to a {\em local rule}, which picks a (non-boundary) neighboring maximum for any minimum. \marginnote{Review Q: Minor 9}\replaced{The global interior maxima play a special role. For simplicity, let us first assume that there is just one unique interior global maximum. The left and right connected components merge at that maximum. Both sides lead to barcodes to the global maximum. One is due to it being essential, the other is due to a bar that is non-essential but connecting at this level. So the global maximum fixes that all minima to the left of the global maximum have to connect to the right neighboring maximum, and all minima to the right of the global maximum have to connect to the left of the neighboring maximum. Now if we allow multiple global interior maxima, we have to pick one of them to capture the essential bar. After that the same rule applies, and the remaining global interior maxima are just treated as interior maxima.}{The global maxima dictate the direction of the neighbor that is picked. If a local maximum is to the left of the global maximum, then a local minimum connects to the local maximum to its right, and for local maxima on the right, local minima connect to local maxima to their left.}

Figure \ref{fig:globalmineldervslocal} shows a simple example of the difference of these two rules with regards to a global minimum. The elder rule connects it to the global maximum which is not a neighbor of the minimum (left) while the local rule connects it to a local neighboring maximum according to the side from the global maximum it is on (right).

\begin{figure}[th]
\centering
    \begin{subfigure}[t]{.49\textwidth}
    \centering
    \includegraphics[width=0.91\textwidth]{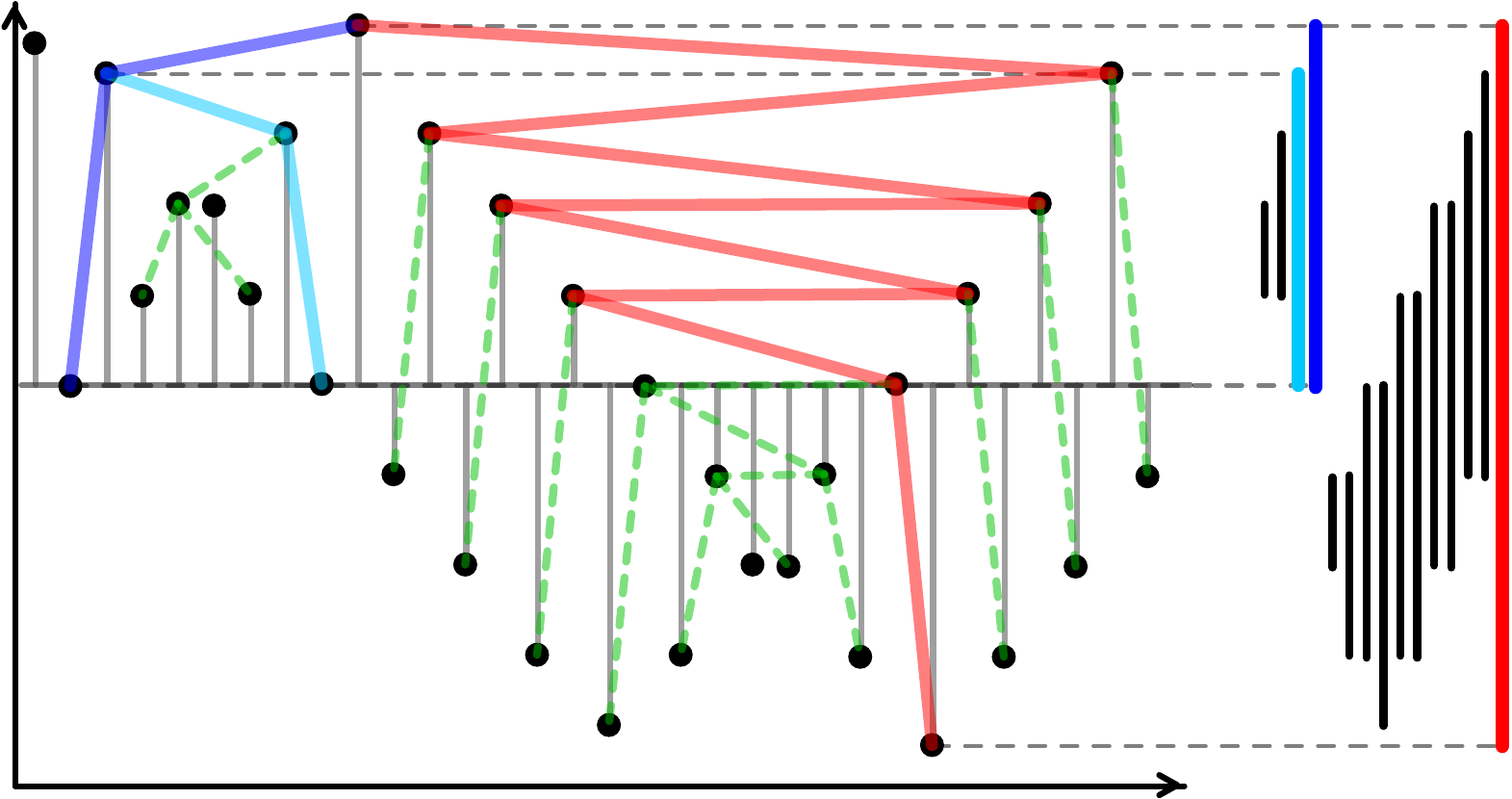}        \caption{}\label{subf:mtelderlocalA}
    \end{subfigure}
    \begin{subfigure}[t]{.49\textwidth}
    \centering
    \includegraphics[width=0.91\textwidth]{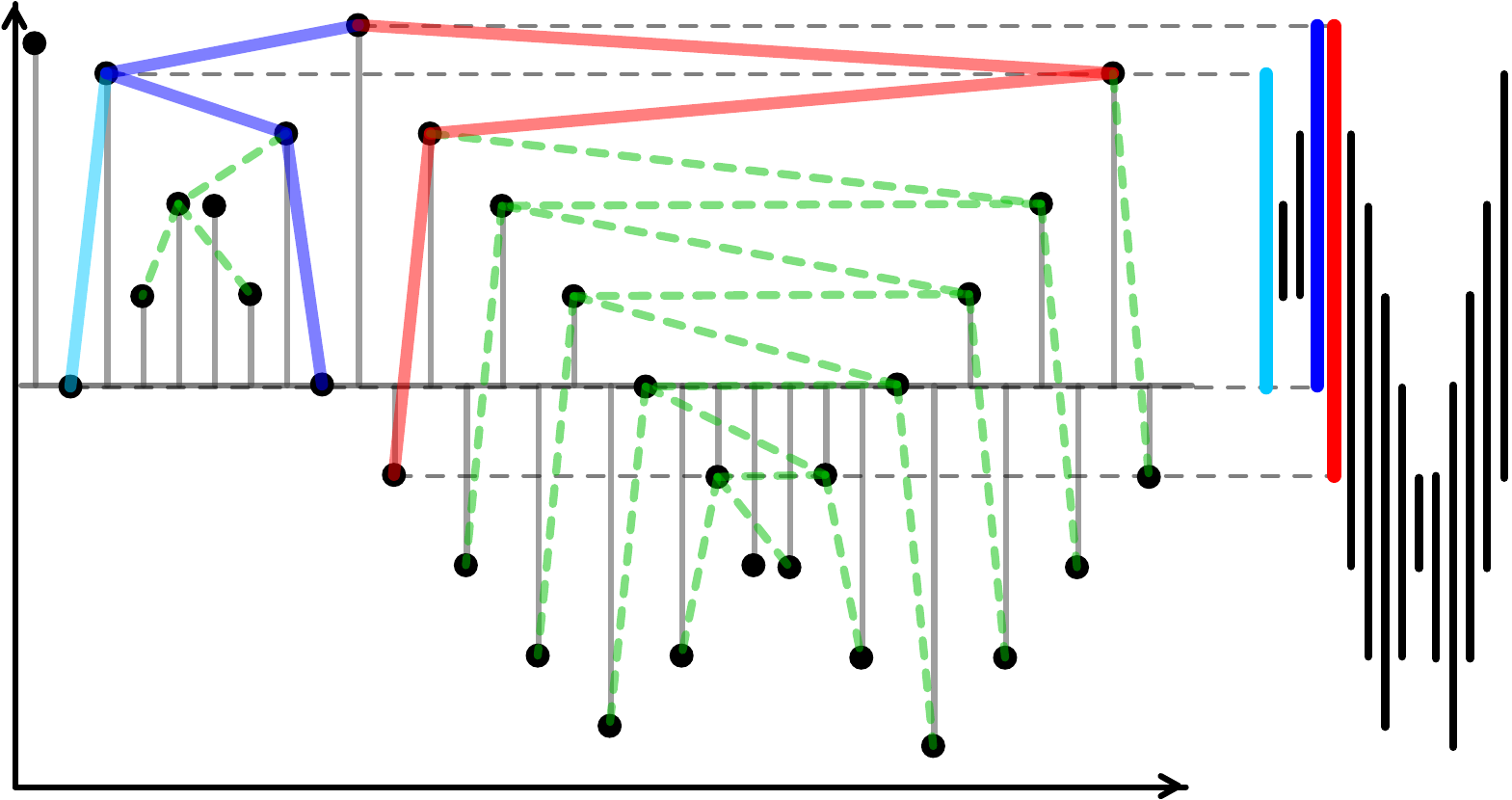}        \caption{}\label{subf:mtelderlocalB}
    \end{subfigure}
\caption{Comparison of elder versus local barcode rules based on bars including global maxima. \subref{subf:mtelderlocalA} The elder rule demands that a global minimum is connected to a global maximum (red solid). Notice that the left side from the global maximum does not have a unique lowest minimum hence we have an ambiguity on what to pick for the bar (blue solid or cyan solid). In this example, the leftmost local elder is picked from the elder council. \subref{subf:mtelderlocalB} The local rule states that a bar starting from a local minimum is connected to a neighboring local maximum. The global maximum will dictate which neighbor. There is no longer an ambiguity between blue and cyan branches as the rule uniquely distinguishes them. The red and blue branches connect to the global maximum. However, either one could be considered ``essential", hence there is a global ambiguity in the rule.}\label{fig:mergetreeelderlocal}   
\end{figure}

In Figure \ref{fig:mergetreeelderlocal}, we show a more complicated example \replaced{of}{on}\marginnote{Review Q: Minor 10} how bars are picked from branches of merge trees. Hence the maximum in the left boundary is ignored as it does not contribute to homology. Minima form leaves of the tree, while maxima form \added{interior}\marginnote{Review Y: 10} nodes of the tree. Global maxima form the root of the tree. If multiple maxima merge at the same level, they form the same node, hence the tree does not necessarily need to be binary. However, every leaf needs to be connected and every node needs to be an endpoint.

An example of the elder rule is \deleted{realized} shown on the left \replaced{in}{of}\marginnote{Review Q: Minor 11} Figure \ref{fig:mergetreeelderlocal}. The red path in the elder case connects the global minimum with the global maximum. At each branch, it survives against other branches until the maximum. There is another connected component on the right side of the global maximum that connects to the global maximum (blue). Notice, however, that both the blue and cyan paths share the same minimum. Hence, the traditional elder rule does not provide sufficient information to pick a survivor. By tribal council rule, the left-most elder (blue) survives.

The local rule is shown on the right in Figure \ref{fig:mergetreeelderlocal}. In contrast to the elder rule, the local rule does not take into account the level of minima, but instead \marginnote{Review Q: Minor 12}\replaced{always constructs bars from immediate opposite extremas}{maintains the local neighborhood}. We see that \added{the} two bars connect to the global maximum in this case (red and blue) are indeed connecting minima neighboring the global maximum. The blue path is uniquely defined by the rule, as it is now the cyan path, which by the rule connects to the maximum immediately to its right. Hence, in the case of the local barcode rule, there is no ambiguity between cyan and blue. However, which of the two paths that connect to the global maximum (blue and red) is considered essential and is not resolved by the rule. Hence, there is an ambiguity here in case the essential property is needed for applications. This ambiguity does not exist for the elder rule for this dataset as there is a unique global minimum. However, if there were multiple global minima, the elder rule would again have to invoke a tribal council rule to resolve which bar is essential.

\begin{figure}[ht]
    \begin{subfigure}[t]{.49\textwidth}
    \centering
    \includegraphics[width=0.91\textwidth]{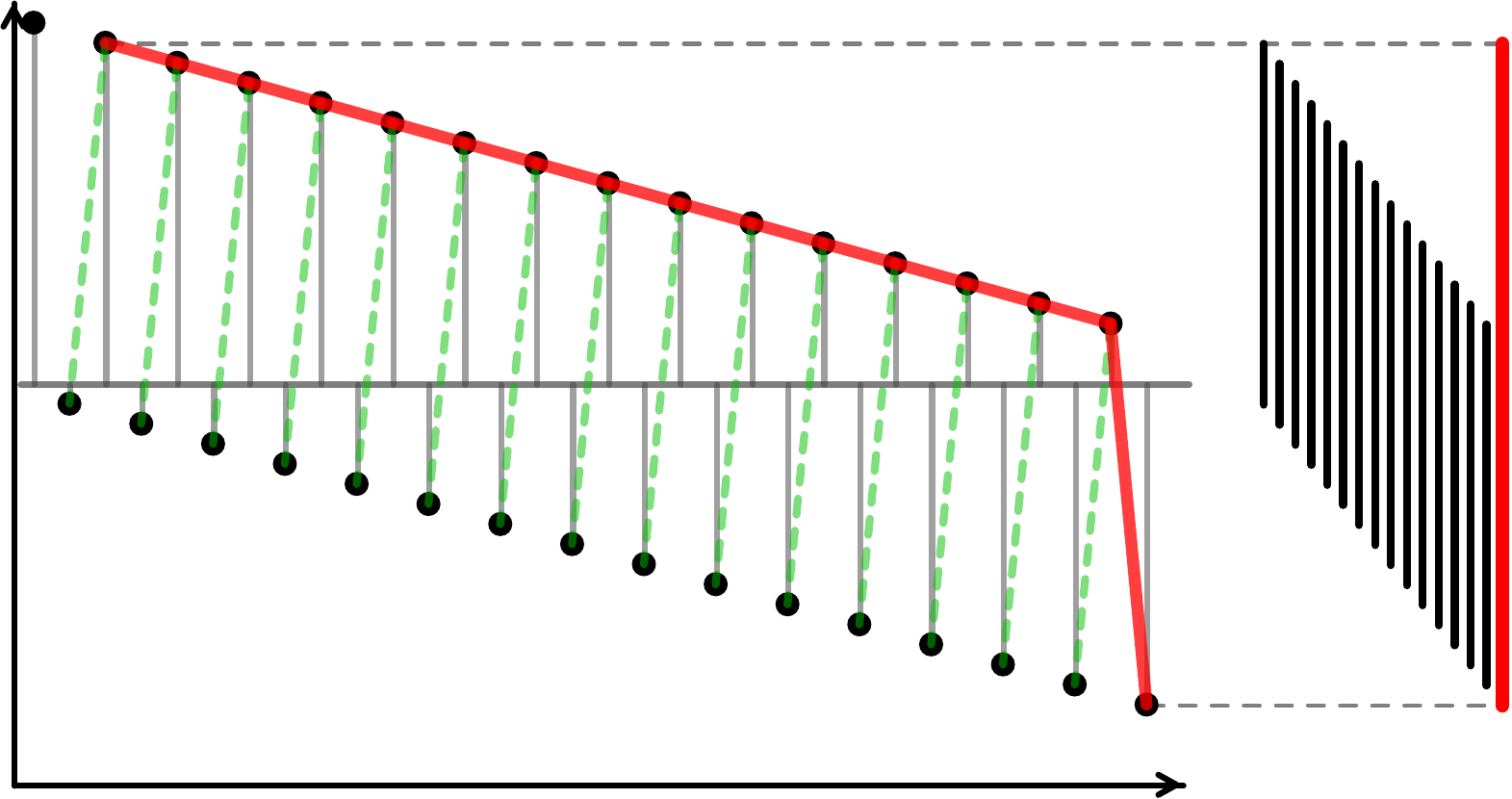}        \caption{}\label{subf:fullrangeelderA}
    \end{subfigure}
    \begin{subfigure}[t]{.49\textwidth}
    \centering
    \includegraphics[width=0.91\textwidth]{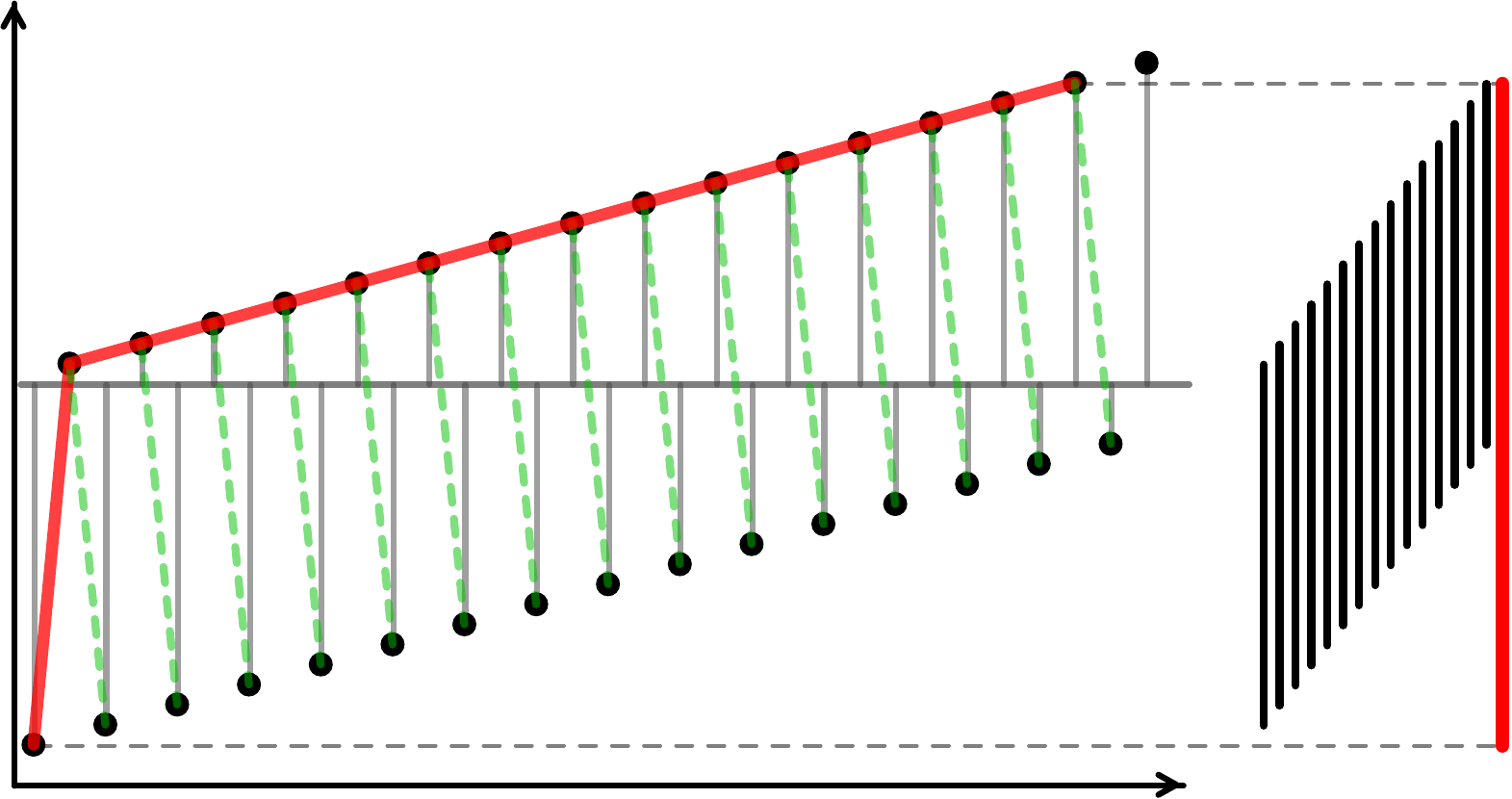}        \caption{}\label{subf:fullrangeelderB}
    \end{subfigure}
\caption{An essential bar that spans the full signal range. \subref{subf:fullrangeelderA} The maximum is at the left boundary while the minimum is at the right boundary, and \subref{subf:fullrangeelderB} the minimum is at the left boundary while the maximum is at the right boundary.}\label{fig:fullrangeelder}
\end{figure}  

The barcode construction rule to use depends on the need of the application domain. In certain settings, such as live updates of sensor data, data arrives in a streaming fashion, which can be incorporated via \emph{shifts}. A certain amount of data are leaving on one side of the dataset and new data are entering on the other. The process for constructing such shifts via surgeries will be discussed in Section \ref{sec:mshift}.

Using the same two rules (elder and local) we contrast their behavior under repeated \emph{$1$-shifts}. A $1$-shift changes the sequence at the boundary. On one side, one old sample is shifted out and discarded, while on the other side, one new sample is shifted in. Hence, some old \replaced{connection}{connectivity} information is lost and some new one is constructed. These constructions are localized to the ends of the domain.

If the elder rule is applied, barcodes can drastically change with each shift. In principle, barcodes can span the full length of the domain if the global minimum is on one end while the global maximum is on the other, as illustrated in Figure \ref{fig:fullrangeelder}. This effect is responsible for the worst case behavior in barcode construction algorithms \cite{baryshnikov2024time,ost2024banana}. Hence, in the presence of an elder rule, there is a potential need for global reconstruction of the barcode pattern at every update. In fact, any bar within the barcode can be subject to change at every 1-shift depending on the new data that arrives and the old data that is removed. Figure \ref{fig:eldercoherence} shows the results of the elder rule on shifted data. In this example, red traces the sample data minima and maxima involved in the elder barcode as the data shifts. Notice how the data first \marginnote{Review Q, Minor 13}\replaced{localizes the bar}{are localized} to the left, but then\added{, as new data shifts in, the bar} completely spans from the leftmost minimum to the rightmost maximum. This means that within two shifts the elder will have to reconfigure again (as the leftmost minimum will be shifted out). But this spanning phenomenon applies to all bars. Blue marks another barcode that is reconfigured by the addition of new data.

\begin{figure}[ht]
\includegraphics[width=\textwidth]{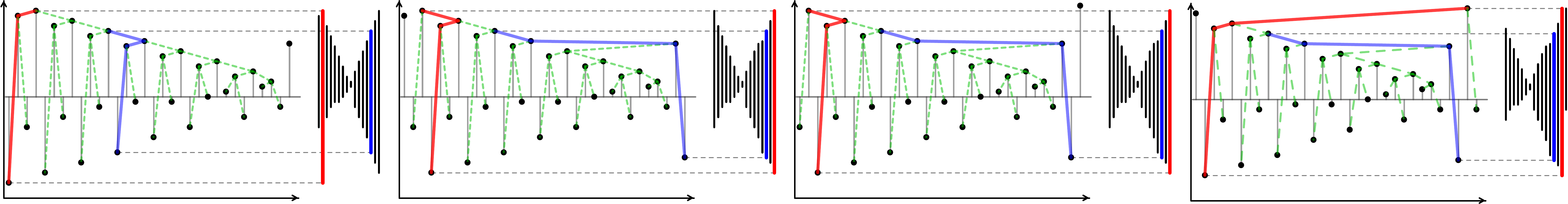}
\caption{Effects of 1-shifts using the elder rule with non-local elders. The sublevel set barcode is given with the essential bar highlighted in red and another non-essential bar in blue. In (outer left), we see the starting configuration. In (center left), we apply a 1-shift by moving the first point of the original sequence to the end. Notice that the new minimum means that the bar associated with blue changed. In (center right), we apply another 1-shift. Now there is a new local maximum at the end but it is at the boundary so does not yet contribute to any homological change. In (outer right), we apply one more 1-shift. The red essential bar is reconfigured and now spans from the leftmost minimum to the rightmost maximum.}\label{fig:eldercoherence}
\end{figure}

\begin{figure}[ht]
\includegraphics[width=\textwidth]{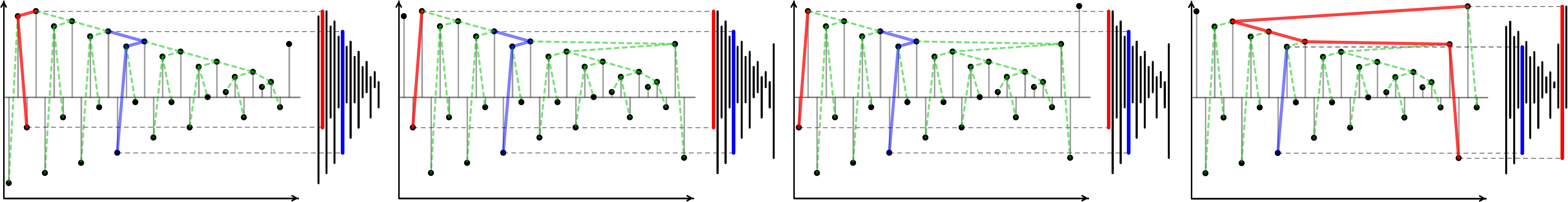}
\caption{Effects of 1-shifts using the local rule. The sublevel set barcode is given with the bar associated with the left neighbor of the global maximum is highlighted in red and another non-essential bar in blue (in the same position as in the elder example of Figure \ref{fig:eldercoherence}). In the first three 1-shifts only the left-most and right-most bars change and bars in between remain unchanged. However, when the global maximum moves to the right side, all barcodes change locally as a consequence of the behavior of the global maximum under the local neighbor barcode rule.}\label{fig:localcoherence}
\end{figure}

The consequence of the local barcode rule to a sequence of $1$-shifts is illustrated in Figure \ref{fig:localcoherence}. The barcode rule is based on local consideration, which means that the first two shifts do not change barcodes that are not associated with the changing data at the boundary. We see that despite new data arriving, the blue barcode remains unchanged. The red barcode does change, but only in response to local changes at the boundary. However, the behavior of the global maximum creates an exception. The last frame creates a non-boundary global maximum. This leads to all bars changing as they are now connecting to the right neighbor instead of the left. This can be seen with the blue barcode which now ends at a different maximum. Broadly, this leads to a rule of changes in barcodes based on new global maxima. If a bar changes sides to a global maximum, then it will change. Otherwise it will not.


We argue that this is a choice worth considering especially for streamed time series data where frame to frame comparisons are important and where a per-sample change of barcodes that is possible under the elder rule can potentially be harder to understand than the less frequent changes of the neighbor rule.






%% file: dual.tex
\section{Dualities and Symmetries}\label{sec:dual}

We have seen from Proposition \ref{prop:trisection} and Figure \ref{fig:inclusionduality} that there is a duality between sublevel sets and superlevel sets. The essence of Proposition \ref{prop:trisection} is that the disjoint union of sublevel, superlevel, and level sets always describes the total data, and that the flipping of the order keeps this union invariant with roles of sets reversed. Figure \ref{fig:inclusionduality} illustrates that this idea carries through to inclusions, and hence there is a dual structure on the level of filtrations as well. This is Theorem \ref{thm:strictsetduality}.

In this section, we explore numerous additional consequences of this duality. It turns out that the connectivity of the domain, specifically the presence of boundaries interacts with duality and can break symmetries, or create exceptions at the boundary. We are considering linear and circular domains (See Figure \ref{fig:sampledsine}). The linear domain has a boundary, while the circular domain does not. Unless stated otherwise, our claims of duality hold for circular domains and for interior points of linear domains. Yet, even when dualities that hold in the circular case are broken, the duality of Proposition \ref{prop:trisection} can still be used to recover dual properties that capture the nature of the original broken duality. Specifically, we will see that the duality under negation for circular domains that is well known \cite[Corollary 3.4]{biswas2023geometric} can be recovered when the duality of sublevel and superlevel sets are jointly considered (Figure \ref{fig:subsuperlevelduality}). Despite this, the effects of samples at the boundary can be quite subtle as we shall see. 

The main effect of the boundary that we have already seen is that local minima at the boundary do create a connected component, but local maxima do not join it with another (as there is no connected component beyond the boundary to merge with). Given that under change of order, minima and maxima swap roles between sublevel and superlevel sets\marginnote{Review Q: Minor 14}\replaced{. T}{, T}his asymmetry carries through but with switched roles. Former maxima at the boundary now create connected components\added{,} while former minima do not have a homological effect. Observe that this is still a duality of the asymmetry.



\subsection{Dualities between Sublevel and Superlevel Sets}
We summarize the dualities between sublevel and superlevel sets in Table~\ref{table:dualityR}. \marginnote{Review Y: Minor 12}\added{The entry in the left column are with respect to the original order of the data. The effects are discussed as the level is progressed in increasing order.}

\begin{table}[th]
    \begin{tabular}{r|c|c}
        Duality & Sublevel Set  & Superlevel Set\\
        \hline
        Minima & Create CCs & Split CCs \\
        Interior Maxima & Join CCs & Remove CCs \\
        Monotones & Grow CCs &  Shrink CCs \\
        Global Minima & Empty Set $\emptyset$ & Total Set $T$ \\
        Global Maxima & Total Set $T$ & Empty Set $\emptyset$ \\
    \end{tabular}
    \caption{Duality of sublevel and superlevel sets. CC is an abbreviation for {\em connected component}. 
    }\label{table:dualityR}
\end{table}

Connected components are created at levels of local minima for sublevel sets, whereas connected components are destroyed at levels of local minima for superlevel sets. Similarly, connected components merge together at levels of local maxima for sublevel sets, and disappear for superlevel sets. For this reason, minima and maxima provide dualities. Furthermore, what is local growth in length of a connected component over a monotone sequence is local shrinking in length of a connected component in the other. This gives the monotone duality. At levels of global minima (or lower), the sublevel set is the empty set whereas the superlevel set is the total set. The opposite is true for levels of global maxima (or higher). From this, we get the duality of global minima and maxima. 

\subsection{Dualities between Sublevel and Superlevel Set Persistence}

Next we consider sublevel set persistence to superlevel set persistence. 
A duality one might suspect is that a birth in the sublevel set barcode corresponds to a death in the superlevel set barcode. This is \emph{almost} true, but the samples at the boundary can break this duality as mentioned previously. 
If the local maximum is at a boundary, then it is not merging two connected components. The level of that local maximum might not be a death in the barcode. See Figure \ref{fig:f-versus-negative-f}.

The key dualities regarding birth and deaths of connected components between sublevel and superlevel set filtrations are in Table \ref{table:dualityF}.

\begin{table}[th]
    \begin{tabular}{c|c|c}
        Connected Component & Sublevel Set  & Superlevel Set\\
        \hline
        Creation & Level of Minima & Level of Interior Maxima \\
        Merging & Level of Interior Maxima & Level of Minima \\
    \end{tabular}
    \caption{Duality of sublevel and superlevel set filtrations with respect to changes in connected components.}\label{table:dualityF}
\end{table}

Dualities between sublevel and superlevel set persistence has been previously studied by taking an alternative perspective. Namely, instead of considering dualities between sublevel sets and superlevel sets, one can focus solely on sublevel sets (or superlevel sets) and consider the sublevel sets of the function and negative values of the function \cite{cohen2009extending,cultrera2024dynamically}. Since minima of a sequence function $f$ become maxima of $-f$, the barcodes of $f$ and $-f$ provide the same basic information as comparing the barcodes of the sublevel sets and superlevel sets of $f$.

The symmetries studied between barcodes of $f$ and $-f$ relies on a classic result in algebraic topology called \emph{Lefschetz duality} which states that the $p$-cohomology of a $d$-manifold is isomorphic to the $d-p$ relative homology of the manifold. This leads to a symmetry theorem from \cite{cohen2009extending,edelsbrunner2010computational} for persistence barcodes. Let $R$ denote a transformation that maps an interval $(a,b)$ to $(-b, -a)$. Let $B_p$ denote the $p^{th}$-dimensional barcode. The symmetry theorem states that for a continuous function $f$ on a $d$-manifold without boundary, the persistence barcodes of $f$ and $-f$ are reflections of each other as follows:
$$B_p(f)=B^{R}_{d-p-1}(-f)$$
where the barcodes are constructed from sublevel set filtrations. 
If $p=0$ and $d=1$, then this result says $B_0(f)=B_0^R(-f)$. 

The Symmetry Theorem does not directly apply to our setting since we are not working with continuous functions on manifolds without boundary. In particular, if our finite discrete function is on a circular domain, then a birth in the sublevel set barcode corresponds to a death in the superlevel set barcode, and vice-versa. Furthermore, observe that negation of a function corresponds to inversion of order of samples in our setting.
    
\begin{proposition}[Symmetry of Sublevel and Superlevel Barcodes on a Circular Domain]
\label{prop:discrete-symmetry-thm-circle}
   Let $X:=x[0, \dots, N-1]$ be a sequence on a circular domain. If $(b,d)$ is in the sublevel set barcode, then $(d,b)$ is in \replaced{a}{the} superlevel set barcode.
\end{proposition}

\begin{proof}
\marginnote{Review Y: Minor 13}
Since we are in the circular domain case, we know that every local extremum corresponds to a birth or death event in \replaced{a}{the} sublevel and superlevel set barcode. Let $(b,d)$ be in \replaced{a}{the} sublevel set barcode. By Propositions \ref{prop:min} and \ref{prop:max}, we know that $b$ is the level of a local minimum and $d$ is the level of a local maximum. Furthermore, there is a monotonically increasing subsequence between levels $b$ and $d$. Next, if we consider \replaced{a}{the} superlevel set barcode, the local maximum at level $d$ is now the level of a birth of a connected component. The monotonically increasing subsequence between levels $b$ and $d$ now indicates the growth of this connected component until it merges with another at level $b$. Hence, we find that $(d,b)$ is \replaced{in the}{in a} superlevel set barcode. 
\end{proof}

\added{
\begin{remark}
Notice that this is a duality on bars. Meaning that it does not dualize the barcode construction rule necessarily. However, the dual barcode will be unique and fixed by the chosen barcode construction rule for the original barcode that is being dualized. In some cases, it is easy to see how the barcode constructions rule behaves under this duality. We give two examples below. An alternative way to see this is to dualize the merge tree and observe that the sublevel and superlevel barcodes are dual paths in the merge tree. Hence, a fixed barcode rule on the sublevel set will induce a fixed barcode rule on the superlevel set.
\end{remark}
}

\added{
\begin{example}[Local Barcode Rule]
Assume that we use a local barcode rule where we connect a minimum with the maximum to the right of it. Then if we look at the same extrema from the superlevel order, we see that we can keep the same bars when we connect the same information, which from this order appears to be to the left. If we pick the neighbor on the other side, then this relationship simply flips.
\end{example}
}
\added{
\begin{example}[Elder Rule with Tribal Council]
Assume that we use the elder rule with tribal council to construct our bars. Assume some tribal council rule to pick a survivor at local maxima, say the left-most elder. Then each bar in the sublevel set has a corresponding bar in the superlevel set, but the tribal council decisions are performed in the inverse order as originally chosen. That is, from the superlevel perspective, the right-most elder would correspond to the bar pair.
\end{example}
}

As mentioned earlier, the symmetry can be broken when a boundary is introduced. For example, in Figure \ref{fig:f-versus-negative-f}, we see there is a bar in the sublevel set barcode that is not present in the superlevel set barcode. This is because the number of minima is not equal to the number of maxima so the number of births differ in the two barcodes.

\begin{figure}[htp]
    \centering
    \begin{subfigure}[b]{0.48\textwidth}
        \centering
        \includegraphics[width=1\textwidth]{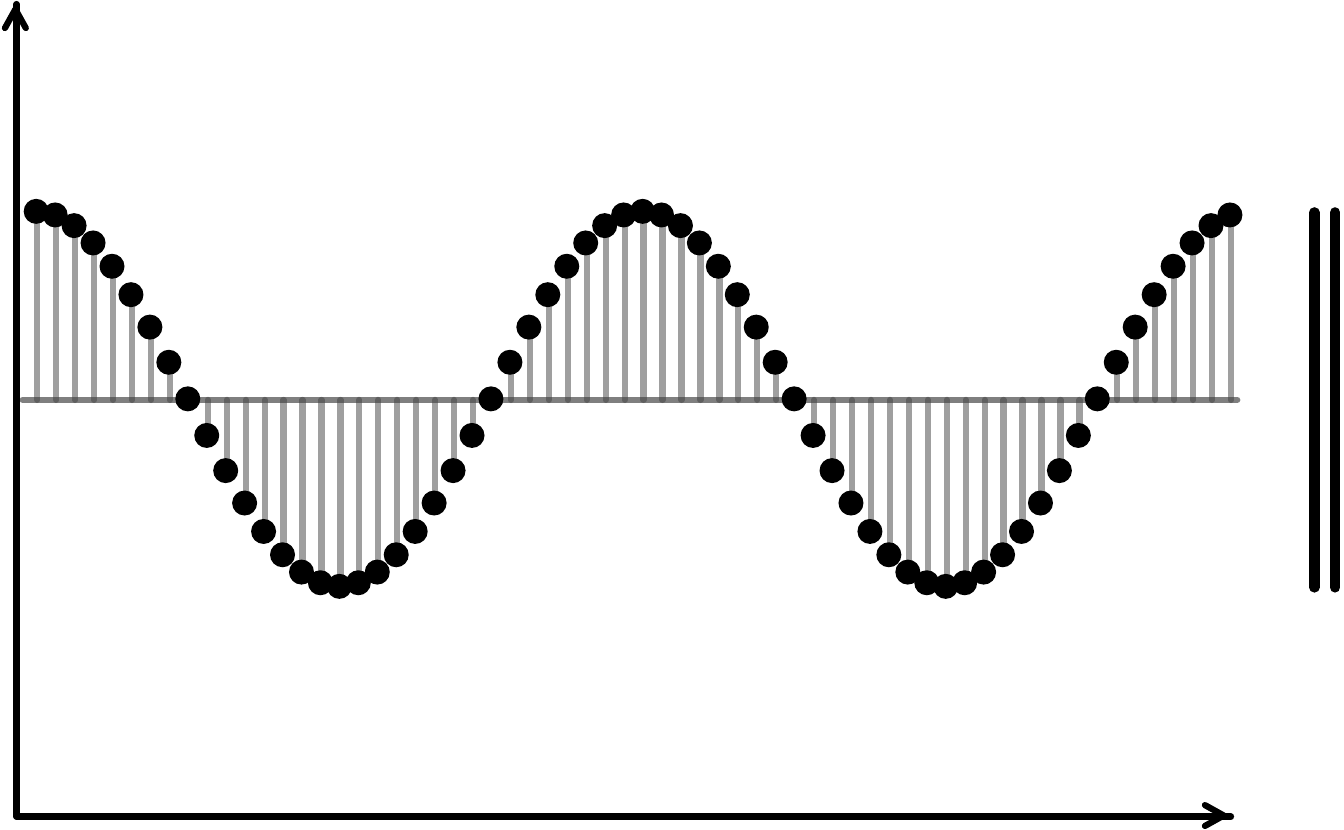}
        \caption{$X$ with given order of levels}
        \label{fig:cosine}
    \end{subfigure}
    \hfill
    \begin{subfigure}[b]{0.48\textwidth}
        \centering
        \includegraphics[width=1\textwidth]{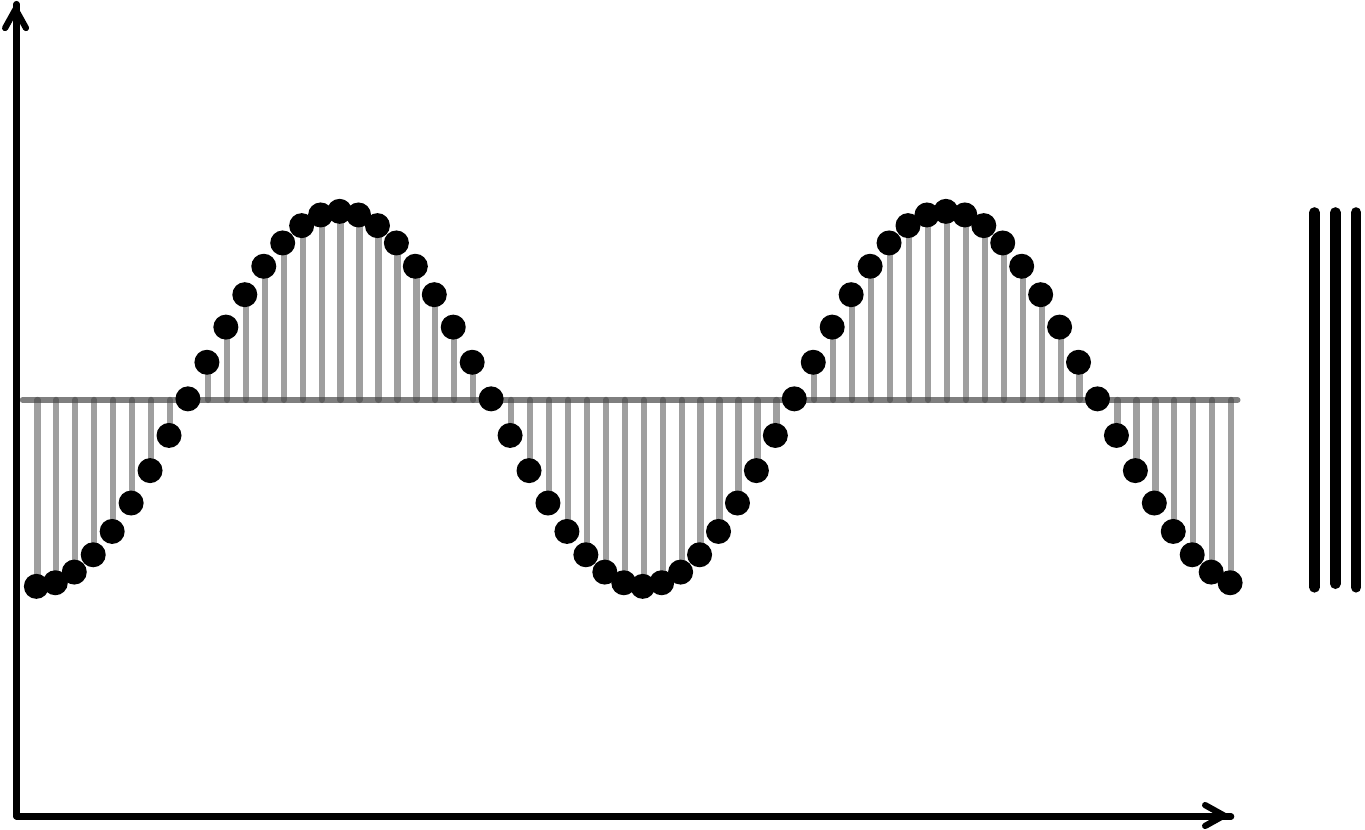}
        \caption{$X$ with inverted order of levels}
        \label{fig:negative-cosine}
    \end{subfigure}
    \caption{Sequence $X$ under a given order $X,<$  shown in \subref{fig:cosine} and its inverted order $X,>$ shown in \subref{fig:negative-cosine} along with their barcodes. Figure \subref{fig:cosine} has two bars in $B(X,<)$ as there are two minima. Since $X$ has three maxima, we see in case Figure \subref{fig:negative-cosine} that $X$ with inverted order has three minima which gives three bars in $B(X,>)$.}
    \label{fig:f-versus-negative-f}
\end{figure}

Additionally, we saw in Figure \ref{fig:one-min-one-max}, that it is possible to have the same number of minima and maxima but different sublevel and superlevel barcodes. However, if the number of minima equals the number of maxima, and every minima corresponds to the birth of a connected component and every maximum corresponds to a death of a connected component, then we get the following symmetry theorem. Note that this is \emph{not} saying that the pairings are dual. It is possible to have $(b,d)$ in the sublevel set barcode and $(d,b)$ \emph{not} in the superlevel set barcode due to the elder convention.

\begin{figure}[htp]
    \centering
    \begin{subfigure}[b]{0.48\textwidth}
        \centering
        \includegraphics[width=1\textwidth]{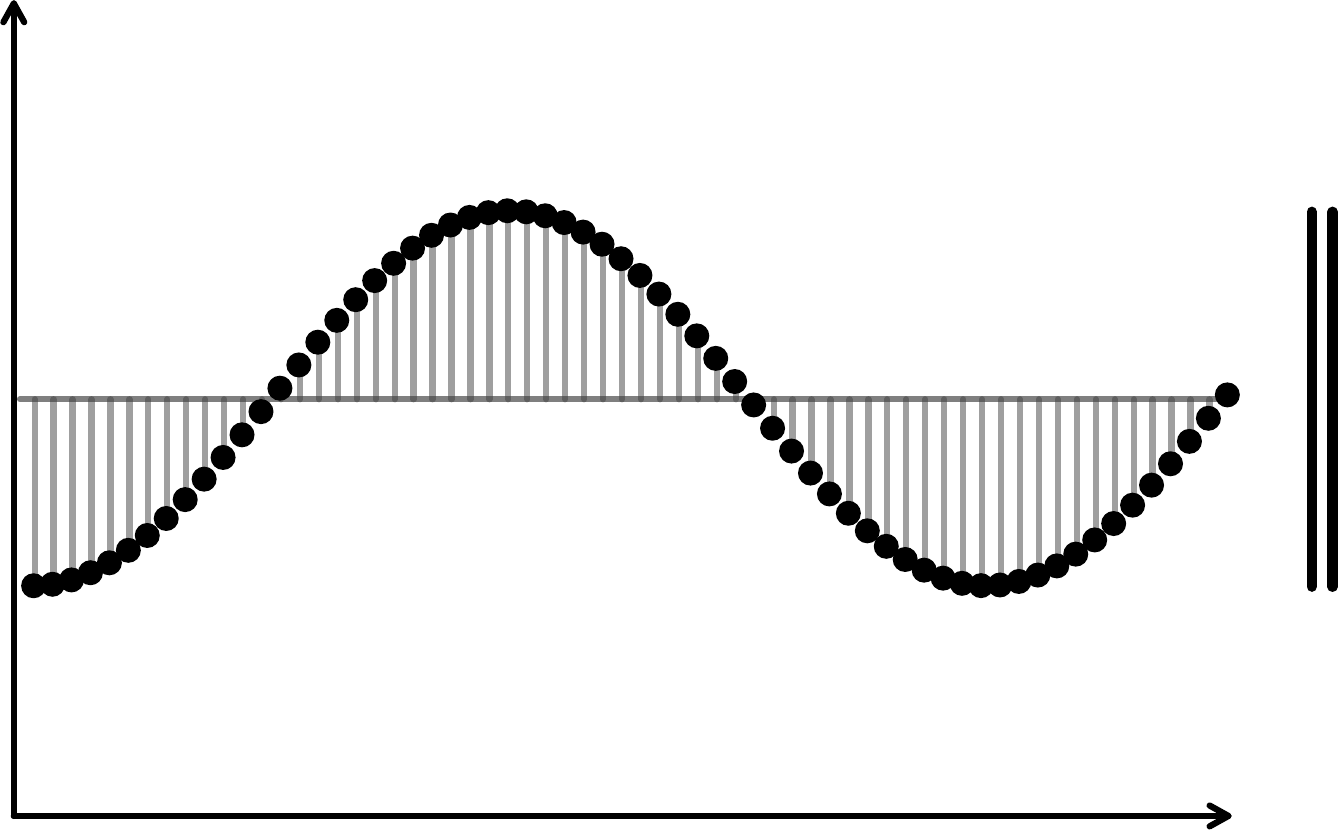}
        \caption{$X$ with a given order of levels}
        \label{fig:f_scrap}
    \end{subfigure}
    \hfill
    \begin{subfigure}[b]{0.48\textwidth}
        \centering
        \includegraphics[width=1\textwidth]{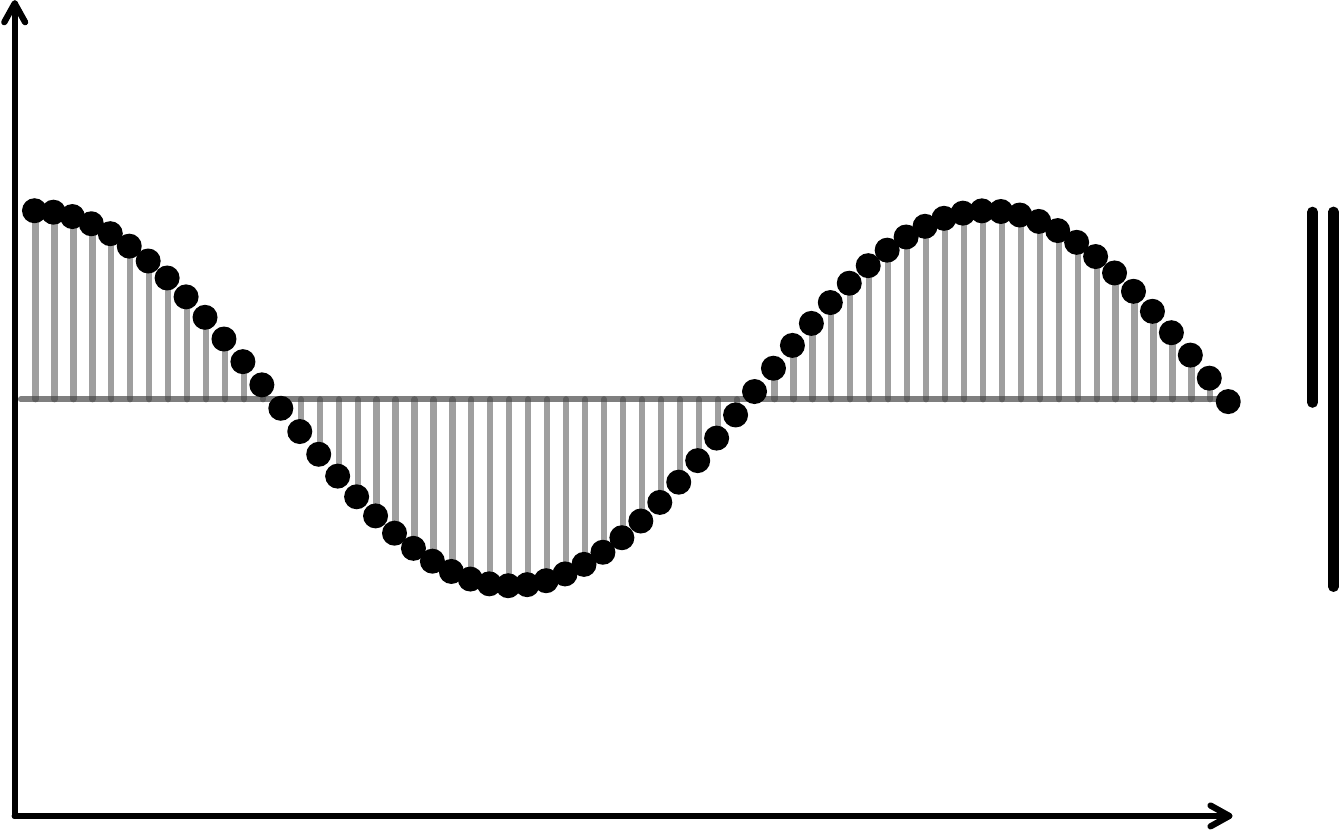}
        \caption{$X$ with inverted order of levels}
        \label{fig:negative-f}
    \end{subfigure}
    \caption{Same number of minima and maxima but asymmetric barcodes in $B(X,<)$ and $B(X,>)$. Although the number of minima equals the number of maxima in $X,<$ and $X,>$, we see that the barcodes differ. In \subref{fig:f_scrap}, two connected components are born at the global minima, and one dies at the level of $\max(X)$. In \subref{fig:negative-f}, two connected components are born at different levels, and one dies at the level of $\max(X)$.}
    \label{fig:one-min-one-max}
\end{figure}

\begin{proposition}[Symmetry of Sublevel and Superlevel Barcodes on a Linear Domain]\label{prop:discrete-symmetry-thm}
Let $X:= x[0]\dots x[N-1]$ be a sequence. Suppose that only one of $x[0]$ or $x[N-1]$ is a local minimum. Additionally, suppose that if $X$ has more than two extrema, then $x[0]$ and $x[N-1]$ are levels of another extremum of the same type.  Then, if $b$ is a birth in the sublevel set barcode of $X$, then $b$ is a death in the superlevel set barcode of $X$. 
\end{proposition}

\begin{figure}[th]
\centering
\includegraphics[width=0.99\textwidth]{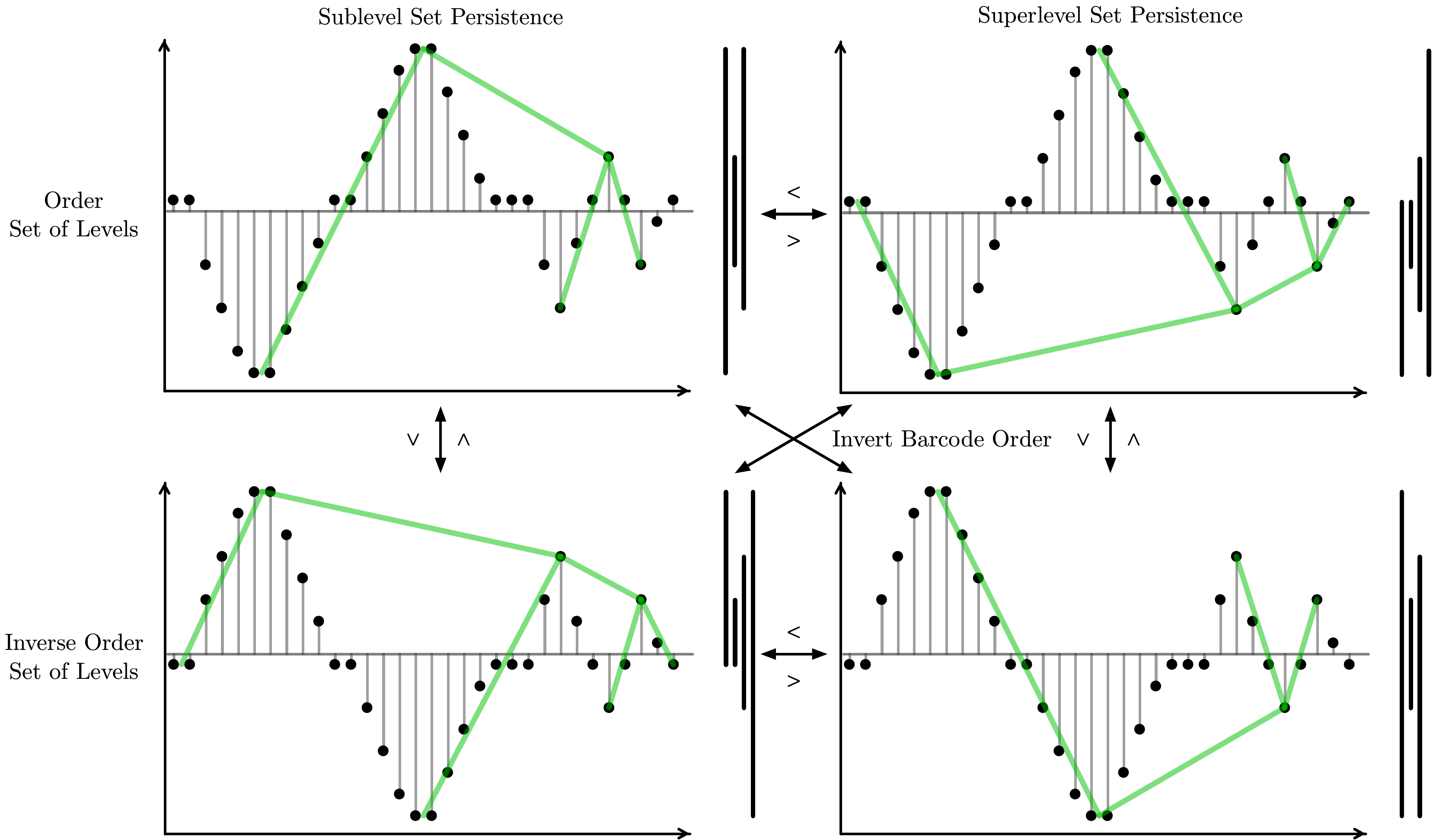}
\caption{Duality of sublevel and superlevel sets and the duality of orders of sets of levels. Notice that the diagonals contain the same information up to inversion, including merge trees (green) and barcodes.}\label{fig:subsuperlevelduality}   
\end{figure}

However, when we combine duality of sublevel and superlevel sets\deleted{,} and inversion of order of levels (``negation")\marginnote{Review Q: Minor 15}\added{,} we recover the following interesting duality shown in Figure \ref{fig:subsuperlevelduality}. The rows of the graph show the same function. The top row shows the direct function while the bottom row shows the same function with the order of levels inverted. This can be thought of as negating a function since negation implies inversion of order. The columns represent which direction the level sets are included to compute persistence. The left column is sublevel set persistence, while the right column is superlevel set persistence. Inversion of order of sets of levels turns local minima into local maxima (and vice versa), which play dual roles in sublevel set persistence. The behavior in Figure \ref{fig:subsuperlevelduality} is captured by the following theorem:

\begin{theorem}[Relationship of Sublevel and Superlevel Set persistence by Order Inversion]
Given a fixed order of levels, the superlevel set of the inverted order of levels is identical to the sublevel set of the original order, and the sublevel set of the inverted order is identical to the superlevel set of the original order.\label{thm:subsuperdual}
\end{theorem}

\begin{proof}
By Theorem \ref{thm:strictsetduality} the sublevel set filtration has a corresponding dual superlevel set filtration. Using the notation in Figure \ref{fig:inclusionduality}, consider a fixed filtration order, $L_0,\ldots,L_s$ and also its inverted order. This inverts the set of levels in the top inclusion sequence as shown in the figure as follows\marginnote{Review Q: Minor 16}\added{:} $L_s,\ldots,L_0$. Notice that this sequence of inclusions is identical to the data of the original sequence as given by the superlevel direction. Hence they are identical up to orientation of the filtration (left for sublevel, right for superlevel).
\end{proof}

Contrary to direct duality within sublevel sets for circular domain this theorem holds for both linear and circular domains. Hence one need not consider effects of the boundary.

\begin{corollary}[Barcodes of Sublevel and Superlevel Set Persistence by Order Inversion]
The bars in a barcode computed from sublevel set persistence on a finite ordered dataset are the same as the bars in a barcode computed from superlevel set persistence with the order of the bars inverted.     
\end{corollary}

\begin{corollary}[Merge Trees of Sublevel and Superlevel Set Persistence by Order Inversion]
Merge trees computed for sublevel set persistence from a finite ordered dataset are the same as merge trees computed from superlevel set persistence with the order of the merge tree inverted. 
\end{corollary}

Both corollaries are immediate as Theorem \ref{thm:subsuperdual} operates on the level of filtrations. Observe that it is sufficient to implement only one direction of the filtration (sublevel or superlevel) and the omitted case can be reconstructed by inverting the data levels, applying a functor on the filtration (such as homology) to compute persistence information such as a barcode or a merge tree, and then \marginnote{Review Q: Minor 17}invert\added{ing} the order of the result to arrive at the equivalent outcome. This effect on both barcodes and merge trees can be seen in Figure \ref{fig:subsuperlevelduality}. \marginnote{Review Y: Minor 14}\added{A duality of merge trees under sign inversion has been asserted without proof before \cite{carr2003computing} and it is known that there is a sign inversion duality in extended persistence known as the Persistence Symmetry Theorem \cite{edelsbrunner2010computational}. Its proof \cite{cohen2009extending} relies on extended homology, which is not necessary in our case. Notice that our result is more general than negation as it applies to any order-preserving inversion, of which the negative is a special case. For example, assume the usual metric structure of $\mathbb{R}$. A map from $f:x\mapsto e^{-1}$ also inverts the order and is order-preserving. Hence, it is a valid example. However, in general, there is no assumption of any metric or group structure in our proof.}\marginnote{Review Y: Minor 15} 

\marginnote{Review Q: Major 2, part 1}
\deleted{The last duality we mention is the dual to homology -- cohomology. Cohomology groups are similar to homology groups but less geometric and are motivated by algebra. For instance, cohomology has a ring structure that homology does not. To compute homology groups, we form a chain complex that typically consists of maps between free abelian groups. To compute cohomology groups, we form a cochain complex that consists of maps between groups of homomorphisms. More details on this can be found in Chapter V of \cite{edelsbrunner2010computational}. A key observation is that the ranks of the homology groups are exactly the same as the ranks of the cohomology groups. Additionally, homology and cohomology are dual as vector spaces and so they have the same dimension. This follows from the Universal Coefficient Theorem (Theorem 3.2 of \cite{HatcherAlgebraic02}). Furthermore, since barcodes are determined by dimensions and ranks, we have that the persistent homology and cohomology barcodes are the same \cite{de2011dualities}.}\marginnote{Review Q: Minor 18}  

\section{Order Preservation and Functions}

All our results only require order preser\-vation of the data. Numerous results stem directly from this requirement. As is known for circular domains, Lefschetz duality leads to a duality of barcodes when negating a function. In our setting, this corresponds to an inversion of order. However, given that any transformation that does not disturb the order does not change barcodes in our setting, one can derive more general results. For details on how order data can be related to data presumed to be in $\mathbb{R}$ see Appendix \ref{app:quantsamp}.

Consider the example of a constant global vertical offset of data. If we have a function $f$, we consider $f-v$, where \replaced{$v$}{$h$}\marginnote{Review Q: Minor 19} is the vertical offset. That is, applying a transformation $f(x) \mapsto f(x)-v$ implies that \replaced{each}{a}\marginnote{Review Q: Minor 20} bar $(b,d)\mapsto (b-v,d-v)$, but note that the height of the barcodes $d-b$ remains unchanged. Applying a vertical shift changes the heights of the minima and maxima, but not the differences in heights. Multiplying a function by $-1$, reflects the function over the $x$-axis so that maxima become minima and vice-versa. Again, the heights between extrema remain the same. Using the result from \ref{prop:discrete-symmetry-thm-circle}, we get that on a circular domain, applying the transformation $f\mapsto -f$ implies that $(b,d)\mapsto (-d,-b)$. Combining all these observations together we get the following:

\begin{align}
f\mapsto f(x)+v &\implies (b,d)\mapsto (b+v,d+v) \label{eqn:circular-transformation1} \\
f\mapsto -(f(x)+v) &\implies (b,d) \mapsto (-(d+v), -(b+v)) \label{eqn:circular-transformation2}
\end{align}

Additionally, horizontal shifts of the samples on a circular domain do not change the barcode.

We emphasize that the result in Equations \ref{eqn:circular-transformation1} and \ref{eqn:circular-transformation2} assumes a circular domain. More general results follow from order-preserving modifications of the dataset that are not constant, but discussion of this case will be omitted here.

\subsection{Using Dualities in Computation}

Utilizing symmetries and dualities can be advantageous in applications where summarizing data uses extrema. One specific example is with \marginnote{Review Q: Minor 21}\replaced{gene expression time series data}{gene expression data} that consists of measurements \deleted{of gene expression} at discrete time points. Capturing the temporal ordering of extremal events provides a coarse summary of the time series experiment and has been used for understanding biological systems like circadian rhythms and cell cycles \cite{BerryUsing20, CumminsModel18, SmithAn20}. The techniques used in \cite{BerryUsing20,belton2023extremal} represent time series data with directed graphs where the number of vertices is equal to the number of local minima and maxima. Furthermore, there are vertex weights or noise levels associated to these directed graphs that are related to values that come from sublevel set persistence. Understanding which symmetry properties arise from the data or gluing the time series to create a circular domain, could simplify these descriptors so that only the local minima (or local maxima) need to be considered.

\subsection{Comparison to Extended Persistence}

Extended Persistent Homology \cite{cohen2009extending}, using a Morse setting, extends persistence data with height functions defined over $\mathbb{R}$ by combining both the sublevel set and the superlevel set data. Extended homology differentiates between ordinary, relative and essential barcodes. Ordinary barcodes correspond (roughly) to sublevel set barcodes in our setting and relative barcodes correspond to superlevel set barcodes. In particular, the {\em essential} bar is dualized to connect the global minimum in the sublevel set with its dual global maximum in the superlevel set. Hence essential barcodes, which in classical persistence over $\mathbb{R}$ are infinitely extended then become finite, a property that helps among other things create finite distance metrics on persistence modules \cite{turner2024extended,bauer2024universal}.  

It is unclear if there is a principled way to translate the notion of an essential bar into our setting because there is no unique definition of the essential property without Morse-restrictions. We have seen this problem already in the discussion of barcode construction rules, requiring arbitrary tie-breakers (tribal councils) even in the case of the elder rule being used (see Figure \ref{fig:mergetreeelderlocal} and its discussion in Section \ref{sec:barcode}). Ad hoc tie-breaking rules or perturbation strategies \cite{cultrera2024dynamically} appear to cover the ambiguity we have discussed but do not contain information of the data to be analyzed. 

However, Propositions \ref{prop:trisection} and \ref{thm:subsuperdual} tell us that we can recover superlevel set filtrations from sublevel set filtrations.

\marginnote{Review Q: Minor 22}\deleted{Given that we work with finite level data, the finiteness condition that is appealing in the context of working with $\mathbb{R}$ which is not compact, is solved by construction. That is, finite bars are achieved by finiteness of the sets of levels.}

\added{The presence of infinite essential bars in sublevel and superlevel set filtrations are solved in our setting by construction. The global maxima and minima are associated with the empty and total set. Hence, one of the reasons that extended persistent homology is employed, which is to make all bars finite length, does not appear in this setting. It suggests that the discrepancy is a problem of the minimum and maximum being not treated dually. In our setting, the empty set dualizes to the total set. Hence, global minima and global maxima play a dual role.}

Furthermore, we have seen that the only data differing in merge tree and barcode representations between sublevel sets and superlevel sets in our setting are maxima at the boundary.

Hence, we get a pragmatic way to extend barcodes in our setting. We can always compute the dual by supplementing the boundary maxima. By construction,  all other extrema are encoded in the merge tree. Hence, to compute the dual merge tree, one takes the supplemental information of the level of the maxima, adds them to the set of dual minima as leafs of the dual merge tree and constructs the nodes from the dual maxima levels, excluding dual maxima that are in the boundary. If the barcode rule is local, then barcodes of the boundary maxima can simply be added to the dual locally, connecting to the local unique neighboring dual maximum, and boundary minima can be removed as they are local. Observe this behavior in Figure \ref{fig:cosine}, where the unchanged barcode is invariant and between the boundaries. Additionally, the change in bars correspond to the removal of one and the addition of the other. One could consider this to be a kind of {\em extended persistent homology without essentials}.

It should be mentioned that there is a known connection between extended persistence and level set zigzag persistent homology \cite{carlsson2009zigzag}, an approach that tracks homological changes by inclusions in opposing directions. While translation to a zigzag-type pattern is possible via \ref{fig:inclusionduality}, the correspondence is not straightforward. Original levelset zigzag homology assumes Morse functions and uses this property to separate ``regular" levels (intervals not containing critical points) from critical ones \cite{carlsson2009topology,dey2022computational}. No such distinction is made in our work. Each level corresponds to a map that is potentially multi-valued with respect to connected components and we resolve this by a subfiltration pattern (see Section \ref{subsec:withinlevelfiltration}.)

%% file: snakeboxes.tex
\subsection{Box Snakes}\label{sec:snakeboxes}

{\em Box Snakes} were introduced in \cite{essl2024deform} as a data structure that captures extrema as well as monotone sequences in a finite sequence. It resembles some of the properties of another recent data structure called windows \cite{biswas2023geometric}\footnote{This notion of window is not to be confused with the widely used concept of window in digital signal processing \cite{harris1978use}.}.

Arnold \cite{arnold1992calculus} coined the term {\em snake} to describe a discrete pattern of numbers that alternate in total order: $y[0]< y[1]> y[2]< \cdots> y[n-1]< y[n]$. This structure illustrates that minima and maxima necessarily alternate. A snake does not necessarily have to start with a minumum, it can also start with a maximum.

\begin{figure}[th]
\includegraphics[width=\columnwidth]{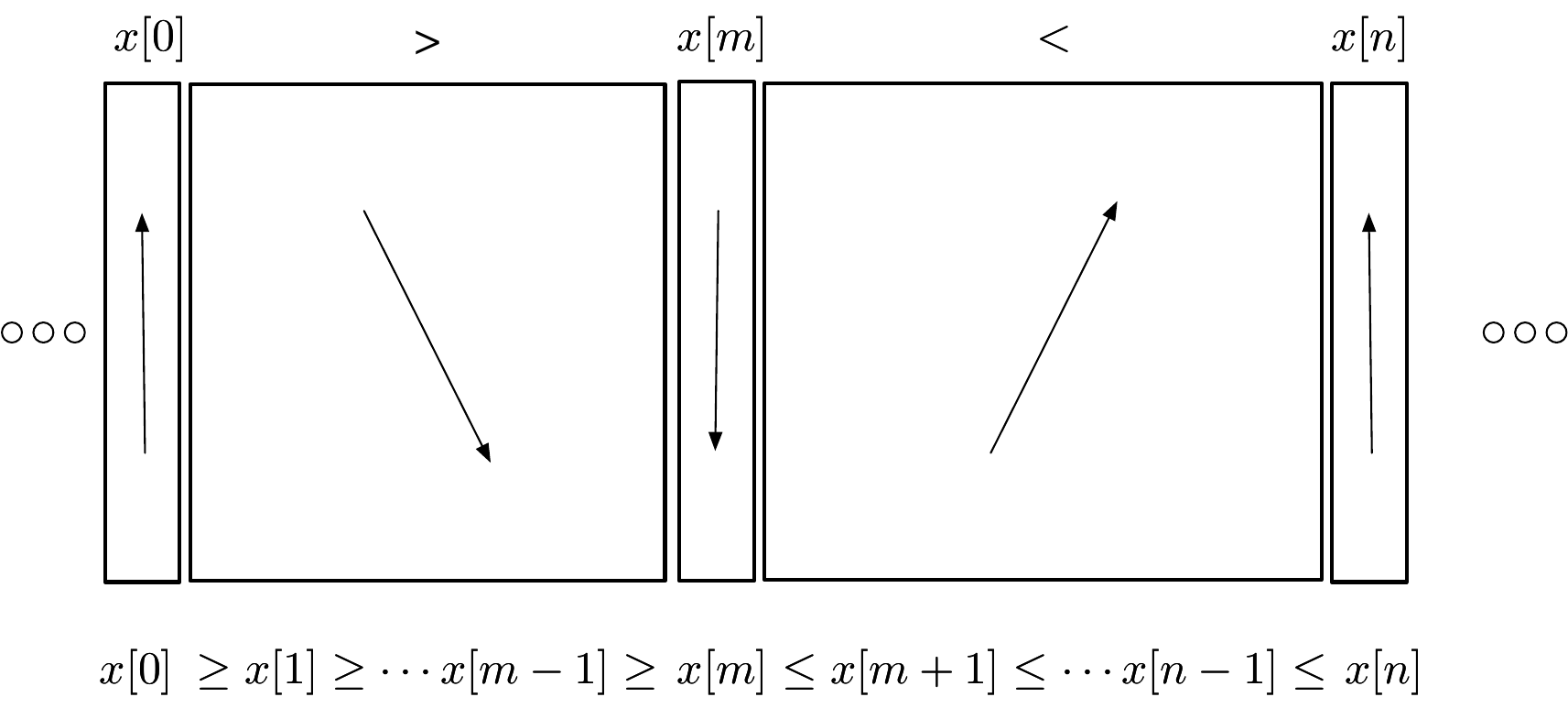}
\caption{The box snake structure of alternating extrema with monotone sequences between them (from \cite{essl2024deform}).}\label{fig:boxsnake}    
\end{figure}

A \emph{box snake} is a snake that allows additional samples between extrema that are required to be monotone. It has the form $x[0]< x[1]\leq \ldots \leq x[n-1]< x[n]> x[n+1]\geq \ldots \geq x[m-1] > x[m]$. The specific sequence of order operators is again arbitrary. The structure of a box snake is depicted in Figure \ref{fig:boxsnake}, where monotone segments can be omitted. If there are no monotones, then box snakes reduce to snakes. We consider snakes and box snakes both on linear and circular domains. Hence\added{,} we do not require the presence of boundaries. We call individual subdivisions in the box snake structure either a \emph{snake box} or simply a \emph{box}, and they can contain either an extremum or an individual monotone sequence.

Given that homological information only changes at extrema, box snakes contain the relevant information. In the original application, monotone boxes capture regions that can be monotonically edited without modifying the sublevel set persistent homology \cite{essl2024deform}. A {\em monotonical edit} is a change to a sample level that does not violate monotonicity of the sequence within an individual box. Hence\added{,} the subdivision of \marginnote{Review Q: Minor 23}\deleted{a} finite sequences into snake boxes contains both the information relevant for level set persistence, as well as a compact representation of monotone editability that keeps the homological information and their representations, such as the barcode, invariant.

Numerous properties of snake boxes are immediate. For extrema, there is no allowable variation in level, as this would violate invariance of barcodes. This makes snake boxes of extrema appear as lines in our graphical representations. For monotones, the levels allowable are defined by the extremes that neighbor the monotone.

Figure \ref{fig:boxsnakeexample} shows the snake boxes computed by the box snake algorithm \cite{essl2024deform}. The vertical boundaries of the boxes are given by the levels of adjacent extrema, if they contain monotones. Otherwise, they contain the level of the extremum.

\begin{figure}[ht]
    \begin{subfigure}[t]{.49\textwidth}
    \centering
    \includegraphics[width=0.91\textwidth]{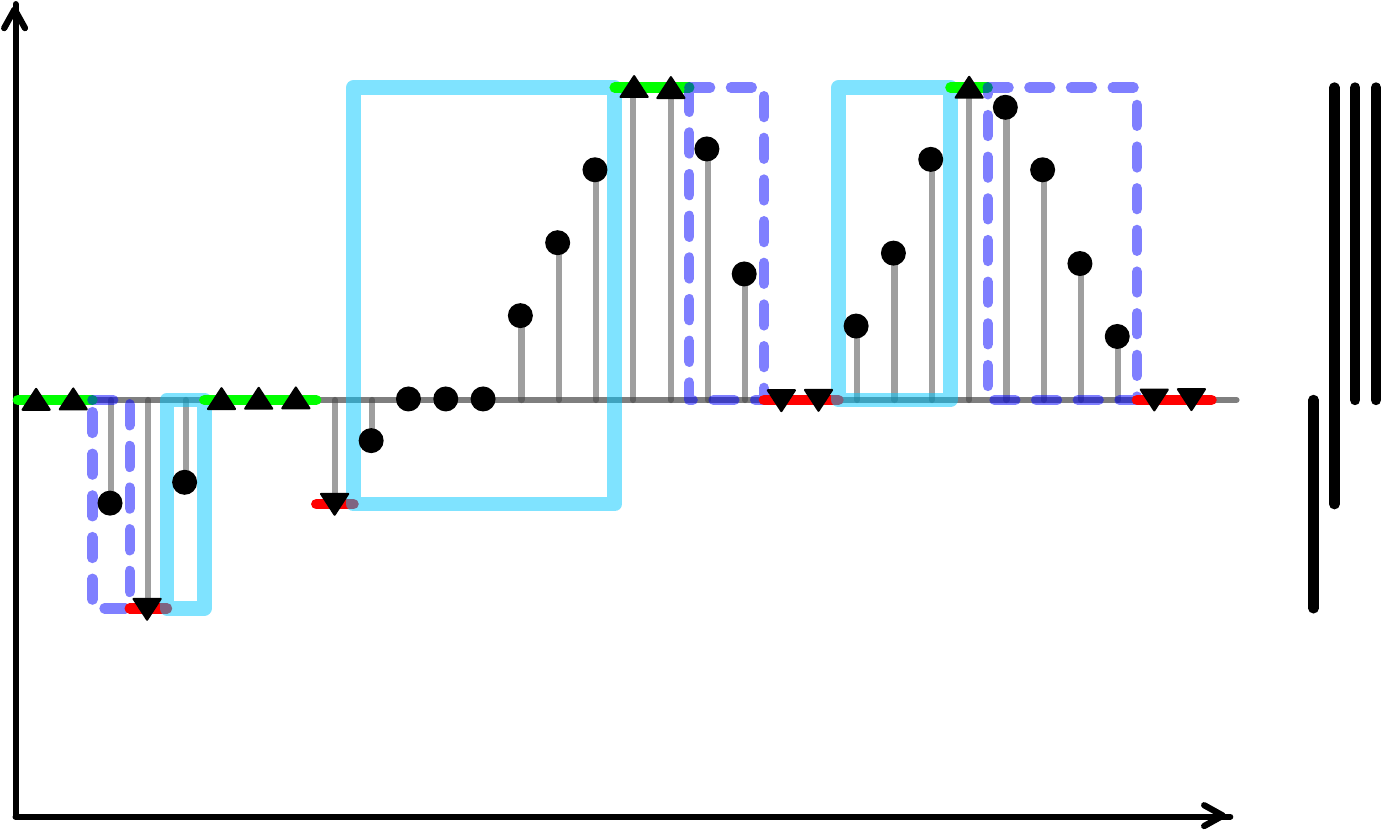}        \caption{}\label{subf:boxsnakeexampleA}
    \end{subfigure}
    \begin{subfigure}[t]{.49\textwidth}
    \centering
    \includegraphics[width=0.91\textwidth]{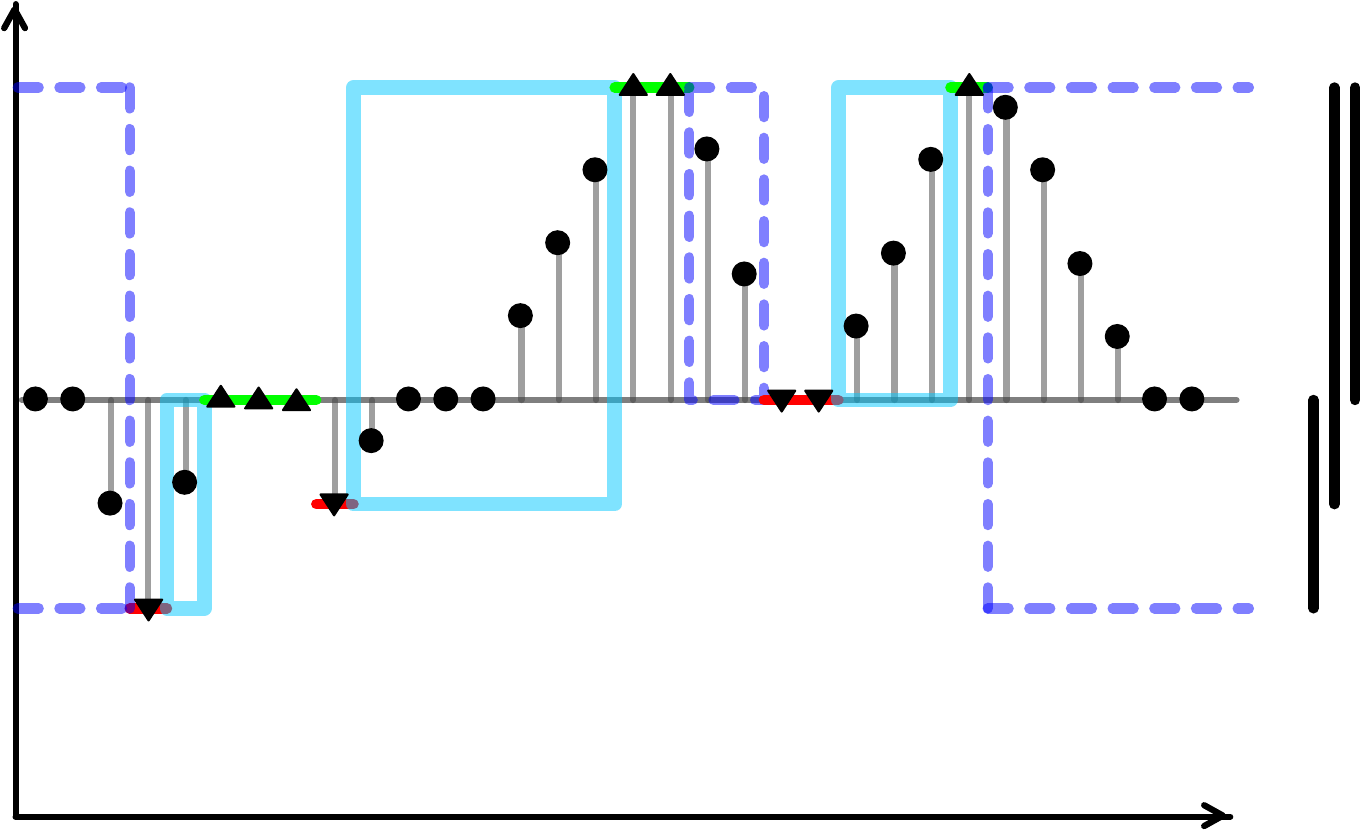}        \caption{}\label{subf:boxsnakeexampleB}
    \end{subfigure}
\caption{Box snakes computed with identical sample data for a linear domain \subref{subf:boxsnakeexampleA} and a circular domain \subref{subf:boxsnakeexampleB}. Ascending monotones (cyan \added{and} solid) are distinguished from descending monotones (blue \added{and} dashed), minima (red and $\blacktriangledown$) and maxima (green and $\blacktriangle$). Observe that monotone deformations inside monotone rectangles do not alter the barcodes. }\label{fig:boxsnakeexample}    
\end{figure}

%% file: surgery.tex
\section{Surgery}\label{sec:surgery}
By {\em surgery}, we consider the process of removing or restoring the adjacency relationship between two neighboring points in the domain. An adjacency relationship consists of one successor relationship from sample to its neighbor $x[i+1]\succ x[i]$ together with the predecessor relationship between the same samples $x[i]\prec x[i+1]$. We notate an adjacency relationship $x[i]\succprec x[i+1]$.

We call the removal of an adjacency relationship {\em cutting} (or a {\em cut}). The restoration of an adjacency relationship is called {\em gluing}. Consider two neighboring samples $x[i]$ and $x[i+1]$. A cut between them will remove $x[i]$ as a predecessor of $x[i+1]$ and $x[i+1]$ is removed as a successor of $x[i]$ ($x[i]\succprec x[i+1] \Rightarrow x[i]\nsuccprec x[i+1]$). We glue the same samples together by creating these exact adjacency relationships when they have not previously existed $x[i]\nsuccprec x[i+1] \Rightarrow x[i]\succprec x[i+1]$.

We only consider two types of domains for surgery: linear and circular. In the case of linear domains, any cut between one linear domain will create two, now separate, linear domains. Any gluing of two separate linear domains creates one new linear domain. 

The second case is that of cutting a circular domain. Recall that a circular domain is a domain where each sample has one unique successor and predecessor. There are no boundary samples. If we cut a circular domain, we create a linear domain with the samples severed by the cut becoming boundary samples. If we glue boundary samples of a linear domain we create a circular domain. Hence, surgery gives us an operation to relate the two types of domains we are considering. \replaced{Cutting and glueing form inverse operations. If a domain is cut and the severed pieces are glued alone the pace where they were previously cut, that restores the original domain. Conversely, if two separate domains are glue together, and then are cut at the gluing location, the original two domains are restored.}{More generally a domain was cut and the same pieces are used to glue to restore the formerly severed adjacency relationship. Hence, gluing serves as an inverse of cutting (and vice versa)}\georg{Review Q: Minor 24}.

Our interest is the effect of surgery on extrema, and consequently on level\marginnote{Review Q: Minor 25}\added{ }set persistence of the finite sample sequence and associated structures such as barcodes and box snakes.
Consider a sequence $X:x[0,\ldots, N-1]$ of length $N$ and two disjoint subsequences $Y:y[0,\ldots,L-1]$, $Z:z[0,\ldots,M-1]$ of length $L$ and $M$ such that $N=L+M$. We define a glue operation 
\begin{align*}
Y\glue Z\rightarrow X: y[0,\ldots,L-1]&\succprec z[0,\ldots,M-1]\\
&\Rightarrow x[0,\ldots, L-1,L,\ldots,L+M-1]\\
&= x[0,\ldots, L-1,L,\ldots,N-1]
\end{align*}
where $\succprec$ adds the adjacency relation between the last sample of the first set and the first sample of the second. Observe that the index of the second set $z$ was adjusted to start at $L$ rather than $0$ in the gluing process.

The cut operation is thusly defined as
\begin{align*}
X\cut L-1: x[0,\ldots, L-1\cut L,\ldots, L+M-1]&=y[0,\ldots,L-1]\\
&\nsuccprec z[0,\ldots,M-1]
\end{align*}
where $\nsuccprec$ removes the adjacency relations between two adjacent samples in the sequence. Again we relabeled items in $x$ starting from $L$ to start their index from $0$. Intuitively these operations behave like concatenation of two arrays and splitting an array into two linearily ordered seperate subarrays that contain all elements of the original array.

The domain\marginnote{Review Q: Minor 26}\added{s} of linear and circular finite sequences have been studied extensively. The definition of cyclically ordered sets goes back to Huntington \cite{Huntington1916cyclic,huntington1924sets,huntington1935inter} and \v{C}ech \cite{Cech1969} and their surgery is developed by Novak \cite{novak1984cuts}. For a discussion of the case of surgery on partial to cyclic orders, see \cite{haar2016cyclic} and references therein.

\subsection{Changes to Extrema from Surgery}

\begin{figure}[th]
\begin{subfigure}{0.188\linewidth}
\includegraphics[width=\textwidth]{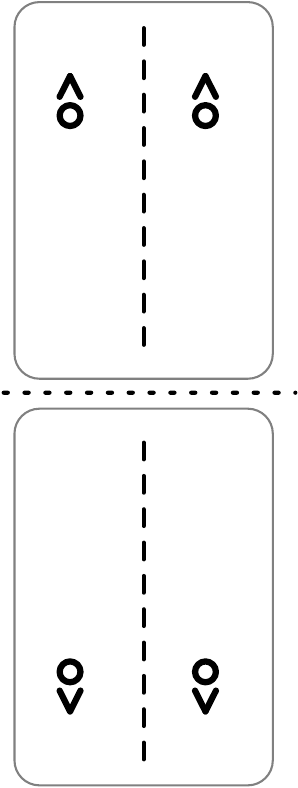}
\caption{Flat Extrema}\label{subfig:surgery-A}
\end{subfigure}
\begin{subfigure}{0.19\linewidth}
\includegraphics[width=\textwidth]{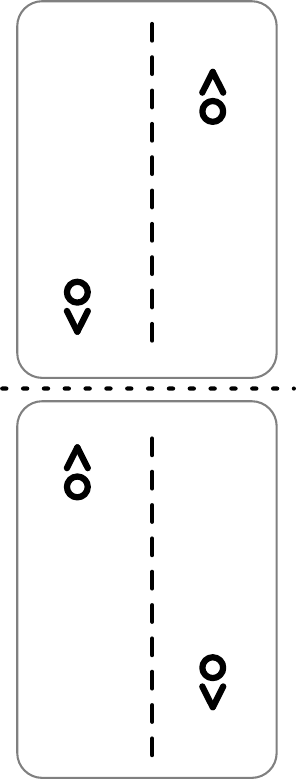}
\caption{Extrema}\label{subfig:surgery-B}
\end{subfigure}
\begin{subfigure}{0.19\linewidth}
\includegraphics[width=\textwidth]{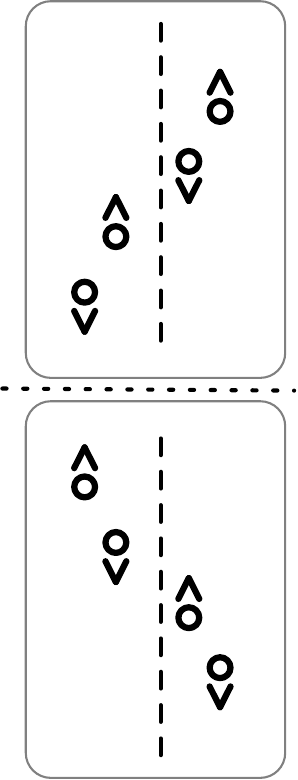}
\caption{Monotone}\label{subfig:surgery-C}
\end{subfigure}
\begin{subfigure}{0.373\linewidth}
\includegraphics[width=\textwidth]{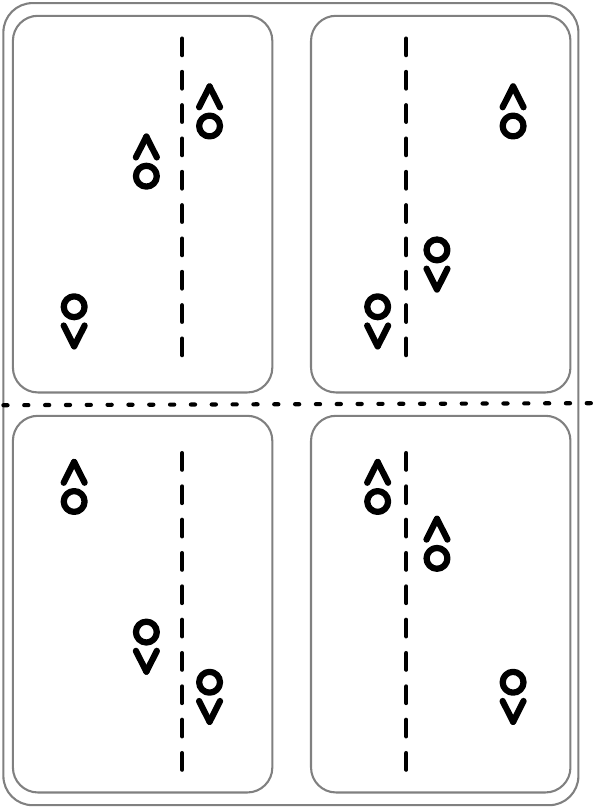}
\caption{Extrema and Monotone}\label{subfig:surgery-D}
\end{subfigure}

\caption{Surgery cases: \subref{subfig:surgery-A} Splitting a flat or merging at the same level, \subref{subfig:surgery-B} splitting between distinct extrema, or merging distinct extrema, \subref{subfig:surgery-C} splitting a monotone, or merging such that both extrema form a monotone, \subref{subfig:surgery-D} splitting between extremum and a monotone, or merging such that one extrema persists while the other is absorbed into a monotone. The up and down markers indicate directions of extrema. For cut operations, the directional indicators give the extrema after the cut. For gluing operations, the directional indicators give the configuration before gluing.}\label{fig:surgery}
\end{figure}

Consider a sequence $X$ of length $L$ denoted as $X:x[0,\ldots, L-1]$, and two disjoint subsequences $Y$,$Z$ of length $L$ and $M$ such that $N=L+M$. We denote these as $Y:y[0,\ldots,L-1]$ and $Z:z[0,\ldots,M-1]$. We have a surgery of gluing 
$$y[0,\ldots,L-1]\succprec z[0,\ldots,M-1]\Rightarrow x[0,\ldots\added{,}L-1,L\added{,}\ldots,L+M-1]$$

and the cut that reverses the gluing 

$$x[0,\ldots, L-1,L \ldots,L+M-1] \nsuccprec y[0,\ldots,L-1],z[0,\ldots,M-1].$$

A cut between samples necessarily creates boundary extrema. In principle, there are only four types of configurations: Cut within an extremal flat, cut between extrema that are not flat, cut in the interior of a monotone, and cut between an extremum and a monotone. These four cases are depicted in Figure \ref{fig:surgery}. The top and bottom row of the figure depict the dual. Given that a cut necessarily creates a boundary, the samples that will be at the boundary must be extrema. In the two left cases, these samples were already extrema before the cut, so there is no change. Only surgery involving monotones will create new boundary extrema. The two right cases in the figure show the nearest extremum after the monotone, and the newly created boundary extrema that follow the monotone.

We observe the following rule: A boundary extremum created on the {\em left} of a monotone surgery always points {\em in the direction of the monotone}, while a boundary extremum on the {\em right} of the said surgery always points {\em in the opposite direction of the monotone}. This rule allows us to determine the boundary extremum purely from the direction of the monotone. Thus, inspection of the nearest extremum is not necessary.

The table does not treat gluing and cutting with constant sequences separately. For example, if two constant functions are glued at the same level, this is captured by case (a) in Figure \ref{fig:surgery}. If two constant functions are glued at different levels, then we have case (b) and convert the constant function to the appropriate extremum based on how the two glued constant functions are ordered. Case (c) cannot occur, and case (d) occurs if a monotone function is glued with a constant one, and the same extrema assignment occurs as discussed for case (b).

\subsection{Changes to Barcodes from Surgery}
We discuss how cutting and gluing affects the barcode. All results in this section assume we are working with a finite discrete sequence of a totally ordered set $X:=x[0,\dots, N-1]$.

\subsubsection{Cut}

We can enumerate the cases of performing a cut and their homological effect from Figure \ref{fig:surgery}:

\begin{enumerate}
\item Cut between a minim\replaced{um}{a}\marginnote{Review Y: Minor 16} connected to a\added{n} \added{adjacent} maximum: The number of minima and maxima remain\added{s}\marginnote{Review Y: Minor 17} the same. Hence, $\#minima=\#maxima$.
\item Cut a monotone: A minimum and maximum are added at the boundary. We have that $\#minima=\#maxima$ and the number of bars increases by one, i.e., $\#bars+1$. 
\item Cut a minimum (or between a minimum an an adjacent monotone): The minimum splits into two minima at each boundary. This means that $\#minima=\#maxima+1$, and $\#bars+1$.
\item Cut a maximum (or between a maximum and an adjacent monotone): The maximum splits into two maxima. We find that $\#maxima = \#minima+1$.
\end{enumerate}

\begin{observation}[Increasing the Number of Bars]
The number of bars increases by one precisely when a minimum is created by the surgery.
\end{observation}

How a cut affects the barcode depends on the barcode construction principle used. For the local barcode construction rules discussed earlier, any cut will at most sever one bar. Surgery for barcodes using a non-local rule such as those following the elder rule is more complicated. A continuous version of this problem has been studied in \cite{cultrera2024dynamically}.

\subsubsection{Glue}

Gluing is the inverse operation of the cut in the previous section. Hence, instead of potential increases in extrema, we get corresponding decreases. As a consequence we get the following Lemma:

\begin{lemma}[Gluing when the Direction of Extrema Agree]
If a boundary is glued where extrema at the boundary are of the same type (both minima or both maxima), then the number of that extrema in the glued domain is reduced by one.
\end{lemma}

This is illustrated in Figure \ref{fig:surgery}. In case (a), the two extrema will merge into one. In case (d), one of the extrema is absorbed into a new monotone.

\begin{lemma}[Gluing when the Directions of Extrema Disagree]
If a boundary is glued where extrema at the boundary are of different types (one is a mimimum and the other is a maximum), then the total number of extrema in the glued domain either stays the same as the sum of glued subdomains (case (b) in Figure \ref{fig:surgery}) or decreases by one each (case (c) in Figure \ref{fig:surgery}).
\end{lemma}

We see from Figure \ref{fig:surgery} that in case (c) the boundary extrema are both absorbed into a monotone, and hence, are no longer extrema in the glued domain. As a result of these lemmas, we can know what kind of gluing configuration was encountered based on the changes in the number of extrema.

\begin{corollary}[Gluing Retaining the Number of Extrema]
If the boundary extrema remain extrema after gluing them together, then the number of extrema before gluing is the same as the number of extrema after gluing (case (b) in Figure \ref{fig:surgery})).
\end{corollary}

\begin{corollary}[Gluing Reducing the Number of Extrema]
If the number of extrema in the glued domain has one less minimum and maximum from the number of extrema before gluing, then the boundary was glued into a monotone (case (c) in Figure \ref{fig:surgery}).
\end{corollary}

Observe that only the case of a decrease in extrema by one is not uniquely defined. This can be due to identical extrema being glued (case (a) in Figure \ref{fig:surgery}) and no new monotone being formed. Or it can be due to a new monotone being formed next to an extremum (case (d) in Figure \ref{fig:surgery}). These are distinguished by whether the glue samples shared the same level or not.

\begin{remark}[Sufficient Conditions for Counting the Number of Bars.]
    Notice that, if you check the directions of glue points and check if what you are gluing resulted in a monotone, then we can determine how the number of bars changes.
\end{remark}

\begin{figure}[htb]
\rotatebox{90}{\footnotesize \qquad\quad Glued}\hspace{2pt}
  \begin{subfigure}{0.155\linewidth}
    \includegraphics[width=\linewidth]{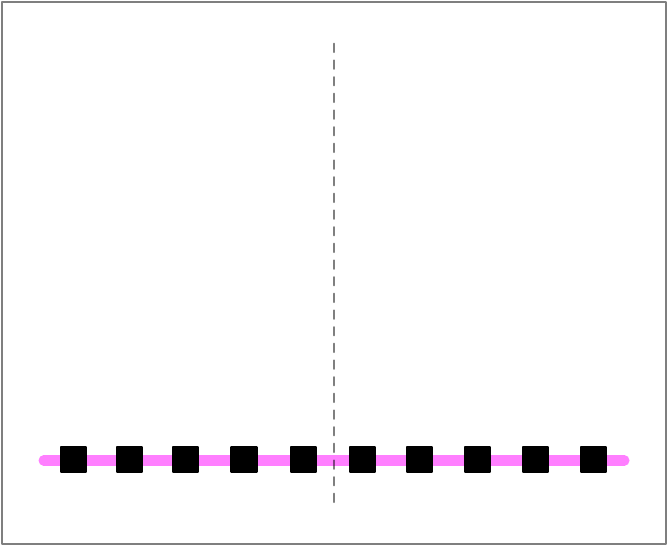}
    \caption{
    } \label{subfig:surgery:const-merged}
  \end{subfigure}
  \begin{subfigure}{0.155\linewidth}
    \includegraphics[width=\linewidth]{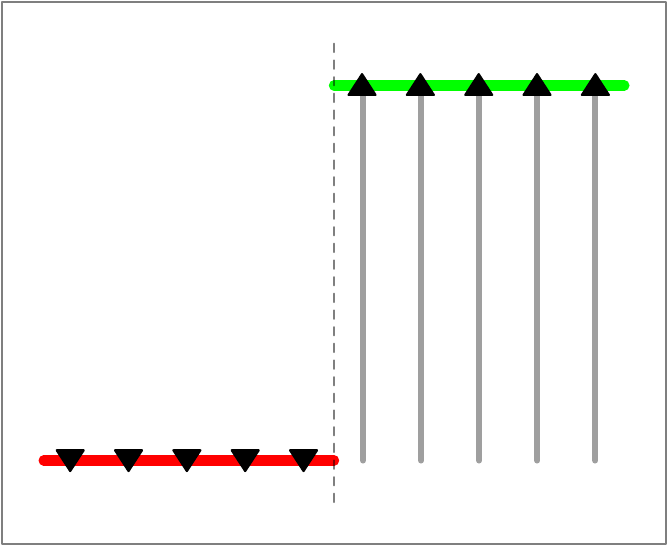}
    \caption{
    } \label{subfig:surgery:step-merged}
  \end{subfigure}
  \begin{subfigure}{0.155\linewidth}
    \includegraphics[width=\linewidth]{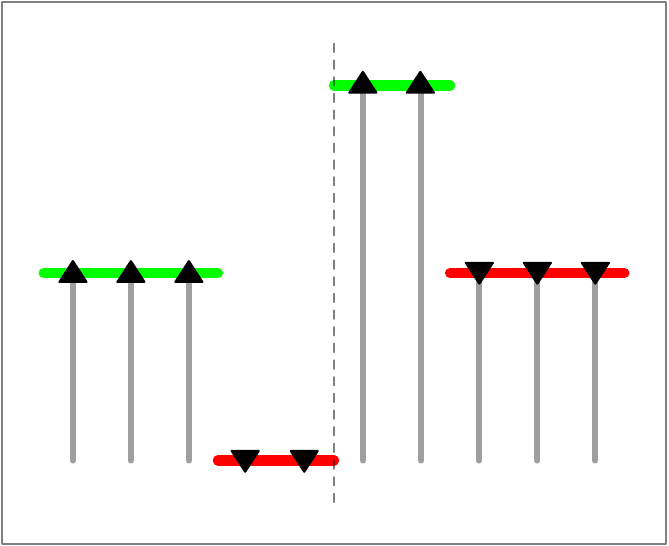}
    \caption{
    } \label{subfig:surgery:eediff-merged}
  \end{subfigure}
  \begin{subfigure}{0.155\linewidth}
    \includegraphics[width=\linewidth]{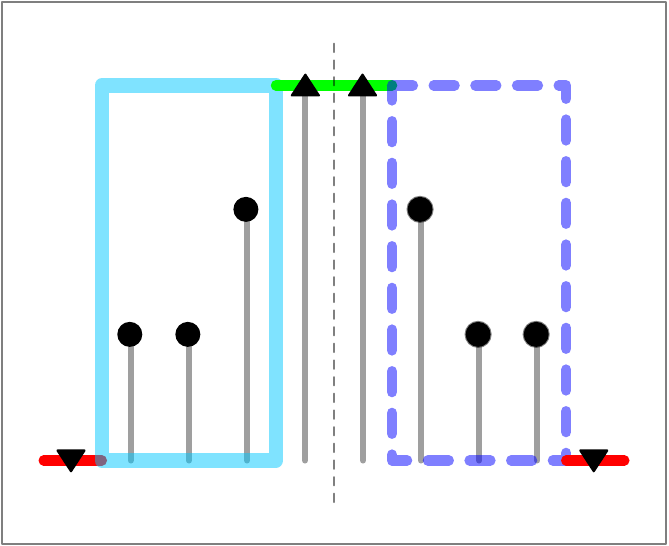}
    \caption{
    } \label{subfig:surgery:eesame-merged}
  \end{subfigure}
  \begin{subfigure}{0.155\linewidth}
    \includegraphics[width=\linewidth]{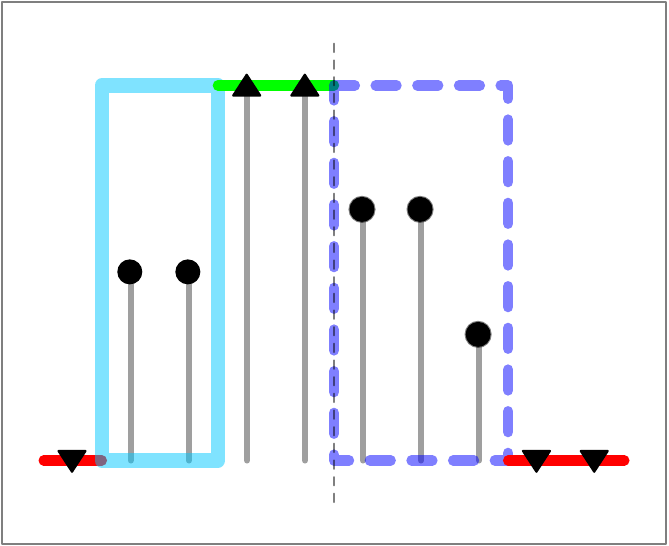}
    \caption{
    } \label{subfig:surgery:em-merged}
  \end{subfigure}
  \begin{subfigure}{0.155\linewidth}
    \includegraphics[width=\linewidth]{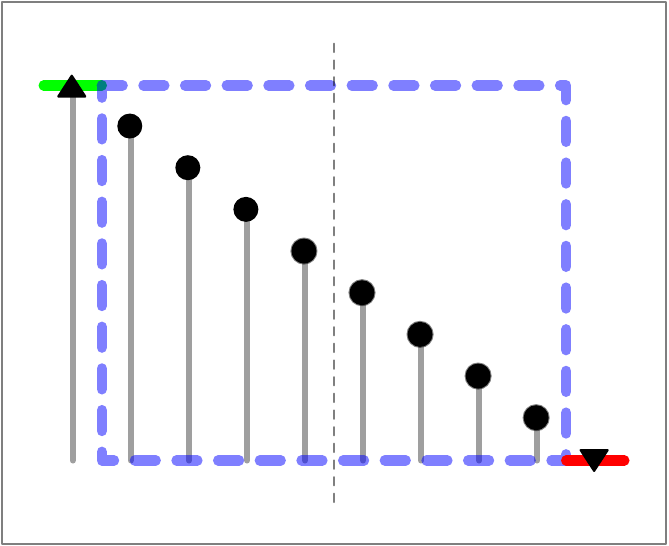}
    \caption{
    } \label{subfig:surgery:monotone-merged}
  \end{subfigure}
    \\
\rotatebox{90}{\footnotesize \qquad\quad Cut}\hspace{2pt}
  \begin{subfigure}{0.155\linewidth}
    \includegraphics[width=\linewidth]{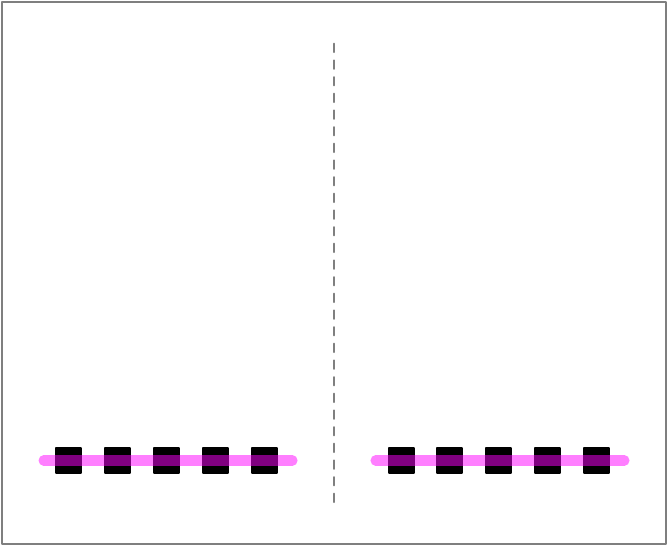} 
    \caption{
    } \label{subfig:surgery:const-split}
  \end{subfigure}
  \begin{subfigure}{0.155\linewidth}
    \includegraphics[width=\linewidth]{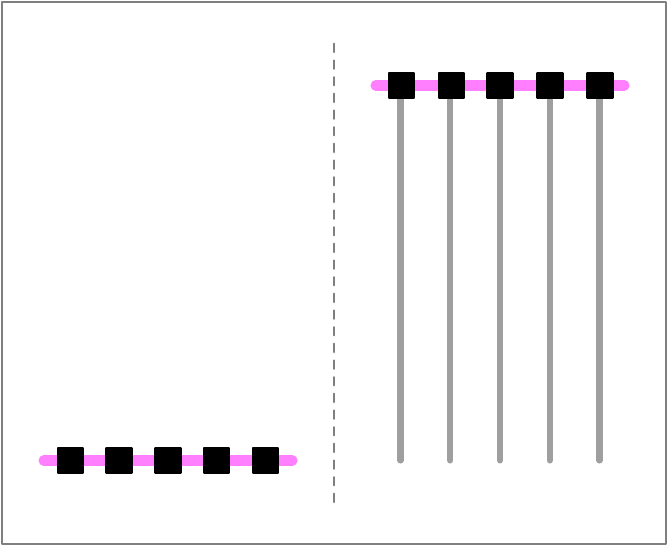}
    \caption{
    } \label{subfig:surgery:step-split}
  \end{subfigure}
  \begin{subfigure}{0.155\linewidth}
    \includegraphics[width=\linewidth]{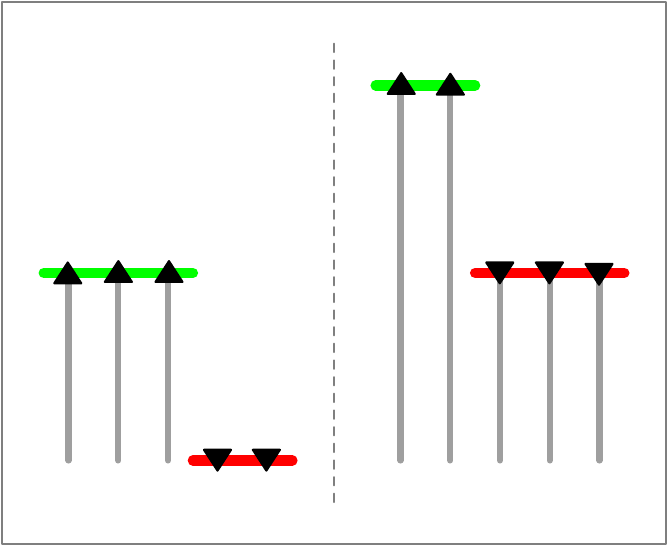}
    \caption{
    } \label{subfig:surgery:eediff-split}
  \end{subfigure}
  \begin{subfigure}{0.155\linewidth}
    \includegraphics[width=\linewidth]{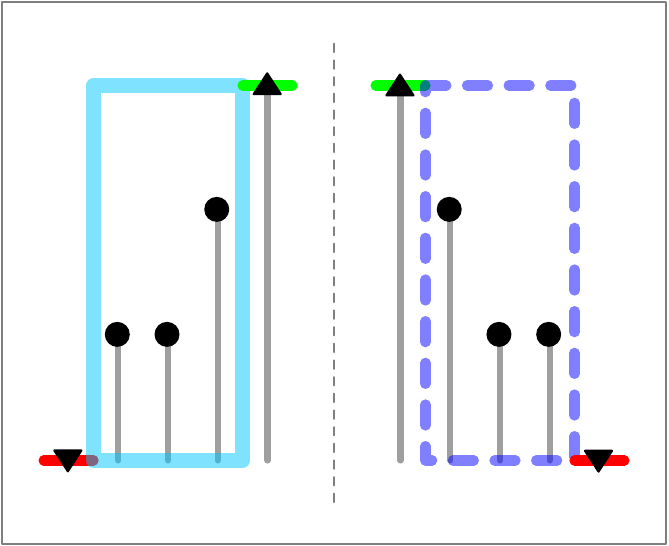}
    \caption{
    } \label{subfig:surgery:eesame-split}
  \end{subfigure}
  \begin{subfigure}{0.155\linewidth}
    \includegraphics[width=\linewidth]{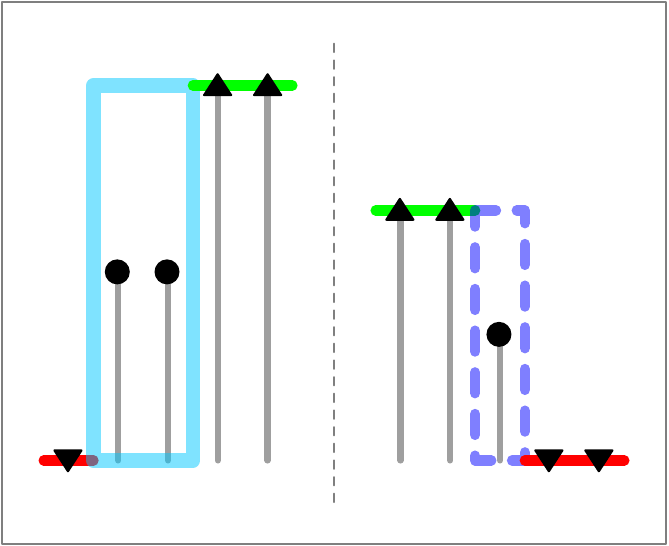}
    \caption{
    } \label{subfig:surgery:em-split}
  \end{subfigure}
  \begin{subfigure}{0.155\linewidth}
    \includegraphics[width=\linewidth]{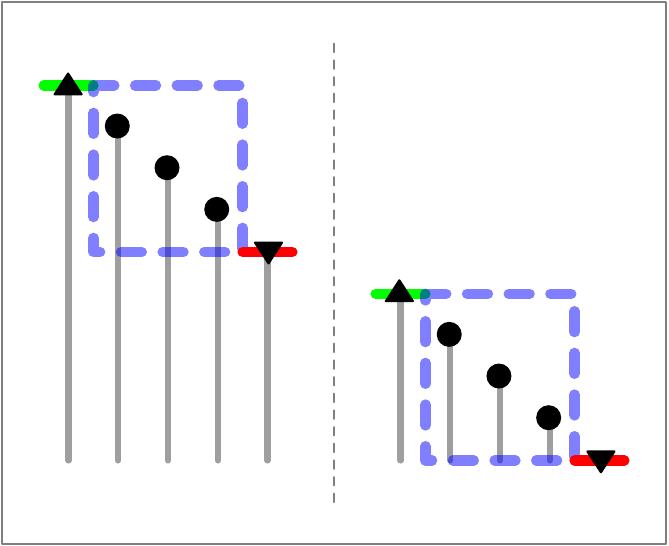}
    \caption{
    } \label{subfig:surgery:monotone-split}
  \end{subfigure}
  \caption{Classification of surgery with changes of box snakes. Top row \subref{subfig:surgery:const-merged}-\subref{subfig:surgery:monotone-merged} shows glued cases, and bottom row \subref{subfig:surgery:const-split}-\subref{subfig:surgery:monotone-split} shows cut cases. Colors indicate the type of box snake. Purple and $\blacksquare$ is a constant (both minimum and maximum), red and $\blacktriangledown$ is a minimum, green and $\blacktriangle$ is a maximum, cyan solid is an ascending monotone, blue dashed is a descending monotone. The dashed vertical line indicates the surgery position.}\label{fig:boxsnakesurgery}
\end{figure}

\subsection{Box Snake Surgery}

The changes of extrema due to surgery directly translate into surgery rules for box snakes. All pertinent cases are shown in Figure \ref{fig:boxsnakesurgery}. The top row shows the glued condition, whereas the bottom row shows the cut condition.
\robin{There is an inconsistency between box snake and snake box. Is one preferred over the other?}\georg{I use box snake for the whole structure, and snake box for a single one. See \ref{sec:snakeboxes}. I guess what is confusing here is that some surgeries are inside one snake box perhaps creating more and some involve multiple snake boxes, perhaps justifying the use of box snake.}

\deleted{We describe all cases as cuts, though they can be inverted going in the direction of gluing if the levels line up as given by the cut. The cases are (from left to right):}

\marginnote{Review Z: Major 1,2}
\addedbegin
We describe all cases as cuts with regard to the modifications on box snakes (from left to right): 
\begin{itemize}
\item A constant set \subref{subfig:surgery:const-merged} $\cut$ two constants \subref{subfig:surgery:const-split}. Given that before the cut all data are captured in one box snake, a second box snake needs to be created to allow for two separated box snakes after the cut. However, the content of the box snakes remain unchanged since they are capturing a constant set.
\item A step level set \subref{subfig:surgery:step-merged} $\cut$ two constants \subref{subfig:surgery:step-split}. Given that the step is already segmented into two separate snake boxes, we merely need to separate them but no generation of a new snake box is necessary. However, the content of the box snakes change, as they are now two constants. We can identify if a snake box must be a constant by noting if it is the only snake box in a connected box snake sequence. This criterion also applies to the next case if one side happens to be an extremum in the boundary of the box snake sequence.
\item Two differing extrema \subref{subfig:surgery:eediff-merged} $\cut$ two separated extrema that have opposite orientation (minima $\leftrightarrow$ maximum) \subref{subfig:surgery:eediff-split}.  This case is easy to handle because one merely needs to separate between already existing box snakes. Per the remark above, one needs to follow by a check if one side has been reduced to only one snake box, in which case it is a constant. This case is not explicitly shown in Figure \ref{fig:boxsnakesurgery}.
\item Within a flat extremum \subref{subfig:surgery:eesame-merged} $\cut$ two separated extrema that have the same orientation (minima $\leftrightarrow$ minima or maxima $\leftrightarrow$ maxima) \subref{subfig:surgery:eesame-split}. Given that the flat extremum is only captured by one snake box, we need to generate a second one to allow the separation.
\item Between an extremum and a monotone \subref{subfig:surgery:em-merged} $\cut$  two separate extrema of the same orientation and a shortened or monotone (the monotone can disappear if the original monotone contained only one sample) \subref{subfig:surgery:em-split}. In this case, we are cutting between existing snake boxes. However, after the cut, the samples that were separated must be boundary extrema. Yet, before the cut, only one was an extremum. Therefore, all flat samples next to the cut in the monotone need to be captured in a new snake box that is an extremum. The orientation of the extremum is dictated by the direction of the extremum on the other side, or equivalently by the direction of the monotone on the side of the cut. The monotone that was shrunk by this creation of the extremum either disappears, if it no longer contains any samples, or is adjusted to match the new boundary extremum (observe how the monotone in \subref{subfig:surgery:em-merged} is larger and contains a wider range of levels than the reduced monotone \subref{subfig:surgery:em-split}, which is now arrowed by the width of the new extremum and also has its levels reduced.
\item Within a monotone \subref{subfig:surgery:monotone-merged} $\cut$ two extrema of opposite orientation with shortened monotones (vanishing on either side if the size is zero) \subref{subfig:surgery:monotone-split}. In a sense, this is similar to the previous case, except that samples in the boundaries need to be captured by new extrema on both sides, instead of just one. Monotones are reduced and level adjusted if capturing any samples. Otherwise, the associated snake box is removed.
\end{itemize}

All cases of gluing are reversed in the cut steps. As we have seen, snake boxes were created, but can also disappear in cases where monotones are completely converted to extrema. Hence, we should expect to undo and merge box snakes, remove extrema, but on occasion create snake boxes for monotones. The type of merge that must be performed can be inspected from the bottom row of Figure \ref{fig:boxsnakesurgery}. Each case is defined by the relative type and position of boundary extrema as we have observed in Figure \ref{fig:surgery}, including constant level sets. Specifically, all cases of gluing in the same order as the cases of cutting discussed above are:
\begin{itemize}
\item Two constant sets at the same level \subref{subfig:surgery:const-split}  $\glue$ one constant \subref{subfig:surgery:const-merged}. One snake box is removed and the single snake box now spans the merged data, which remains a constant.
\item Two constant sets at different levels \subref{subfig:surgery:step-split} $\glue$ one step level set \subref{subfig:surgery:step-merged}. The first snake box is connected to the second snake box. Depending on the direction of the level difference, they are converted into opposing extrema, where the higher level is the maximum.
\item Two differing extrema where their relative direction matches their extrema type (e.g. the maximum in the boundary is higher than the minimum in the boundary) \subref{subfig:surgery:eediff-split}  $\glue$ two extrema that have opposite orientation. One only needs to connect the previously disconnected snake boxes.
\item To equally oriented maxima at the same level \subref{subfig:surgery:eesame-split} $\glue$ one merged maximum where both boundary extrema are combined into one flat extremum of the given type.
\item Two extrema of the same type but at different levels \subref{subfig:surgery:em-split}\ $\glue$ the one of the two extrema will become part of the adjacent monotone. Which one depends on the type of extremum. For maxima, it is the one that is at a lower level. For minima, it is the one that is on a higher level. The width and level of the monotone snake box is adjusted to agree with its new adjacent extrema. \robin{I am having a hard time parsing this one, especially the phrase, "extremum less extended in the direction of the extremum". }\georg{Reworded. How's this version?}\robin{good!}
\item Two differing extrema where their relative direction disagrees with their extrema type. That is, they form a monotone. \subref{subfig:surgery:monotone-split} $\glue$ a monotone \subref{subfig:surgery:monotone-merged} that is possibly merged with existing monotones adjacent to the previous boundary snake boxes if they exist. The width is the result of the total width of all merged snake boxes, and the levels are matched with the extrema before and after the monotone.
\end{itemize}
\addedend

\subsection{Shifts as \added{an} Application of Box Snake Surgery}\marginnote{Review Q: Minor 27}

\added{Box snake structures can be utilized for persistent homology-preserving deformations without any surgery as was demonstrated in an application for audio data deformation on a fixed audio time series.\cite{essl2024deform}.}

\replaced{This can be too restrictive for some applications, where time series data arrives in real-time such as in the case of live sensor data.}{In applications, time series data can arrive in real-time from sensor data.} A typical way of handling such live updating data with finite memory is known as \emph{shifts}. A shift consists of keeping a constant length block of time series data, where old information is discarded on one side, and new information is appended on the other side. Shifts play an important role in streaming applications where new information arrives and a processing window is updated. 

Windowed processing is widely used in digital signal processing \cite{harris1978use}. \replaced{This allows for}{Hence, allowing} topological processing that is compatible with existing procedures \added{that} facilitate\deleted{s} the combination of existing techniques with box snake-based persistent homology computations and manipulations. \marginnote{Review Q: Major 3} \deleted{A non-streaming application that uses snake boxes to allow persistent homology-invariant deformations of audio data \cite{essl2024deform}} \added{We have elsewhere demonstrated that persistent homology-preserving deformations of streaming audio data can be realized} \deleted{on streaming data} via this surgical process \cite{essl2025shiftaudio}.


\subsubsection{$M$-Sample Shifts}\label{sec:mshift}
\robin{I'm assuming $\ldots$ should be used in the subsequence notation here in order to maintain consistency?}\georg{Yes, good catch! Fixed (hopefully)}
Consider a sequence of length $N+M$ and two overlapping subsets of length $N$ such that the first is the subsequence $[0,\ldots,N-1]$ and the second is the subsequence $[M,\ldots,N+M-1]$. This can be viewed as an $M$-sample shift from a sequence starting at $0$ to a sequence starting at $M$.

An $M$-shift can be realized by sequences of surgeries. The $M$-shift on a linear domain constitutes one cut, to remove the data that has been ``shifted out", and one glue of new data of the same size that is being ``shifted in".

A left-$M$-shift performs the following cut-glue sequence: $X\cut M\Rightarrow X_L, X_R$ where $X_L$ is of length $M$ and $X_R$ is of length $N-M$. This is followed by $X_R \glue Y$ where $Y$ is a new sequence of length $M$. An example of a left-$3$-shift is displayed in Figure \ref{fig:left3shift}.
A right-$M$-shift is similarly $X\cut M\rightarrow X_L, X_R$ where $X_L$ is of length $N-M$ and $X_R$ is of length $M$. This is followed by $Y\glue X_L$ with $Y$ again of length $M$.

\begin{figure}[ht]
\includegraphics[width=\textwidth]{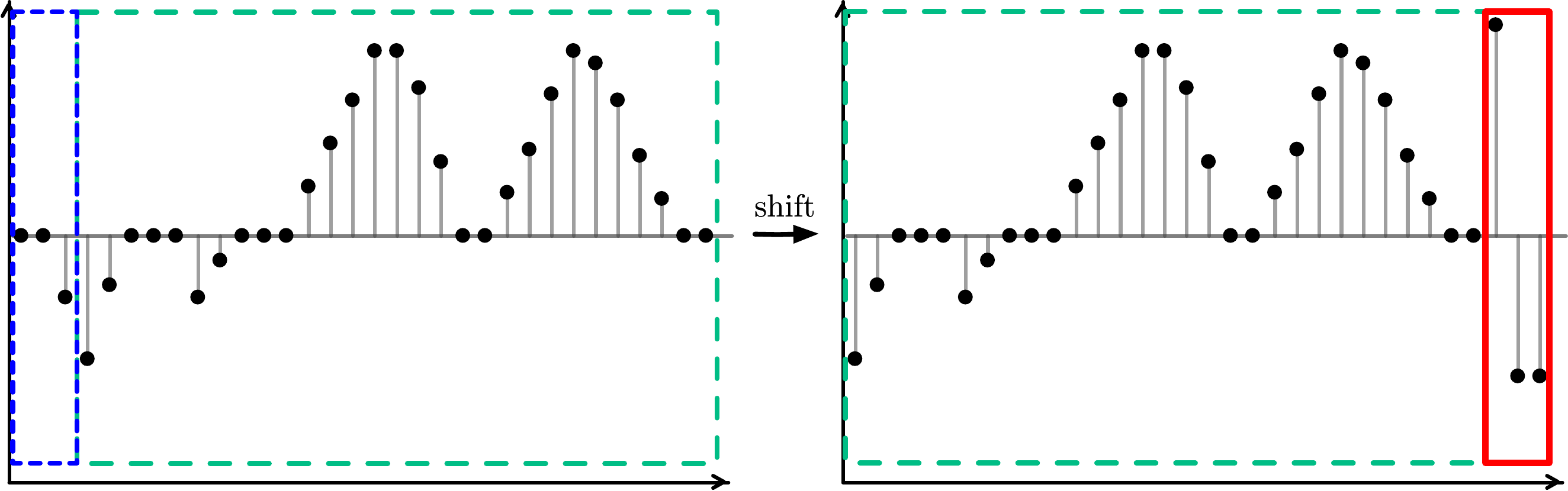}
\caption{Example of a left-$3$-shift. First the original sequence is cut on the left and the left side (blue narrow dashed) is discarded. Then the remaining piece (green wide dashed) is glued on the right with a new sequence of three samples (red solid).}\label{fig:left3shift}
\end{figure}

Any shift can be converted into a circular shift by first {\em decircularizing} (or {\em linearizing}) the circular domain with a cut a\added{t}\marginnote{Review Q: Minor 28} the point of the circular domain where the shift is to occur, and {\em recircularizing} the domain by gluing the two boundaries together after the above shift sequence is completed.

\section{Computational Cost and Real-time Performance}\marginnote{Review Y: Major 7}

\added{Overall the performance of algorithms is in line with what has been reported in the literature \cite{baryshnikov2024time,ost2024banana}. All algorithms on samples can be performed in the order $O(N)$ of the number of samples $N$. The worst-case size of the box-snake structure appears when a box snake is required to capture every individual sample, such as by alternating extrema, or by extrema interspersed with one-sample monotones. Therefore, the worst-case box-snake size is also $O(N)$. The best-case scenario is the constant function, which can be captured with one box snake independent of the size of $N$. To indicate that the box snake size can differ from the number of samples, we will denote its size by $B$.}

\added{A merge tree can be computed from a box snake in linear time on the order of the size of the box snake $B$ using a stack data structure. The number of nodes of a merge tree is precisely the number of interior maxima together with the number of minima; hence, it is at worst the total number of extrema $E$. This can be smaller than the box snake structure which also captures monotones between extrema. Both the elder and the local bar code construction rule can be implemented by a one-pass tree traversal of the merge tree. Hence, we get the following hierarchy of sizes of the data structures: $N\geq B\geq E$. We see that $N=B=E$ precisely when every sample is an extremum and both samples in the boundary are minima.}

\marginnote{Review Z: Major 3}\added{It is well established that persistent homology on functions can be computed in order of the number of samples $N$ \cite{baryshnikov2024time}. For example, Baryshnikov's algorithm \cite{baryshnikov2024time} can be extended to handle flat sample regions without changing this computational cost, by simply detecting and marking flat regions \cite{essl2023combinatorial}. In this implementation scenario, the algorithm resembles an incremental peak finder such as the {\tt find\_peaks} function of the scipy Python library. Our work establishes that flat detecting peak finders can be used without problem in the computation of sublevel set persistence of 1-dimensional data. However, box snake structures can help isolate flat extrema and improve performance of repeated follow-up processing, if the extrema in the data are sufficiently sparse.}

%% file: conclusion.tex
\section{Conclusion}\label{sec:conclusion}

We have jointly developed sublevel and superlevel set persistent homology on sequential data viewed as ordered finite sets over ordered levels. This point of view is close to practical computation in that it captures the finiteness of number representations and the discretization of many time series or other sequential data encountered in applications, while differing in approach to models that assume continua such as the real line and data drawn from $\mathbb{R}$. We show that concepts like non-isolated extrema and extrema that fall onto the same level can be handled in this context leading to a justification of ad hoc decisions that have already appeared elsewhere. This means that digital sequential data can be analyzed without any need to make a priori functional assumptions or potential perturbations of \deleted{the} data. We describe a duality theorem between sublevel and superlevel set persistence that shows \deleted{that} in this setting, \replaced{they}{these two settings} are related and a range of duality results can be derived. We discuss the impact of circular versus linearly ordered domains, and give a surgery theory that allows us to modify output and auxiliary structures such as barcodes and box snakes for streaming applications. All our results are up to order. Hence, if data are deformed in an order-preserving fashion, our results remain valid unaltered. Numerical implementation on current day computational hardware utilizes finite number representations that preserve order. Thus, our results apply to customary numerical implementations.

%% file: appendix.tex
\section{Relationship of codomain \texorpdfstring{$\mathbb{R}$}{R} and linearly ordered Finite Sets}\label{app:quantsamp}

Kulisch \cite{kulisch2013computer} summarizes the key advantage of considering order structure when relating $\mathbb{R}$ and digital representations such as floating point numbers as follows:

\blockquote{\textit{``It is well known that floating-point numbers and floating-point arithmetic do not obey the rules of the real numbers $\mathbb{R}$. However, rounding is a monotone function. So the changes to the order structure are minimal. This is the reason why the order structure plays a key role for an axiomatic approach to computer arithmetic."}}

In our setting, we only need the order structure, hence giving justification to the approach taken in this paper of avoiding the use of $\mathbb{R}$ or bounded intervals thereof.

\subsection{Quantization}
Given the dominance of modeling using $\mathbb{R}$, it is helpful to understand both the process of going from an interval of $\mathbb{R}$ to a finite linearly ordered set, and how to reverse the process. The first process is known as {\em quantization} in the signal processing literature (see for example \cite[p. 43ff]{steiglitz1997digital}). The latter can be understood as a process of embedding a finite set into $\mathbb{R}$ while preserving the order.

For our purpose, we only require a weaker condition than is typically required for quantization. Quantization retains the presumption of a shared metric structure between the real value and the quantized discrete representation. However, we only require the preservation of total order. Then the assumption of a metric structure simply labels the finite ordered set. It is fruitful to think of these questions categorically.

Let $\mathbb{R}(<)$ be the real line with its order structure and let $L_N(<)$ be a finite set with an order structure. Let $\mathbb{R}_N(<)$ be a finite full subcategory of $\mathbb{R}(<)$ with $N$ objects. We can define a forgetful functor $\mathbb{R}_N(<)\Rightarrow L_N(<)$, where we only keep the order structure and forget the real numbers associated with each point. By construction, this functor is order-preserving. 

\begin{example}
    {\em Uniform quantization}: Take an interval $[a,b]\in\mathbb{R}$ with $b>a$ and $N$ quantization steps greater than $0$. Additionally, assume the standard Euclidean metric on $\mathbb{R}$. The {\em uniform quantization step} is computed from the distance $q=d(a,b)/N$. Then $\{a,a+q,\ldots, a+(N-1)*q\}$ is a finite discrete subset preserving the order on $\mathbb{R}$.
\end{example}

\begin{example}
    {\em Non-uniform quantization}: For an interval $[a,b]\in\mathbb{R}$ with $b>a$, take $N$ positive numbers $q_0,\ldots q_{n-1}\in\mathbb{R}$ such that $\sum_{i=0}^{n-1}q_i=d(a,b)$. Then $\{a,a+q_0,\ldots,q_{n-1}\}$ is a finite discrete subset that preserves order. Uniform sampling is the special case where all $q_i$ are the same.
\end{example}

\subsection{Machine Number Representation and Order-Preservation}

Much of practical software is implemented on the standardized number types provided by the hardware. This essentially means there are two types of numbers, direct interpretation of binary representation as integers or fixed-point models, or by use of hardware supported floating-point computation. It is today safe to assume that this will obey the IEEE Standard 754 and its direct revisions. Order of integers in computation is straight-forward, but due to the underlying binary representation, the same holds true for the IEEE-754 family floating point numbers \cite{kahan1996ieee,IEEE2019float}. For applications, this means that implementation of the results of this paper are correctly ordered when using ordered comparisons ($<,>,\leq,\geq$) on floating point numbers and any other order-preserving concrete number representation in computation. Otherwise standard numerical problems of floating point numerical computations \cite{goldberg1991every} are avoided.

\subsection{Geometric Realization or Embedding}

We may want to be able to go in the inverse direction of quantization and go from a finite ordered set to an embedding thereof in $\mathbb{R}(<)$. We have already seen the existence of injective maps from a finite ordered set into an ordered finite subset $\mathbb{R}_N(<)$.

The process can be understood as assigning each member in the finite ordered set a real number such that order is preserved:
$L_N(<)\rightarrowtail\mathbb{R}(<)$.

\begin{example}
    Let $\beta_0,\ldots,\beta_{n-1}$ be a finite set of numbers in $\mathbb{R}$ such that $\beta_0<\ldots<\beta_{n-1}$. Then we have a fully faithful functor: $L_N(<)\Rightarrow\mathbb{R}(<)$ that preserves the order of $L_N(<)$ in $\mathbb{R}(<)$. Observe that there is no restriction on the numbers on $\mathbb{R}$ except for the order.
\end{example}

\begin{figure}[th]
\centering
    \begin{subfigure}[t]{.325\textwidth}
    \centering
    \includegraphics[width=0.91\textwidth]{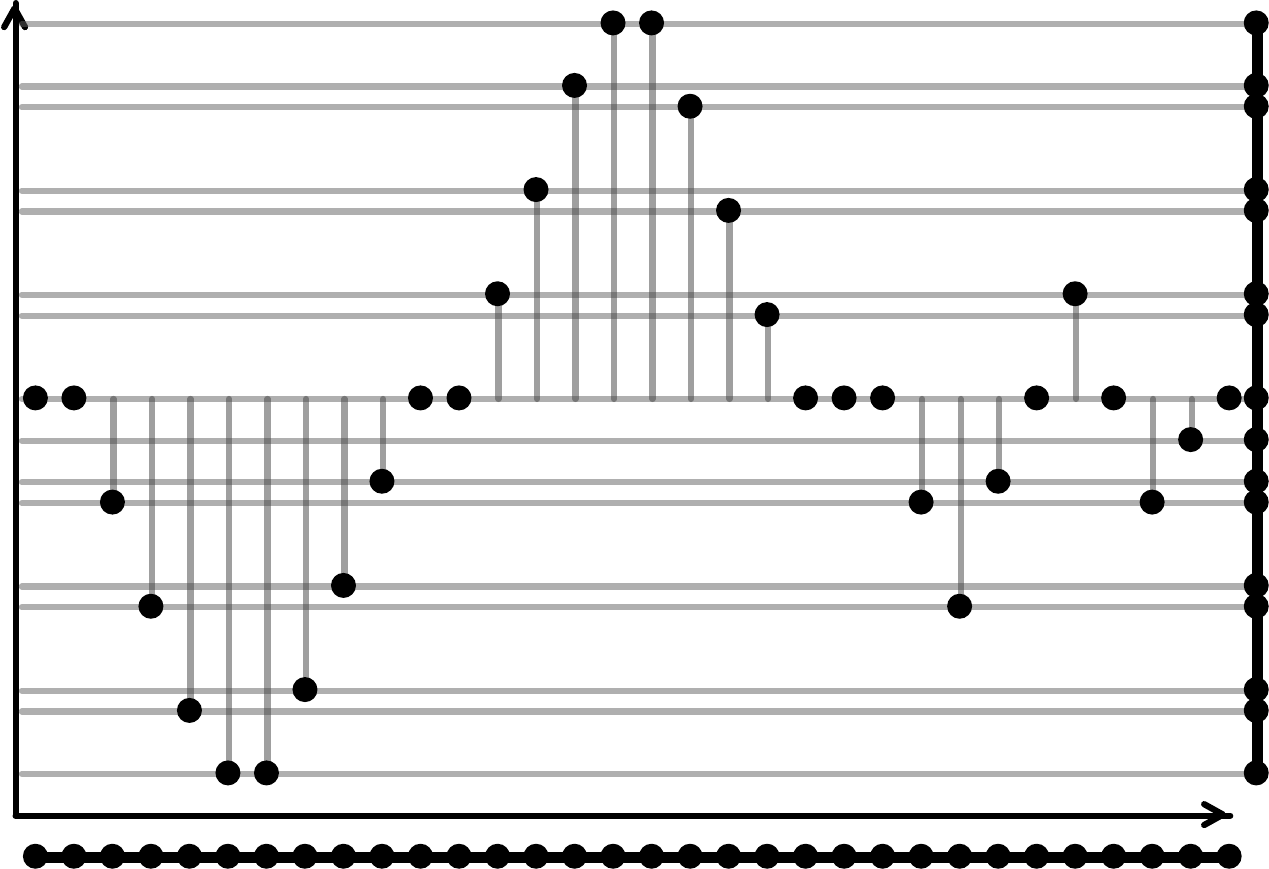}        \caption{}\label{subf:uniformA}
    \end{subfigure}
    \begin{subfigure}[t]{.325\textwidth}
    \centering
    \includegraphics[width=0.91\textwidth]{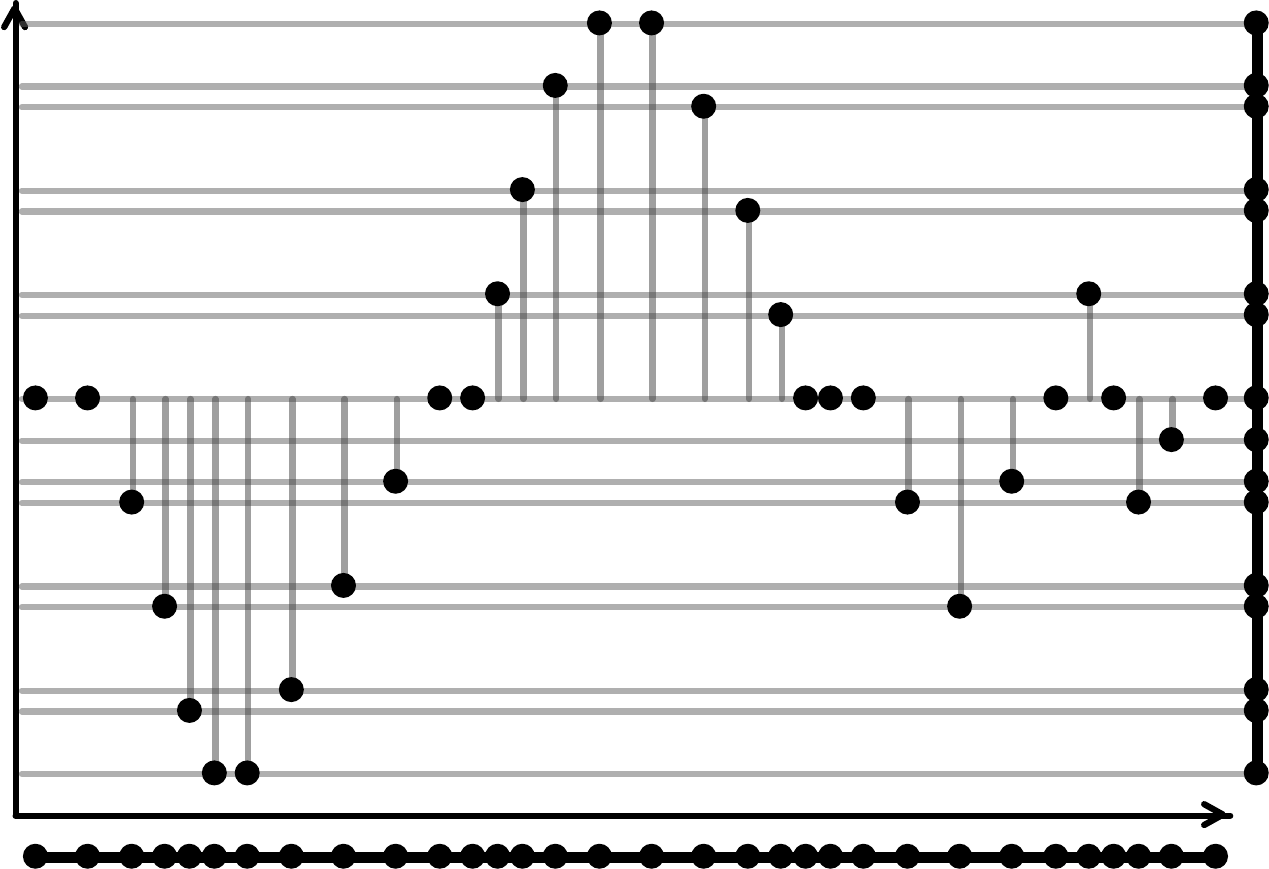}        \caption{}\label{subf:uniformB}
    \end{subfigure}
    \begin{subfigure}[t]{.325\textwidth}
    \centering
    \includegraphics[width=0.91\textwidth]{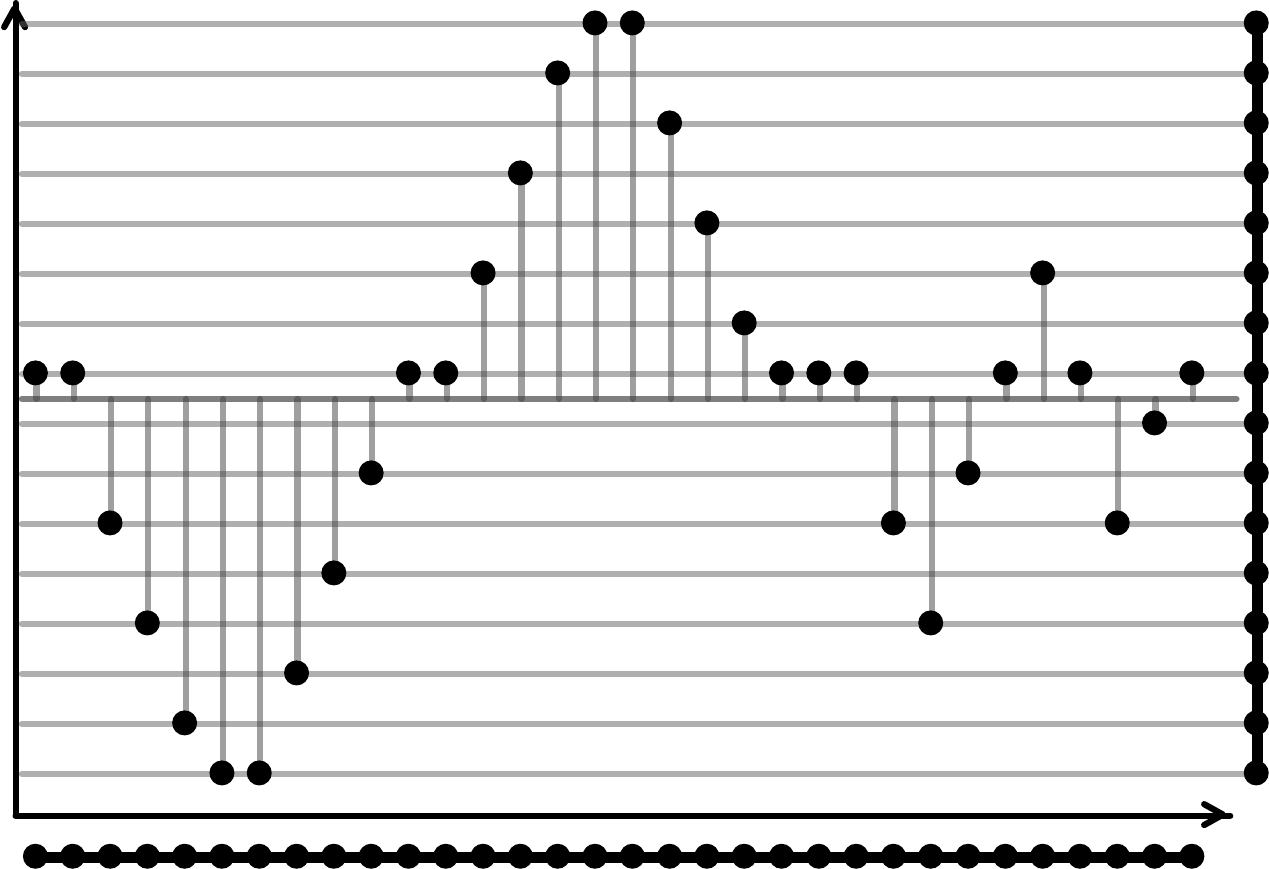}        \caption{}\label{subf:uniformC}
    \end{subfigure}
\caption{Geometric realization of the ordered set as a function graph: The x-axis should be understood as a linear (or cyclic) order. The y-axis also consists of a linearly ordered set that indexes level sets. Neither axis should be assumed to have any geometric information. However, it is convenient to relate the real axis to an ordered set. \subref{subf:uniformA} Uniform sampling follows from a uniform order-preserving embedding of the ordered samples along the x-axis. \subref{subf:uniformB} Non-uniform sampling is a consequence of an order-preserving embedding that is not uniform. \subref{subf:uniformC} Same order levels but with different (uniform) level embedding.}\label{fig:samplinganddiscretization}   
\end{figure}

Figure \ref{fig:samplinganddiscretization} provides a visualization two geometrizations of the same data set. Throughout this paper we  generally depict samples equally spaced on the paper (see Figure \ref{fig:samplinganddiscretization} (left)). However this is for familiarity only, the results equally apply for non-uniformly sampled data as depicted in Figure \ref{fig:samplinganddiscretization} (right). The sample levels can be taken from levels of functions, hence may exhibit non-uniformity, as depicted. However, it is noted that the levels are also only fixed up to order.

\section{Simplicial Complexes from Adjacency Relationships on Finite Sets}\label{app:connected}

In our setting of finite discrete functions, all connected components are subsets of $\mathbb{I}=\{0,1,\dots, n-1\}$. If we consider the standard topology on $\mathbb{R}$ and endow $\mathbb{I}$ with the subspace topology, then $\mathbb{I}$ is a totally disconnected set. This means only the singleton sets are connected and form the connected components. This is not very enlightening for understanding our data. However, we can easily change our perspective of connectedness of subsets of $\mathbb{I}$ to get something non-trivial. Namely, we can map subsets $K \subset \mathbb{I}$ to a graph $G=(V,E)$ where the connectedness of the graph matches our \marginnote{Review Y: Major 5}\replaced{notion}{definition} of connectedness \added{of adjancency in a discrete ordered subset}\footnote{\addedbegin{}Terms connectedness and connectivity have different technical meaning in different fields. Our use leans on the common use in topology \cite{munkres2000topology} and should not be confused with the use of the terms in graph theory \cite{harary2018graph}.\addedend}.

More specifically, let $|K|$ denote the cardinality of a set, $\mathcal{P}(\mathbb{I})$ denote the power set of $\mathbb{I}$, and $\mathcal{G}$ as the set of subgraphs of the path graph on $n$ vertices. Recall the path graph on $n$ vertices denoted as $P_n$ has $n$ vertices that can be listed in order $v_0, v_1, \dots, v_{n-1}$, with edges $\{v_i, v_{i+1}\}$ for $i=0,1,\dots, n-2$. 

\added{The purpose of the construction is to relate connected set members to edges, and pairs of adjacent set members to edges.}\marginnote{Review Y: Minor 18 and 19}
We define the \emph{graph realization} to be the map $f: \mathcal{P}(\mathbb{I}) \to \mathcal{G}$ where $K \mapsto G$ such that each element $k\in K$ maps to a vertex and $j,k \in K$ maps to an edge if and only if $j=k+1$ or $j+1=k$. Observe this map is bijective. Every subgraph of $P_n$ gets mapped to by the subset of $\mathbb{I}$ that has the same number of points as vertices in the subgraph. Additionally, these points are spaced appropriately to get the correct edges. Furthermore, this map is injective since distinct subsets of $\mathbb{I}$ map to different subgraphs.

Furthermore, by definition of this map, we have the following observation. 

\begin{proposition}[Preservation of Connectedness]
\label{prop:connectedness}
    Let $f:\mathcal{P}(\mathbb{I})\to \mathcal{P}$ be the graph realization map. Then $K\subset \mathbb{I}$ is a connected component of $\mathbb{I}$ if and only if $f(K)$ is the (connected) path graph with $|K|$ vertices.
\end{proposition}

Hence, Proposition~\ref{prop:connectedness} implies there is a one-to-one correspondence between connected components of subsets of $\mathbb{I}=\{0,1,\dots, n-1\}$ and connected components of subgraphs of $P_n$. Let $T\subset \mathbb{N}$. Recall a \emph{filtration} of a set $X$ is a nested family of subsets $(X_i)_{i\in T}$ starting at the empty set, such that for all $i, j \in \mathbb{N}$ where $i\leq j$, we have $X_i \subset X_j$, and $\bigcup_{i \in T} X_i = X$. Using the graph realization map, we have that filtrations on $\mathbb{I}$ induce filtrations on $P_n$ and vice-versa. 

\begin{proposition}[Preservation of Filtrations]
\label{prop:filtrations}
    Let $f:\mathcal{P}(\mathbb{I})\to \mathcal{L}$ be the graph realization map. Then $\emptyset=X_0 \subset X_1 \subset \dots \subset X_N=\mathbb{I}$ is a filtration of $\mathbb{I}$ if and only if $\emptyset \subset f(X_0)\subset f(X_1)\subset \dots \subset f(X_N)=P_n$ is a filtration on $P_n$, the path graph on $n$ vertices. 
\end{proposition}

\begin{proof}
    First we assume $\emptyset=X_0 \subset X_1 \subset \dots \subset X_N=\mathbb{I}$ is a filtration and show $\emptyset \subset f(X_0)\subset f(X_1)\subset \dots \subset f(X_N)=P_n$ is a filtration on $P_n$. By definition of $f$, $f(\emptyset)=\emptyset$ and $f(\mathbb{I})=P_n$. Additionally, if $X_i \subset X_j$, then $X_j = X_i \sqcup K$ where $K$ is a subset of points of $\mathbb{I}$. Hence, $f(X_j)$ can be partitioned into two subgraphs, $f(X_i)$ and $f(K)$, showing that $f(X_i)$ is a subgraph of $f(X_j)$. This shows that $\emptyset \subset f(X_0)\subset f(X_1)\subset \dots \subset f(X_N)=P_n$ is a filtration on $P_n$.

    To show the other direction, we can apply a symmetric argument and utilize the fact that $f$ is bijective.
\end{proof}

We see from Propositions~\ref{prop:connectedness} and \ref{prop:filtrations}, we can use the standard language of sublevel set persistent homology on simplicial complexes to analyze the sublevel set persistent homology of finite discrete sequences. 

\section{Software Realization}

The ideas discussed in this paper have been realized as an executable software in JavaScript and the source code is available at \url{https://github.com/gessl/DiscreteLevelSetPersistence/tree/FoDS}. Most of the figures in this paper are rendered with this implementation. Technical fine details such as detailed adjustments of snake boxes under surgery have been omitted from the discussion in the body of the paper for length but can be found in the source code. All algorithms associated with computing barcodes via various methods, computing and surgery of box snake structures, the merge tree algorithm are all functionally implemented and can be found in {\tt LevelSetPersistence.js}. The file {\tt discretegraph.js} provides the graph rendering used for the figures and the interactive demonstration. Furthermore, {\tt Interactions.js} realizes the interaction as well as audio playback discussed elsewhere \cite{essl2024deform}. JavaScript is not a popular language for academic dissemination of code. However, given the application domain of interest of one of the authors and the ability to create interactive demonstration, it was found to be preferable over other alternatives such as Python. A full running demonstration of discrete level set persistence using this code can be found at \url{https://gessl.github.io/DiscreteLevelSetPersistenceDemo/}.